\documentclass[3p]{elsarticle}

\usepackage{amsfonts,amsmath,amsthm,amssymb,mathtools}
\usepackage{mathrsfs}
\usepackage{bm}
\usepackage{bbm}

\usepackage{mathbbol}
\usepackage{cases}
\usepackage{dsfont}

\usepackage{tikz}
\usetikzlibrary{automata,arrows,positioning,calc}
\DeclareGraphicsExtensions{.png,.jpeg,.jpg,.pdf}
\graphicspath{{./Figures/}}
\usepackage{enumitem}
\usepackage{multicol}
\usepackage{booktabs}

\usepackage{xcolor}
\usepackage{graphicx}
\usepackage{hyperref}
\usepackage[nameinlink,capitalize]{cleveref}

\usepackage{todonotes}
\usepackage{amsopn}
\usepackage{comment}

\allowdisplaybreaks[4]

\newtheorem{theorem}{Theorem}[section]
\newtheorem{definition}[theorem]{Definition}

\newtheorem{corollary}[theorem]{Corollary}
\newtheorem{lemma}[theorem]{Lemma}
\newtheorem{assumption}[theorem]{Assumption}

\newtheorem{remark}[theorem]{Remark}

\crefname{assumption}{Asm.}{Assumptions}
\crefname{definition}{Def.}{Definitions}

\DeclarePairedDelimiter{\set}{\{}{\}}
\DeclarePairedDelimiter{\abs}{\lvert}{\rvert}
\DeclarePairedDelimiter{\norm}{\lVert}{\rVert}

\DeclareMathOperator{\diverg}{Div}
\DeclareMathOperator{\sign}{sign}
\DeclareMathOperator{\supp}{Supp}

\DeclareMathOperator{\loc}{loc}
\DeclareMathOperator{\dist}{dist}
\DeclareMathOperator*{\esssup}{ess-sup}

\newcommand{\R}{\mathbb{R}}

\newcommand{\NN}{\mathbb{N}}
\newcommand{\N}{\mathbb{N}}

\newcommand{\bx}{\bm{x}}
\newcommand{\bv}{\bm{v}}

\newcommand{\by}{\bm{y}}
\newcommand{\bgamma}{\boldsymbol{\gamma}}
\newcommand{\bq}{\bm{q}}

\newcommand{\bbF}{\bm{F}}

\newcommand{\bJ}{\bm{J}}
\newcommand{\bV}{\bm{V}}
\newcommand{\bPhi}{\boldsymbol{\Phi}}

\newcommand{\bPsi}{\boldsymbol{\Psi}}
\newcommand{\bz}{\boldsymbol{z}}
\newcommand{\bF}{\boldsymbol{F}}

\newcommand{\mBV}{\mathcal{BV}}
\newcommand{\mTV}{\mathcal{TV}}
\newcommand{\mL}{\mathsf{L}}
\newcommand{\mC}{\mathcal{C}}
\newcommand{\mK}{\mathcal{K}}
\newcommand{\mX}{\mathcal{X}}
\newcommand{\mY}{\mathcal{Y}}
\newcommand{\mI}{\mathcal{I}}

\newcommand{\mZ}{\mathcal{Z}}
\newcommand{\mU}{\mathcal{U}}
\newcommand{\mW}{\mathsf{W}}
\newcommand{\mQ}{\mathsf{Q}}
\newcommand{\mF}{\mathsf{F}}
\newcommand{\mbW}{\boldsymbol{\mathsf{W}}}

\newcommand{\Def}{\coloneqq}
\newcommand{\e}{\mathrm{e}}

\newcommand{\dd}{\;\mathrm{d}}
\newcommand{\D}{\partial}

\newcommand{\tup}[1]{\textup{(}#1\textup{)}}

\newcommand{\nnorm}[1]{{\left\vert\kern-0.25ex\left\vert\kern-0.25ex\left\vert #1 \right\vert\kern-0.25ex\right\vert\kern-0.25ex\right\vert}}

\begin{document}
\begin{frontmatter}
\title{Existence and uniqueness of nonlocal nonlinear conservation laws via fixed-point methods}

\author[inst1]{Xiaoqian Gong$^{1,}$}
\author[inst2]{Alexander Keimer$^{1,}$}
\author[inst3]{Lorenzo Liverani$^{1,}$}
\author[inst4]{Hossein Nick Zinat Matin$^{1,}$}

\affiliation[inst1]{organization={Amherst College, Department of Mathematics},
            city={Amherst, MA 01002},
            country={USA}}

\affiliation[inst2]{organization={Institute of Mathematics, University of Rostock},
            city={Rostock, 18051},
            country={Germany}}

\affiliation[inst3]{organization={FAU Erlangen-Nürnberg, Department of Mathematics, 91058},
            city={Erlangen},
            country={Germany}}

\affiliation[inst4]{organization={Laboratory of Computer Science (LIX), Ecole Polytechnique},
            city={Inria-Saclay, Alan Turing Building, Palaiseau, 91120},
            country={France}}

\fntext[fn1]{Authors are listed in alphabetic order.}


\begin{abstract}
We investigate the well-posedness of scalar conservation laws whose flux depends on the solution both pointwise and nonlocally through integral averages. Our analysis is based on a fixed-point formulation, in which the nonlocal dependence is incorporated as a space- and time-dependent component of the flux, together with classical stability estimates for entropy solutions. This framework unifies and extends several models previously considered in the literature and applies, in particular, to conservation laws with memory effects (nonlocality in time) or delay. We prove the existence and uniqueness of weak entropy solutions on a sufficiently short time horizon and show that under additional assumptions, existence and uniqueness can be obtained on any finite time horizon. In addition, we present numerical simulations to illustrate the qualitative effects of memory on the solution dynamics.

\end{abstract}
\end{frontmatter}
\vspace{-0.4cm} 

\section{Introduction}
\label{sec:related-work}

The mathematical theory of conservation laws with nonlocal interactions has developed rapidly over the past two decades. The distinguishing feature of these models is that the (possibly space- and time-dependent) flux depends both on the pointwise value of the solution and on averaged or delayed information about it, a situation that arises naturally in vehicular traffic flow, crowd dynamics, sedimentation processes, and material transport systems.

The classical formulation of a nonlocal conservation law reads
\begin{equation*}
    \partial_t q =-\diverg\Big(\bF\big(t, \bx,\mW[J(q),\gamma](t,\bx), q\big)\Big), \quad (t, \bx) \in (0, T) \times \R^n,
\end{equation*}
where $\diverg$ represents the total  divergence of the flux function $\bF$ with respect to the spatial variable $\bx$. The nonlocality is modeled by the term $\mW[J(q),\gamma](t,\bx)$, which represents a convolution-type integral of a nonlinear transformation $J(q)$ of the solution against a kernel $\gamma$ that may depend on space, time, or both. Generally, one is interested in studying the Cauchy problem associated with the conservation law, given a suitable initial datum $q_0$.

\subsection{Local conservation laws}
The theory of entropy solutions for scalar conservation laws was established in the seminal work of Kru\v{z}kov \citep{kruvzkov1970first}, which provides existence, uniqueness, and stability proofs for multidimensional scalar conservation laws with sufficiently regular fluxes. Comprehensive accounts of the theory and its analytical techniques can be found in the monographs of Bressan \citep{Bressan2000} and Holden and Risebro \citep{holden2015front}.

Among the many applications of scalar conservation laws, traffic models have been a particularly active area. The Lighthill--Whitham--Richards (LWR) model \citep{lighthill1955kinematic,lighthill1955kinematic2,richards1956shock} is the best known of these; it describes traffic flow using a scalar conservation law for vehicle density. The LWR model has since inspired numerous extensions incorporating additional physical effects, including higher-order variants such as the Aw--Rascle--Zhang model \citep{aw2000resurrection,zhang2000structural,zhang2002non} and models accounting for diffusion or relaxation mechanisms \citep{burger2003diffusively,nelson2002traveling}; broader reviews of macroscopic traffic models can be found in \citep{bellomo2002mathematical,bellomo2011modeling,piccoli2009vehicular,rosini2013macroscopic}.

\subsection{Nonlocal conservation laws} 
Nonlocal conservation laws ``extend'' the classical framework by allowing the flux to depend on spatial averages or convolution operators applied to the solution, capturing situations in which the dynamics are governed by long-distance interactions rather than purely local information. Examples include traffic flow models in which drivers respond to downstream density averages. One of the early analyses in this direction appeared in \citep{betancourt2011nonlocal}, which studied sedimentation models involving nonlocal fluxes. In the context of road traffic, Blandin and Goatin \citep{blandin2016well} analyzed well-posedness for a conservation law with nonlocal velocity and established the existence and uniqueness of entropy solutions; similar models with anisotropic kernels and traffic interactions were subsequently studied in \citep{chiarello2018global,goatin2016well,keimer2017existence}.

A general framework for nonlocal systems with several spatial dimensions was developed by Aggarwal, Colombo, and Goatin \citep{aggarwal2015nonlocal}, whose analysis uses finite-volume approximations to construct solutions and prove convergence to entropy solutions. Related existence and numerical results for nonlocal scalar conservation laws were obtained by Amorim, Colombo, and Teixeira \citep{amorim2015numerical}. Chiarello, Goatin, and Rossi \citep{chiarello2019stability} derived stability estimates and analyzed the continuous dependence of solutions with respect to the interaction kernel; existence results for models with discontinuous flux functions have also been obtained \citep{chiarello2023existence}.

Traffic models in which vehicles respond to both upstream and downstream conditions are also well studied. The qualitative behavior of such models, including shock formation and wave propagation, has been investigated in \citep{hu2021shock}, and related models incorporating nonlocal interactions and delay effects have been considered in \citep{ciaramaglia2024nonlocaltrafficflowmodels} by Ciaramaglia, Goatin, and Puppo, in \citep{contreras2025wellposedness} by Contreras, Goatin and Villada, in \cite{Keimer2019} by Keimer and Pflug, and in \cite{kloeden2016nonlocal} by Kloeden and Lorenz.

Beyond traffic flow, nonlocal conservation laws appear in crowd dynamics and pedestrian motion, for which Colombo, Garavello, and Lécureux-Mercier \citep{colombo2012class} introduced a class of nonlocal models and analyzed their well-posedness; multidimensional extensions were further investigated in \citep{bruno2011non}. Applications also arise in industrial processes, including material transport on conveyor belts \citep{goatin2025nonlocalmodelheterogeneousmaterial,Rossi2020} and clustered dynamics of interacting particle systems \citep{colombo2025non}. Borsche, Colombo, and Garavello \cite{borsche2015differential} introduced a general multidimensional conservation law model in which the nonlinear flux depends in an operator sense on the solutions. This includes, in particular, nonlocal (in space) conservation law models whose nonlocal kernel is sufficiently smooth.

\subsection{Existence and stability under weak regularity}
A central theme in the recent literature is the weakening of regularity assumptions on the interaction kernels or the flux functions, as many applications involve kernels that are only bounded or of bounded variation. Significant progress in this direction has been achieved by Coclite et al., who established the existence and uniqueness of weak solutions for scalar nonlocal conservation laws with kernels of bounded variation \citep{coclite2022existence} and Colombo, Crippa and Spinolo, who extended this to the multi-d case in \cite{colombo2026multi}.

On the stability side, Lécureux-Mercier \citep{Mercier2011improved} derived refined estimates that quantify the dependence of entropy solutions on perturbations of the flux, an idea that plays a central role in the present work, in which the nonlocal operator is treated as a space- and time-dependent coefficient in the flux and well-posedness is established within a fixed-point framework. Recent work has also considered conservation laws with rough or discontinuous fluxes. Aggarwal and Vaidya \citep{aggarwal2025convergence} analyzed the well-posedness and convergence of numerical approximations for nonlocal conservation laws with rough fluxes. Error estimates for related multilane traffic models were further developed in \citep{aggarwal2025error}. 

\subsection{Particle approximations and numerical methods} Another important line of research investigates nonlocal conservation laws using particle or Lagrangian approximations. Deterministic particle schemes approximating nonlocal transport equations were analyzed by Di Francesco, Fagioli, and Radici \citep{DiFrancesco2019deterministic}, who provided constructive approximations of entropy solutions that are particularly useful for singular kernels; Radici and Stra \citep{radici2023entropy} extended these techniques to mildly singular kernels and obtained results on convergence to entropy solutions. From a numerical perspective, Friedrich, Kolb, and Göttlich \citep{friedrich2018godunov} studied Godunov-type schemes adapted to nonlocal LWR traffic models, and high-order methods were proposed by Chalons, Goatin, and Villada \citep{chalons2018high}.

More broadly, numerical methods for conservation laws employ well-established techniques such as finite-volume methods \citep{leveque2002finite}, discontinuous Galerkin methods \citep{cockburn2012discontinuous,cockburn2001runge}, and high-order ENO/WENO schemes \citep{harten1997uniformly,shu1988efficient,shu1989efficient}. These approaches have been adapted to handle the additional complexity introduced by nonlocal operators.

\section{The contribution of this work}

In this work, we study a general class of scalar conservation laws whose flux depends on the solution both locally and nonlocally through integral operators in space and time. The goal of the paper is twofold.

The first objective is to derive generalized stability estimates for local conservation laws with space- and time-dependent fluxes. Specifically, we consider models of the form
\begin{equation}
\begin{aligned}    
\partial_t \rho + \diverg \big(\bPhi(t, \bx, \rho) \big) &= 0, &&(t, \bx) \in (0, T) \times \R^n,\\
\rho(0, \cdot) &\equiv \rho_{\circ},&& \text{ on }\bx\in\R^n.
\end{aligned}
    \label{E:model}
\end{equation}
The stability of \eqref{E:model} with respect to its components, such as the flux, is a classical topic in the analysis of conservation laws, with the first estimates dating back to the seminal work of Kru\v{z}kov \cite{kruvzkov1970first}. In this paper, we adapt and generalize the results of Lécureux-Mercier \cite{Mercier2011improved} to our setting via a mollification procedure, which allows us to use the estimates of \cite{Mercier2011improved} and recover their analogues for \eqref{E:model} in the relaxation limit, while requiring less regularity on the flux function. This constitutes our first main result, \cref{theo:existence_uniqueness_less_regular_flux}.

The second goal is to establish well-posedness for a general class of nonlinear nonlocal conservation laws, covering nonlocality both in space and in time. Depending on the structure of the nonlocal interaction, we distinguish two models:
\begin{description}
    \item[Nonlocal in space:] 
    \begin{equation}
    \begin{aligned}    
        \partial_t q &=-\diverg\Big(\bF\big(t, \bx,\mW[J(q),\gamma](t,\bx), q\big)\Big), &&(t, \bx) \in (0, T) \times \R^n,\\
        \mW[J(q),\gamma](t,\bx)&=\int_{\R^n}\gamma(\bx-\by)J(q(t,\by))\dd\by,&& (t,\bx)\in (0,T)\times\R^n,\\
        q(0, \cdot) &\equiv q_{\circ},&& \bx\in\R^n.
    \end{aligned}
    \label{E:model_nonlocal}
    \end{equation}
\item[Nonlocal in space, memory in time:]
    \begin{equation}
    \begin{aligned}    
        \partial_t q &=-\diverg\Big(\bF\big(t, \bx,\mW[J(q),\kappa](t,\bx), q\big)\Big), &&(t, \bx) \in (0, T) \times \R^n,\\
        \mW[J(q),\kappa](t,\bx)&=\int_{\R_{<0}}\int_{\R^n}\kappa(t-s,\bx-\by)J(q(s,\by))\dd\by\dd s,&& (t,\bx)\in (0,T)\times\R^n,\\
        q(t, \cdot) &= q_{\circ}(t,\bx),&& (t,\bx)\in\R_{\leq0}\times\R^n.
    \end{aligned}
    \label{E:model_nonlocal_memory}
    \end{equation}
\end{description}
Here, $\bF:(0,T)\times\R^n\times\R\times\R\rightarrow\R^n$ is the flux; $\gamma:\R^n\rightarrow\R_{\geq0}$ and $\kappa:(0,T)\times\R^n\rightarrow\R_{\geq0}$ are the spatial and space--time memory kernels, respectively; $J:\R\rightarrow\R$ is the nonlinear nonlocal transformation; $q_{\circ}:\R^n\rightarrow\R$ is the initial datum; and $q_{\circ}:\R_{<0}\times\R^n\rightarrow\R$ is the historical datum. We refer to $\mW$ as the nonlocal operator.

The well-posedness analysis presented for both models is based on a fixed-point formulation, in which the nonlocal dependence is incorporated into a space- and time-dependent flux to reduce the nonlocal problem to a family of local ones. Using the stability estimates developed in the first part of the paper, together with classical compactness arguments, we establish the existence and uniqueness of entropy solutions under Lipschitz regularity assumptions on the flux. The results are stated in \cref{theo:existence_uniqueness_nonlocal_small_time_horizon} for the spatially nonlocal model \eqref{E:model_nonlocal} and in \cref{theo:existence_uniqueness_nonlocal_memory_small_time_horizon} for the model with memory \eqref{E:model_nonlocal_memory}.

Notably, the suggested fixed-point approach when dealing with nonlocal conservation laws has only been considered for two different model classes: When the flux is affine linear in the local quantity (see for instance \cite{Coron2010analysis,keimer2017existence,crippa2012existence}) where the method of characteristics and a specifically tailored fixed-point problem guarantees the existence and uniqueness of weak solutions, or when the kernel is smooth \cite{Mercier2011improved,colombo2012class,colombo2011control}.

The proposed framework is sufficiently general that it can be applied directly to several nonlocal models studied in the literature. These include \cite{amorim2015numerical,betancourt2011nonlocal,chiarello2019stability,ciaramaglia2024nonlocaltrafficflowmodels,DiFrancesco2019deterministic,hu2021shock,radici2023entropy} for the scalar case and \cite{aggarwal2015nonlocal,beckers2026monotone,colombo2025non,colombo2012class,colombo2011control,courcel2025well-posedness,goatin2025nonlocalmodelheterogeneousmaterial,Rossi2020} for related multidimensional models. In some cases, our results improve upon those in the existing literature by requiring less regularity on the input data or on the flux function, at least on a sufficiently short time horizon.

\medskip
\noindent
\textbf{Structure of the paper.}
The rest of the paper is organized as follows. In \cref{sec:notation}, we introduce the functional setting and notation used throughout. In \cref{sec:stability}, we establish existence, uniqueness, and stability estimates for the local conservation law \eqref{E:model} under minimal regularity assumptions on the flux, culminating in \cref{theo:existence_uniqueness_less_regular_flux} and the more specialized \cref{cor:specific_tailored}. \cref{sec:nonlocal_conservation_laws} contains the core well-posedness analysis: for each of the two nonlocal models \eqref{E:model_nonlocal} and \eqref{E:model_nonlocal_memory}, we construct a fixed-point mapping, show that it is a self-mapping and a contraction on a suitable admissible set, and prove the existence and uniqueness of entropy solutions. Extensions to systems and delay equations are also discussed. \cref{sec:numerics} presents numerical simulations illustrating the qualitative effects of memory on the solution dynamics, using a Lax--Friedrichs scheme within the fixed-point iteration. Finally, \cref{s:future} collects some open problems and future perspectives.

\section{Functional setting and notation}
\label{sec:notation}
\subsection{Functional spaces}
We will frequently use the classical Lebesgue and Sobolev spaces. Let $\Omega \subset \R^n$ 
be a domain of sufficient regularity. We denote by $\mL^p(\Omega)$ the Lebesgue spaces 
and by $\mathcal{W}^{k,p}(\Omega)$ the Sobolev spaces of functions in $\mL^p(\Omega)$ whose weak 
derivatives up to order $k$ belong to $\mL^p(\Omega)$, with their standard norms 
$\|\cdot\|_{\mL^p(\Omega)}$ and $\|\cdot\|_{\mathcal{W}^{k,p}(\Omega)}$. The local versions are 
denoted with the subscript ${\loc}$. More generally, given a Banach space $X$, we denote 
by $\mL^p(\Omega; X)$ the Bochner space of strongly measurable functions $f:\Omega\to X$ 
with finite norm
\[
\|f\|_{\mL^p(\Omega;X)} \coloneqq 
\begin{cases} 
\displaystyle\left(\int_\Omega \|f(x)\|_X^p \,\mathrm{d}x\right)^{1/p}, & p \in [1,\infty), \\[6pt] 
\displaystyle\esssup_{x\in\Omega}\|f(x)\|_X, & p = \infty. 
\end{cases}
\]
The choices $X = \mL^q(\Omega')$, $X = \R^k$, and $X = \R^{k\times l}$ recover, 
respectively, the mixed-norm spaces $\mL^p(\Omega;\mL^q(\Omega'))$ and the standard 
vector- and matrix-valued Lebesgue spaces. We also recall the following definition.

\begin{definition}[Total variation seminorm]
Let $\Omega \subset \R^n$ be open. The total variation seminorm of a function
$f \in \mL^1(\Omega)$ is defined as
\[
|f|_{\mTV(\Omega)}
\coloneqq
\sup_{\substack{\boldsymbol{\phi}\in \mC^{1}_{\mathrm{c}}(\Omega;\R^n)\\
\|\boldsymbol{\phi}\|_{L^{\infty}(\Omega)} \le 1}}
\int_{\Omega}\diverg(\boldsymbol{\phi}(\mathbf{x}))\,f(\mathbf{x})\,\mathrm{d}\mathbf{x}.
\]
\end{definition}

In light of this definition, we introduce the following classical functional spaces.

\begin{definition}[Space of bounded variation]
Let $\Omega\subset\R^n$ be an open set. We define the seminormed space
\[
\mTV(\Omega)
\coloneqq
\left\{ f \in \mL^1(\Omega) : |f|_{\mTV(\Omega)} < \infty \right\}.
\]
Furthermore, we define the Banach space
\[
\mBV(\Omega)
\coloneqq
\left\{ f \in \mL^1(\Omega) :
\|f\|_{\mBV(\Omega)}
=
\|f\|_{\mL^{1}(\Omega)} + |f|_{\mTV(\Omega)} < \infty
\right\}.
\]
\end{definition}

\subsection{Notation}

We introduce for \(n\in\N_{\geq 1}\) the cylindrical domain
\[
\Omega_T \Def (0,T)\times\R^n.
\]
We also represent by $W_n$ the value of the Wallis integral
\[
W_n \coloneqq \int_0^{\frac{\pi}{2}} \cos(\theta)^n d\theta.
\]
We recall that this can also be computed as 
\[
W_n = \tfrac{\omega_n}{2\omega_{n-1}},
\]
where $\omega_n$ is the volume of the unit ball $B(0,1)$ in $\R^n$, which is set to $1$ for consistency when $n=0$. 
Finally, for \(a\in\R\) we use the notation $\R_{\geq a}$ and $R_{\leq a}$ to denote the unbounded intervals $[a,+\infty)$ and $(-\infty,a]$, respectively.

For a function $(t, \bx, u) \mapsto \bPhi(t, \bx, u)$, $\nabla \bPhi$ and $\diverg \bPhi$ denote the gradient and local divergence with respect to the spatial variable---that is, $\bx$---and $\diverg(\bPhi) = \diverg \bPhi + \partial_u \bPhi \circ \nabla u$ represents the total divergence of the function $\bPhi$.

\section{Stability estimates for the local conservation law}
\label{sec:stability}

This section analyzes and generalizes stability estimates for the local model \eqref{E:model}.
We work within the following minimal assumptions on the flux $\Phi$ and initial datum $\rho_0$.

\begin{assumption}[assumptions on flux and initial datum]
\label{Ass:General}
Let $T\in\R_{>0}$ be fixed. We make the following assumptions.
For every open and bounded $ \mU \subset \R$, the flux satisfies the following:
      \begin{align}
    \bPhi &\in \mL^{\infty}\big((0, T); \mathcal{W}^{1,\infty}(\R^n\times \mU;\R^n)\big) \label{E:Lipschitz-data},\\
    \nabla\bPhi&\in\mL^{\infty}(\Omega_T;\mathcal{W}^{1, \infty}(\mU))\label{E:D2-Phi-Linf-Lipschitz},\\
    \diverg\bPhi&\in\mL^{1}(\Omega_T;\mL^{\infty}(\mU))\label{E:div_Phi_L1_uniform},\\
\nabla\diverg\bPhi&\in\mL^{1}(\Omega_T;\mL^{\infty}(\mU)).\label{E:integrable-TV-norm}
     \end{align}
     The initial datum satisfies $\rho_{\circ}\in \mL^{\infty}(\R^n)\cap \mTV(\R^n)$.
\end{assumption}

We now define the notion of an entropy solution for the conservation law \eqref{E:model}, following Kru\v{z}kov's framework \cite{kruvzkov1970first}. In this work, we will consider solutions with slightly stronger time regularity than in the classical theory, namely,
\[
\rho \in \mC([0,T]; \mL^1_{\mathrm{loc}}(\R^n)) \cap \mL^{\infty}(\Omega_T),
\]
rather than merely $\mL^{\infty}(\Omega_T)$.

\begin{definition}[entropy solution]
\label{defi:entropy_solution_local}
Let \cref{E:model} be given and let $\rho_{\circ}\in \mL^{\infty}(\R^n)$. We say that  
\[\rho\in \mC\big([0,T];\mL^{1}_{\textnormal{loc}}(\R^n)\big)\cap \mL^{\infty}(\Omega_T)
\] is a Kru\v{z}kov entropy solution on the time horizon $[0,T]$ with $T\in\R_{>0}$ if, for every test function $\phi\in \mC^{1}_{\textnormal{c}}((0,T)\times\R^n;\R_{\geq0})$ and every $k\in\R$, it holds that
\begin{gather*}
    \int_{0}^{T}\int_{\R^n}|\rho(t,\bx)-k|\partial_{t}\phi(t,\bx)+\sign(\rho(t,\bx)-k)\big(\bPhi(t,\bx,\rho(t,\bx))-\bPhi(t,\bx,k)\big)\circ \nabla\phi(t,\bx)\dd\bx\dd t\\
    -\int_{0}^{T}\int_{\R^n}\sign(\rho(t,\bx)-k)\diverg \bPhi(t,\bx,k)\phi(t,\bx)\dd \bx\dd t\geq 0,
\end{gather*}
where $\circ $ is the scalar product between two vectors in $\R^n$. In addition, it holds that
\[
\rho(0,\cdot)\equiv\rho_{\circ}\ \text{ in } \mL^{1}_{\textnormal{loc}}(\R^n).
\]
\end{definition}

Our goal is to show that the local conservation law \eqref{E:model} admits an entropy solution, in the sense of the above definition, for short time horizons given the weaker assumptions on the flux in \cref{Ass:General}. Along with the well-posedness, we will devise suitable a priori bounds on the solution, as well as stability estimates with respect to time and the initial datum. 
Because of the lack of general existence and uniqueness results for conservation laws under \cref{Ass:General}, we will show well-posedness by relaxing the problem and then passing to the limit. The relaxation is achieved by smoothing the flux function so that classical estimates can be applied. 
\begin{definition}[mollified flux]\label{defi:bPhi_delta}
    Let $\Theta \in \mC_c^\infty(\R; \R)$, with $\supp \Theta \subset (-1, 1)$. For $\bm{z}\in\R^{n+2}$ and $\delta\in\R_{>0}$, we define the mollifier $\Theta_\delta(\bm z) \coloneqq \Theta \Big(\norm{\bm z}/\delta\Big) / \delta^{n+2}$. Furthermore, we introduce the extended flux
\begin{equation}\label{E:Phi_extension}
    \tilde \bPhi(t, \bx, \rho) \Def \begin{cases}
        \bPhi(t, \bx, \rho), & (t, \bm x, \rho) \in [0, T] \times \R^n \times \mathcal \R, \\
        0, & (t,\bx, \rho)\in (\R \setminus [0, T]) \times \R^n \times \R.
    \end{cases}
\end{equation}
The mollified flux is thus defined as
\begin{equation}\label{E:lambda_moll}
  \bPhi_\delta(t, \bm x, \rho) \Def (\tilde \bPhi * \Theta_\delta)(t, \bm x, \rho),
\end{equation}
for $ (t, \bx ,\rho) \in \R \times \R^n  \times \R $, with support
\begin{equation*} 
\supp(\bPhi_\delta) \subset \Omega_T^\delta \times \R \Def (- \delta, T + \delta) \times \R^n \times \R. 
\end{equation*}
\end{definition}


\begin{lemma}[properties of the mollified flux]\label{lem:properties_mollified_flux}
Let \cref{Ass:General} hold. Then, for every open bounded $\mU\subset\R$, the  mollified $\bPhi_{\delta}$ defined at \cref{defi:bPhi_delta} satisfies 
\begin{align} 
    \bPhi_\delta& \in \mC^{\infty}(\Omega_T \times \mU), \label{E:regularity_Phi}\\
    \partial_3 \bPhi_\delta&\in \mL^\infty(\Omega_T \times\mU),\label{E:regularity_partial_u_Phi}\\
   \nabla\bPhi_\delta &\in \mL^\infty(\Omega_T \times\mU), \label{E:D-Phi-delta}\\
    \partial_3 \nabla\bPhi_\delta&\in \mL^\infty(\Omega_T \times \mU), \label{E:du-grad-Phi}\\
    \diverg\bPhi_{\delta}&\in \mL^{1}(\Omega_T;\mL^{\infty}(\mathcal{\mU})), \label{E:div-Phi-delta}\\
    \nabla \diverg \bPhi_{\delta}&\in \mL^{1}(\Omega_T;\mL^{\infty}(\mathcal{\mU})). \label{E:grad-div-Phi}
\end{align}
Moreover, \crefrange{E:regularity_partial_u_Phi}{E:grad-div-Phi} are uniform in $\delta$. Specifically, let 
\begin{equation} \label{E:U-delta}
    \mU^\delta \Def \mU + \set{u \in \R\setminus \mU : \dist(u, \partial \mU) \le \delta}
    \end{equation} 
    and $\tilde {\mU} \coloneqq \mU^1$. Then, 
for all $\mU\subset\R$, it holds that
\begin{align}
        \sup_{\delta\in\R_{>0}}\|\partial_3 \bPhi_\delta\|_{\mL^\infty(\Omega_T \times \mU)}&\leq \|\partial_3 \bPhi\|_{\mL^\infty(\Omega_T \times \tilde{\mU})}, \label{E:du-phi-uniform}\\
    \sup_{\delta\in\R_{>0}}\|\nabla \bPhi_\delta\|_{\mL^\infty(\Omega_T \times\mU)} &\leq \|\nabla\bPhi\|_{\mL^\infty(\Omega_T \times\tilde {\mU})}, \label{E:grad-phi-delta-uniform}\\
    \sup_{\delta\in\R_{>0}}\|\partial_3 \nabla\bPhi_\delta\|_{\mL^\infty(\Omega_T \times \mU)}&\leq \|\partial_3 \nabla\bPhi\|_{\mL^\infty(\Omega_T \times \tilde{\mU})}, \label{E:du-grad-phi-delta-uniform}\\
    \sup_{\delta\in\R_{>0}}\|\diverg\bPhi_{\delta}\|_{\mL^{1}(\Omega_T;L^{\infty}(\mU))}&\leq \|\diverg\bPhi\|_{\mL^{1}(\Omega_T;L^{\infty}(\tilde{\mU}))}, \label{E:Div-PhiDelta-uniform}\\
    \sup_{\delta\in\R_{>0}}\|\nabla\diverg \bPhi_{\delta}\|_{\mL^{1}(\Omega_T;L^{\infty}(\mU))} &\leq \|\nabla \diverg \bPhi\|_{\mL^{1}(\Omega_T;L^{\infty}(\tilde{\mU}))}. \label{E:grad-Div-PhiDelta-uniform}
\end{align}
Finally, we have that 
\begin{equation}
\lim_{\delta\rightarrow0} \|\bPhi_{\delta}-\bPhi\|_{\mL^{1}(\Omega_T;\mL^{\infty}(\mU))}=0,\label{eq:uniform_convergence_bphi}
\end{equation}
as well as
\begin{equation}
\lim_{\delta\rightarrow0} \|\diverg \bPhi_{\delta}-\diverg \bPhi\|_{\mL^{1}(\Omega_T;\mL^{\infty}(\mU))}=0.
\label{eq:uniform_convergence_divergenz_bphi}
\end{equation}
\end{lemma}

\begin{proof}
    The regularities of \crefrange{E:regularity_Phi}{E:grad-div-Phi} follow immediately from the definition \cref{E:lambda_moll} under \cref{Ass:General}. In particular, \cref{E:du-grad-Phi} directly follows from the assumption of \cref{E:D2-Phi-Linf-Lipschitz}. 
    To prove \cref{E:du-phi-uniform}, we use \cref{defi:bPhi_delta} and the classic Young inequality for convolutions. In particular, we have 
    \begin{equation*}
        \begin{split}
            \norm{\D_3 \bPhi_\delta}_{\mL^\infty(\Omega_T \times \mU)}  &\le  \norm{\D_3 \tilde{\bPhi} * \Theta_\delta }_{\mL^\infty(\Omega_T \times {\mU})} \le \norm{\D_3 \bm \Phi}_{\mL^\infty(\Omega_T \times \tilde{\mU})}.
        \end{split}
    \end{equation*}
    Here, we have used \cref{E:Phi_extension} and the definition \eqref{E:U-delta} of $\tilde {\mU}$. The proof of \cref{E:grad-phi-delta-uniform} and \eqref{E:du-grad-phi-delta-uniform} follows the same approach.
    To prove \cref{E:Div-PhiDelta-uniform}, by invoking the definition \cref{E:lambda_moll} and \cref{Ass:General}, we see that 
    \begin{equation}
    \begin{split}    
       & \int_0^T \int_{\R^n} \Big\{\esssup_{u \in \mU} \Big(\iiint_{(-\delta, \delta)^{n+2}} \abs[\big]{\diverg \tilde{\bPhi}(t-s, \bx-\by, u - v) \Theta_\delta(s, \by, v)} \dd s \dd \by \dd v\Big) \Big\} \dd \bx \dd t \\
       & \hspace{0.5in} \le \int_0^T \!\!\!\int_{\R^n}  \Bigg\{ \iint_{(-\delta,\delta)^{n+1}} \Big(\norm[\big]{\diverg \tilde{\bPhi}(t-s, \bx-\by,\cdot)}_{\mL^\infty(\tilde{\mU})}\int_{(-\delta,\delta)}\!\!\!\!\! \Theta_\delta(s, \by, v) \dd v \Big) \dd s \dd \by\Bigg\} \dd \bx \dd t \\
       & \hspace{0.5in} =    \iiint_{(-\delta,\delta)^{n+2}}\Theta_\delta(s, \by, v) \Big(\int_0^T \!\!\! \int_{\R^n}\norm[\big]{\diverg \tilde{\bPhi}(t-s, \bx-\by,\cdot)}_{\mL^\infty(\tilde{\mU})}\dd \bx \dd t \Big) \dd s \dd \by \dd v  \\
       & \hspace{0.5in} \le \norm[\big]{\diverg \bPhi}_{\mL^1(\Omega_T; \mL^\infty(\tilde {\mU}))}.
    \end{split}
    \end{equation}
    Here, we are using Fubini's theorem, \cref{Ass:General}, and \cref{defi:bPhi_delta}. The proof of \cref{E:grad-Div-PhiDelta-uniform} follows similarly. 

Concerning \cref{eq:uniform_convergence_bphi}, we estimate directly. 
Since $\Theta_\delta$ is a mollifier with $\norm{\Theta_\delta}_{\mL^1(\R^n)} = 1$, we write
\begin{align*}
&\int_{0}^{T}\!\!\!\int_{\R^n}\esssup_{u\in\mU}\bigg|\iiint_{(-\delta, \delta)^{n+2}} \tilde{\bPhi}(t-s, \bx-\by, u - v) \Theta_\delta(s, \by, v) \dd s \dd \by \dd v-\tilde{\bPhi}(t,\bx,u)\bigg| \dd \bx \dd t\\
&\leq \int_{0}^{T}\int_{\R^n}\esssup_{u\in\mU}\bigg|\iiint_{(-\delta, \delta)^{n+2}}\Big( \tilde{\bPhi}(t-s, \bx-\by, u - v)-\tilde{\bPhi}(t-s,\bx-\by,u)\\
&\qquad+\tilde{\bPhi}(t-s,\bx-\by,u) -\tilde{\bPhi}(t,\bx,u) \Big)\Theta_\delta(s, \by, v)\dd s \dd \by \dd v\bigg| \dd \bx \dd t\\
&\leq\int_{0}^{T}\!\!\!\int_{\R^n}\esssup_{u\in\mU}\iiint_{(-\delta, \delta)^{n+2}}\|\partial_{3}\tilde{\bPhi}(t-s,\bx-\by,\cdot)\|_{L^{\infty}(\tilde{\mU})}|v|\Theta_{\delta}(s,\by,v)\dd s\dd\by\dd v\dd \bx\dd t\\
&\quad +\int_{0}^{T}\int_{\R^n}\esssup_{u\in\mU}\iiint_{(-\delta, \delta)^{n+2}}\big|\tilde{\bPhi}(t-s,\bx-\by,u)-\tilde{\bPhi}(t,\bx,u)\big|\Theta_{\delta}(s,\by,v)\dd s\dd \by\dd v\dd\bx\dd t\\
&\leq \delta\|\partial_{3}\tilde{\bPhi}\|_{L^{1}(\Omega_T;\mL^{\infty}(\tilde{\mU}))}\\
&\quad +\int_{0}^{T}\int_{\R^n}\iiint_{(-\delta, \delta)^{n+2}}\big\|\tilde{\bPhi}(t-s,\bx-\by,\cdot)-\tilde{\bPhi}(t,\bx,\cdot)\big\|_{L^{\infty}(\mU)}\Theta_{\delta}(s,\by,v)\dd s\dd \by\dd v\dd\bx\dd t.
\intertext{At this point, exchanging the order of integration,}
&\leq \delta\|\partial_{3}\tilde{\bPhi}\|_{L^{1}(\Omega_T;\mL^{\infty}(\tilde{\mU}))}\\
&\quad+\iiint_{(-\delta, \delta)^{n+2}}\Big(\int_{0}^{T}\int_{\R^n}\big\|\tilde{\bPhi}(t-s,\bx-\by,\cdot)-\tilde{\bPhi}(t,\bx,\cdot)\big\|_{L^{\infty}(\mU)}\dd\bx\dd t\Big)\Theta_{\delta}(s,\by,v)\dd s\dd \by\dd v\\
&\leq \delta\|\partial_{3}\bPhi\|_{L^{1}(\Omega_T;\mL^{\infty}(\tilde{\mU}))}+\sup_{\by\in[-\delta,\delta]^{n},\ s\in(-\delta,\delta)}\|\tilde{\bPhi}(\cdot -s,\ast-\by,\star)-\tilde{\bPhi}(\cdot,\ast,\star)\|_{\mL^{1}(\Omega_T;\mL^{\infty}(\mU))}.
\end{align*}
The first term vanishes as $\delta\to 0$, since 
$\|\partial_3\bPhi\|_{\mL^1(\Omega_T;\mL^\infty(\tilde{\mU}))}$ is finite by 
\cref{Ass:General}. The second term vanishes by the $\mL^1$-continuity of translations; 
see \cite[Lemma~4.3]{brezis2011}.

We now turn to \cref{eq:uniform_convergence_divergenz_bphi}. The same 
argument as above, now applied to $\diverg\bPhi$ instead of $\bPhi$, yields
\begin{align}
    \big\|\diverg\bPhi_{\delta}-\diverg \bPhi\big\|_{\mL^{1}(\Omega_T;\mL^{\infty}(\mU))} =\int_{0}^{T}\int_{\R^n}\esssup_{u\in \mU}
    |\diverg\bPhi_{\delta}(t,\bx,u)-\diverg \bPhi(t,\bx,u)|\dd \bx\dd t.\label{eq:div_Phi_Phi_delta}
\end{align}
Since $\bPhi$ is regular enough for the divergence to commute with the convolution, 
we may write $\diverg\bPhi_\delta = \diverg\widetilde{\bPhi}*\Theta_\delta$ and use 
$\int\Theta_\delta = 1$ to obtain
\begin{align*}
    \eqref{eq:div_Phi_Phi_delta}&=\int_{0}^{T}\!\!\!\int_{\R^n}\!\!\!\!\esssup_{u\in \mU}\bigg|
    \iiint_{(-\delta,\delta)^{n+2}}\!\!\!\!\!\!\!\!\!\!\!\!\!\!\!
    \big(\diverg\bPhi(t-s,\bx-\by,u-v)-\diverg\bPhi(t,\bx,u)\big)
    \Theta_{\delta}(s,\by,v)\dd s\dd\by\dd v\bigg|\dd \bx\dd t.
\intertext{Adding and subtracting $\diverg\bPhi(t-s,\bx-\by,u)$ inside the integral and applying the triangle inequality gives}
    &\leq\int_{0}^{T}\!\!\int_{\R^n}\!\!\esssup_{u\in \mU}\!\bigg|\!
    \iiint_{(-\delta,\delta)^{n+2}}\!\!\!\!\!\!\!\!\!\!\!\!\!\!\!
    \big(\diverg\bPhi(t-s,\bx-\by,u-v)-\diverg\bPhi(t-s,\bx-\by,u)\big)
    \Theta_{\delta}(s,\by,v)\!\dd s\!\dd\by\dd v\bigg|\!\dd \bx\!\dd t\\
    &\quad +\int_{0}^{T}\!\!\!\int_{\R^n}\!\!\!\!\esssup_{u\in \mU}\bigg|
    \iiint_{(-\delta,\delta)^{n+2}}\!\!\!\!\!\!\!\!\!\!\!\!\!\!\!
    \big(\diverg\bPhi(t-s,\bx-\by,u)-\diverg \bPhi(t,\bx,u)\big)
    \Theta_{\delta}(s,\by,v)\dd s\dd\by\dd v\bigg|\dd \bx\dd t.
\intertext{At this point, using straightforward estimates, we obtain}
    &\leq\int_{0}^{T}\int_{\R^n}\iiint_{(-\delta,\delta)^{n+2}}\big\|\partial_{3}\diverg\bPhi(t-s,\bx-\by,\cdot)\big\|_{\mL^{\infty}(\tilde{\mathcal U})}|v|\Theta_{\delta}(s,\by,v)\dd s\dd\by\dd v\dd \bx\dd t\\
    &\quad +\int_{0}^{T}\int_{\R^n}\iiint_{(-\delta,\delta)^{n+2}}\!\!\!\!\!\!\!\!\!\big\|\diverg\bPhi(t-s,\bx-\by,\cdot)-\diverg\bPhi(t,\bx,\cdot)\big\|_{\mL^{\infty}(\mU)}\Theta_{\delta}(s,\by,v)\dd s\dd\by\dd v\dd \bx\dd t\\
    &\leq \delta\big\|\partial_{3}\diverg\bPhi\big\|_{\mL^{1}(\Omega_T;\mL^{\infty}(\tilde{\mathcal U}))}+\sup_{\by\in[-\delta,\delta]^{n},s\in(-\delta,\delta)}\big\|\diverg\bPhi(\cdot-s,\ast-\by,\star)-\diverg\bPhi(\cdot,\ast,\star)\big\|_{\mL^{1}(\Omega_T;\mL^{\infty}(\mU))}.
\end{align*}
Both terms vanish as $\delta\to 0$ by the same reasoning: the first by the bound on
$\|\partial_3\diverg\bPhi\|_{\mL^1(\Omega_T;\mL^\infty(\tilde{\mU}))}$, which follows 
from \cref{E:D2-Phi-Linf-Lipschitz}, and the second by the $\mL^1$-continuity of translations.
\end{proof}

We are now in a position to apply the results in \cite{Mercier2011improved} to the smoothed conservation law for $\delta\in\R_{>0}$. In particular, we consider 
\begin{equation}
\begin{aligned}    
\partial_t \rho + \diverg \bPhi_{\delta}(t, \bx, \rho) &= 0, &&(t, \bx) \in (0, T) \times \R^n,\\
\rho(0, \cdot) &\equiv \rho_{\circ},&& \text{ on }\R^n.
\end{aligned}
    \label{E:model_smoothed}
\end{equation} 
The definition of an entropy solution for this equation is exactly the same as in \cref{defi:entropy_solution_local}, except with $\Phi_\delta$ in place of $\Phi$. 

\medskip
\noindent
\textbf{A word of warning.}
\Cref{Ass:General} requires us to fix a time horizon $T>0$.
Accordingly, throughout the remainder of the paper, we work with this fixed $T$, which will no longer be stated explicitly in the results.

We can now state our first existence result.
\begin{theorem}[existence, uniqueness, and regularity for the mollified equation]\label{theo:existence_uniqueness_smooth}
Let \cref{Ass:General} hold and $\delta>0$ be given, and let $\bPhi_\delta$ be the mollified flux from \cref{defi:bPhi_delta}. Then there exists a time horizon $T^* \in (0,T]$ and a unique weak entropy solution of the mollified problem \eqref{E:model_smoothed},
\[
\rho_\delta \in \mC([0,T^*]; \mL^1(\R^n)) \cap \mL^\infty((0,T^*); \mL^\infty(\R^n) \cap \mTV(\R^n)),
\]
in the sense of \cref{defi:entropy_solution_local}.
A lower bound on $T^*$ is provided by the maximal existence interval of the solution to the initial value problem 
\begin{equation}
\tfrac{d}{dt} \alpha(t) = \norm{\diverg \bPhi_\delta(t,\cdot,\alpha(t))}_{\mL^\infty(\R^n)}, 
\quad \alpha(0) = \|\rho_0\|_{\mL^\infty(\R^n)} .
\label{eq:L_infty_increase}
\end{equation}
Moreover, if we set 
\begin{equation} \label{E:bar-rho} 
\bar{\rho} \coloneqq 42 \max\{\|\rho_0\|_{\mL^\infty(\R^n)},1\},
\end{equation} 
for
\begin{equation}\label{E:small-time-delta}
t \in \Big[0, T^* \coloneqq \min\Big\{T, \tfrac{41 \max\{\|\rho_0\|_{\mL^\infty(\R^n)},1\}}{\|\diverg \bPhi\|_{\mL^\infty(\Omega_T \times (-\bar{\rho}, \bar{\rho}))}} \Big\} \Big],
\end{equation}
then the solution is uniformly bounded in both the $\mL^\infty$-norm and the total variation. More precisely, over $[0, T^*]$, we have
\begin{align} 
\|\rho_\delta(t,\cdot)\|_{\mL^\infty(\R^n)} &\le \bar{\rho}, \label{E:boundedness-delta}\\
|\rho_\delta(t,\cdot)|_{\mTV(\R^n)} &\le |\rho_0|_{\mTV(\R^n)} \, \e^{\kappa t} 
+ n W_n \int_0^t \e^{\kappa(t-s)} \int_{\R^n} 
\|\nabla \diverg \bPhi(s,\bx,\cdot)\|_{\mL^\infty((-\bar{\rho},\bar{\rho});\R^n)} \, \dd {\bx} \, \dd s, \label{E:TV-bound-delta}
\end{align}
where
\[
\kappa \coloneqq (2n+1) \|\diverg \partial_3 \bPhi\|_{\mL^\infty(\Omega_{T^*} \times (-\bar{\rho},\bar{\rho}))}.
\]
Finally, the solution is Lipschitz continuous in time with respect to the $\mL^1$-norm: for any $s,t \in [0,T^*]$ with $s < t$,
\begin{equation}
\begin{split}
\|\rho_\delta(t,\cdot) - \rho_\delta(s,\cdot)\|_{\mL^1(\R^n)} 
&\le \int_s^t \int_{\R^n} \|\diverg \bPhi(\tau,\bx,\cdot)\|_{\mL^\infty((-\bar{\rho},\bar{\rho}); \R)} \, \dd \bx \, \dd\tau \\
&\quad + |t-s| \, \|\partial_3 \bPhi\|_{\mL^\infty(\Omega_{T^*} \times (-\bar{\rho},\bar{\rho}))} \, \|\rho_\delta\|_{\mL^\infty((0,T^*); \mTV(\R^n))}.
\end{split}
\label{eq:time_continuity-delta}
\end{equation}
Let $\rho_\delta$ and $\tilde{\rho}_\delta$ be the solutions of \cref{E:model_smoothed} with fluxes $\bPhi_\delta$ and $\hat{\bPhi}_\delta$ and initial data $\rho_\circ$ and $\tilde{\rho}_\circ$, respectively. Then, the following general stability result holds: \begin{equation}\label{E:rho_delta_general_statbility}
    \begin{split}
    \|\rho_\delta(t,\cdot) &- \tilde{\rho}_\delta(t,\cdot)\|_{\mL^1(\R^n)} \le \e^{\kappa_{\bPhi} t} \|\rho_\circ - \tilde{\rho}_\circ\|_{\mL^1(\R^n)}
      + t \, \e^{\kappa_{\bPhi} t} |\rho_\circ|_{\mTV(\R^n)} \|\partial_3 \bPhi_\delta - \partial_3 \hat{\bPhi}_\delta\|_{\mL^\infty(\Omega_T \times [-\bar{\rho},\bar{\rho}])} \\
    &+ n W_n \|\partial_3 \bPhi_\delta - \partial_3 \hat{\bPhi}_\delta\|_{\mL^\infty(\Omega_T \times [-\bar{\rho},\bar{\rho}])} 
          \int_0^t s \, \e^{\kappa_{\bPhi} (t-s)} \|\nabla \diverg \bPhi_\delta(s, \cdot, \cdot)\|_{\mL^1(\R^n;\mL^\infty((-\bar{\rho},\bar{\rho})))} \dd s \\
    &+ \int_0^t \e^{\kappa_{\bPhi_\delta }(t-s)} \|\diverg \bPhi_\delta(s, \cdot, \cdot) - \diverg \hat{\bPhi}_\delta(s, \cdot, \cdot)\|_{\mL^1(\R^n;\mL^\infty((-\bar{\rho},\bar{\rho})))} \dd s,
    \end{split}
    \end{equation}
    where
      \[
    \kappa_{\bPhi_\delta} \coloneqq \max\Big\{ \|\partial_3 \bPhi_\delta - \partial_3 \hat{\bPhi}_\delta\|_{\mL^\infty(\Omega_T \times [-\bar{\rho},\bar{\rho}])}, \, 
    (2n+1) \|\nabla \partial_3 \bPhi_\delta\|_{\mL^\infty(\Omega_T \times (-\bar{\rho},\bar{\rho}))} \Big\}.
    \]
\end{theorem}

\begin{proof}
The Assumptions \crefrange{E:regularity_Phi}{E:grad-div-Phi} on the flux $\bPhi_\delta$ provide the required regularities of \cite[Assumptions $(\bm H1^*)$--$(\bm H3^*)$]{Mercier2011improved} which ensure the well-posedness of \cref{E:model_smoothed}.\\
To show that the solution remains bounded uniformly in $\delta$, our approach follows from the fact that the evolution of the $\mL^\infty$-norm in \cref{eq:L_infty_increase} arises from tracking the maximal value of the PDE solution and obtaining a worst-case estimate with respect to the spatial variable.

The right-hand side of the ODE in \cref{eq:L_infty_increase} is locally Lipschitz continuous. To see this, consider the map 
\begin{equation}
h \mapsto \|\diverg \bPhi_\delta(t,\cdot,h)\|_{\mL^\infty(\R^n)}.
\label{eq:Lipschitz_continuity_Div_bPhi}
\end{equation}
For any $h, \tilde{h} \in \R$, classical estimates \tup{using the Minkowski inequality} yield
\begin{equation*}
\begin{split}
\big|\norm{\diverg \bPhi_\delta(t,\cdot,h)}_{\mL^\infty(\R^n)} - \norm{\diverg \bPhi_\delta(t,\bx,\tilde{h})}_{\mL^\infty(\R^n)}\big| & \le \|\diverg \bPhi_\delta(t,\cdot,h) - \diverg \bPhi_\delta(t,\cdot,\tilde{h})\|_{\mL^\infty(\R^n)} \\
& \le \|\diverg \partial_3 \bPhi_\delta(t,\cdot,\ast)\|_{\mL^\infty(\R^n \times \mathcal{H})} |h - \tilde{h}|,
\end{split}
\end{equation*}
where $\mathcal{H} \coloneqq (\min\{h,\tilde{h}\}, \max\{h,\tilde{h}\})$. Since the quantity $\|\diverg \partial_3 \bPhi_\delta(t,\cdot,\ast)\|_{\mL^\infty(\R^n \times \mathcal{H})}$ is finite by \cref{Ass:General} and in particular \cref{E:du-grad-phi-delta-uniform}, the differential equation
\[
\tfrac{d}{dt}\alpha(t) = \|\diverg \bPhi_\delta(t,\cdot,\alpha(t))\|_{\mL^\infty(\R^n)}, \quad \alpha(0) = \|\rho_0\|_{\mL^\infty(\R^n)},
\]
admits a solution on a sufficiently short time interval. To make this interval uniform in $\delta$, we apply the comparison principle for ODEs. Let $T>0$ be such that $|\alpha(t)| \le 41 \|\rho_0\|_{\mL^\infty(\R^n)}$ for every $t \in [0,T]$ \tup{note that such a $T$ exists by continuity, since $\alpha(0) = \|\rho_0\|_{\mL^\infty(\R^n)}$}. Then, in light of \cref{lem:properties_mollified_flux}---in particular, \cref{E:grad-phi-delta-uniform}---we have
\[
\tfrac{d}{dt}\alpha(t) = \|\diverg \bPhi_\delta(t,\cdot,\alpha(t))\|_{\mL^\infty(\R^n)} \le \|\diverg \bPhi\|_{\mL^\infty(\Omega_T \times (\|\rho_0\|_{\mL^\infty(\R^n)}, \bar \rho))}.
\]
Note that we can take the asymmetric interval in the $\mL^\infty$-norm on the right-hand side, since the initial datum of the ODE is positive, and so is the value of the derivative.
Integrating in time gives
\[
\alpha(t) \le \|\rho_0\|_{\mL^\infty(\R^n)} + t \, \|\diverg \bPhi\|_{\mL^\infty(\Omega_T \times (\|\rho_0\|_{\mL^\infty}, 42\|\rho_0\|_{\mL^\infty}))} \overset{!}{\le} 42 \norm{\rho_0}_{\mL^\infty(\R^n)}.
\]
Here, $\overset{!}{\le}$ denotes the desired bound on the right-hand side. Consequently, we obtain the uniform time horizon
\[
\overline{T} = \frac{41 \|\rho_0\|_{\mL^\infty(\R^n)}}{\|\diverg \bPhi\|_{\mL^\infty(\Omega_T\times (\|\rho_0\|_{\mL^\infty}, \bar \rho))}}.
\]
In fact, we have established a time horizon for which the solution remains bounded uniformly in $\delta$, with the stated $\mL^\infty$ bound. In fact, this proves \cref{E:boundedness-delta}.

The bounds \crefrange{E:TV-bound-delta}{E:rho_delta_general_statbility} follows directly from \cite{Mercier2011improved} under the assumptions \crefrange{E:regularity_Phi}{E:grad-div-Phi}. This completes the proof.
\end{proof}
\begin{remark}[the domain of existence of a solution]
It should be noted that $T^*$ is only a lower bound on the existence of such a time horizon---i.e. is not a tight bound. In addition, it depends on the constant $42 \norm{\rho_\circ}_{\mL^\infty(\R^n)}$. 
\end{remark}
Having established a uniform time horizon for existence, we now turn to the compactness and convergence of entropy solutions corresponding to the smoothed flux. In particular, we show that, up to subsequences, these solutions converge to a limit in $\mC([0,T^*];\mL^1_{\textnormal{loc}}(\R^n))$.

\begin{lemma}[{convergence of $\rho_\delta$ as $\delta \to 0$ in $\mC([0,T^*]; \mL^{1}_{\textnormal{loc}}(\R^n))$}]
\label{lem:compactness_rho_delta}
Let \cref{Ass:General} hold, and suppose that a sufficiently small $T^*$ is chosen to satisfy the condition of \cref{theo:existence_uniqueness_smooth}. More precisely, the sequence of weak entropy solutions $\{\rho_\delta\}_{\delta>0}$---as in \cref{defi:entropy_solution_local}---satisfying
\[
\rho_\delta \in \mC([0,T^*];\mL^1_{\textnormal{loc}}(\R^n)) \cap \mL^\infty((0,T^*);\mL^\infty(\R^n))
\]
is relatively compact in $\mC([0,T^*];\mL^1(K))$ for any bounded measurable set $K \subset \R^n$. That is,
\[
\forall \text{ bounded measurable sets } K \subset \R^n: \quad 
\Big\{\rho_\delta|_{[0,T^*]\times K} : \delta \in \R_{>0} \Big\} \text{ is totally bounded in } \mC([0,T^*];\mL^1(K)).
\]
In particular, there exists a subsequence $(\delta_k)_{k \in \mathbb{N}} \subset \R_{>0}$ with $\lim_{k \to \infty} \delta_k = 0$ and a limit function $\rho_* \in \mC([0,T^*];\mL^1(K))$ such that
\[
\lim_{k \to \infty} \|\rho_{\delta_k} - \rho_*\|_{\mC([0,T^*];\mL^1(K))} = 0.
\]
\label{lem:rhostar-limit}
\end{lemma}

\begin{proof}
We start by showing the following conditions.

\smallskip
\noindent\textit{Step 1: Compact embedding.}
For any bounded measurable set $K\subset\R^n$, the embedding
$\mBV(K)\overset{\mathrm{c}}{\hookrightarrow}\mL^{1}(K)$
is compact; see \cite[Theorems 13.32, 13.35]{leoni2024first}, which proves that the corresponding map is a compact operator.

\smallskip
\noindent\textit{Step 2: Uniform boundedness in 
$\mL^{1}((0,T^{*});\mBV(K))$.}
By \cref{theo:existence_uniqueness_smooth}, the family 
$\{\rho_{\delta}\}_{\delta>0}$ satisfies uniform bounds in 
$\mL^{\infty}$ and $\mTV$ on $[0,T^*]$, provided that the 
solution remains bounded in $\mL^{\infty}(\R^n)$ on that interval.
In particular,
\[
\sup_{\delta>0}\|\rho_{\delta}\|_{\mL^{\infty}((0,T^{*});\mBV(K))}<\infty,
\]
from which uniform boundedness in $\mL^{1}((0,T^{*});\mBV(K))$
follows immediately by integrating in time.

\smallskip
\noindent\textit{Step 3: Equicontinuity in time.}
It remains to show that the family $\{\rho_\delta\}_{\delta > 0}$ 
is equicontinuous in time with values in $\mL^1(K)$, that is,
\[
\lim_{h\to 0}\;\sup_{\delta>0}\;
\sup_{t\in[0,T^{*}-h]}
\|\rho_{\delta}(t+h,\cdot)-\rho_{\delta}(t,\cdot)\|_{\mL^{1}(K)}
= 0.
\]
This follows directly from the time continuity estimate 
\cref{eq:time_continuity-delta}. This follows directly from \cref{eq:time_continuity-delta} and \crefrange{E:D2-Phi-Linf-Lipschitz}{E:div_Phi_L1_uniform}. 
\smallskip
Putting these together and considering \cite[Chapter 6, Theorem 3]{simon1986compact}, we observe that
$\big\{\rho_{\delta}|_{[0,T^{*}]\times K}:\delta \in \R_{>0}\big\}$
is relatively compact in $\mC([0,T^{*}];\mL^{1}(K))$; the claim follows. 
\end{proof}

We finally turn to our main result of this section. 

\begin{theorem}[existence and uniqueness on a short time horizon for \cref{E:model}]
\label{theo:existence_uniqueness_less_regular_flux}
Let \cref{Ass:General} hold, let $\rho_\delta$ be the solution of \cref{E:model_smoothed}, and suppose that $\rho_
*$ is defined as in \cref{lem:rhostar-limit}. Then, there exists a short time horizon $T^*\in\R_{>0}$ such that 
\[\rho_{*}\in\mC\big([0,T^{*}];\mL^{1}_{\loc}(\R^n)\big)\cap\mL^{\infty}((0,T^*);\mL^{\infty}(\R^n))\]
is the unique Kru\v{z}kov entropy solution \cref{E:model} in the sense of \cref{defi:entropy_solution} over $[0,T^*]$.
Furthermore, $\rho*$ satisfies the following:
\begin{description}
    \item[$\mL^\infty$ bound:] Considering the $\bar \rho$ defined as in \cref{E:bar-rho}, we have for \(t\in(0,T^{*})\) a.e.
    \begin{equation}\label{E:rho_star_bound}
    \|\rho_*(t,\cdot)\|_{\mL^\infty(\R^n)} \le \bar{\rho}.
    \end{equation}

    \item[$\mTV$ bound:] We have for \(t\in(0,T^{*})\) a.e.
    \begin{equation}\label{E:rho_star_TVbound}
    \abs{\rho_*(t,\cdot)}_{\mTV(\R^n)} 
    \le \abs{\rho_\circ}_{\mTV(\R^n)} \, \e^{\kappa t}
    + n W_n \int_0^t \e^{\kappa(t-s)} 
      \int_{\R^n} \|\nabla \diverg \bPhi(s, \bx, \cdot)\|_{\mL^\infty((-\bar{\rho},\bar{\rho}); \R^n)} \, \dd\bx \, \dd s,
    \end{equation}
    with 
    \[
    \kappa \coloneqq (2n+1) \|\diverg \partial_3 \bPhi\|_{\mL^\infty(\Omega_T\times(-\bar{\rho},\bar{\rho}))}.
    \]

    \item[Time continuity:]
    For $t,s \in [0,T^*]$ with \(s<t\),
    \begin{equation}\label{E:rho_star-timecontinuity}
    \begin{split}
    \|\rho_*(t,\cdot) - \rho_*(s,\cdot)\|_{\mL^1(\R^n)} &\le \int_s^t \int_{\R^n} \|\diverg \bPhi(\tau, \bx, \cdot)\|_{\mL^\infty([-\bar{\rho},\bar{\rho}])} \, \dd\bx \, \dd\tau \\
    &\quad + |t-s| \, \|\partial_3 \bPhi\|_{\mL^\infty((0,T)\times \R^n \times (-\bar{\rho},\bar{\rho}))} \, \|\rho_*\|_{\mL^\infty((0,T);\mTV(\R^n))}.
    \end{split}
    \end{equation}

    \item[General stability:] Let $\tilde{\rho}_\circ \in \mL^\infty(\R^n) \cap \mTV(\R^n)$ and $\hat{\bPhi}$ satisfy \cref{Ass:General}, with corresponding existence horizon $T^{*}$. Denote by $\rho$ and $\tilde{\rho}$ the Kru\v{z}kov entropy solutions for $\rho_\circ$ and $ \tilde{\rho}_\circ$ with fluxes $\bPhi$ and $\hat{\bPhi}$, respectively. Then, for all $t \in [0,T^{*})$,
    \begin{equation}\label{E:rho_star_general_statbility}
    \begin{split}
    \|\rho(t,\cdot) &- \tilde{\rho}(t,\cdot)\|_{\mL^1(\R^n)} \le \e^{\kappa_{\bPhi} t} \|\rho_\circ - \tilde{\rho}_\circ\|_{\mL^1(\R^n)}
      + t \, \e^{\kappa_{\bPhi} t} |\rho_\circ|_{\mTV(\R^n)} \|\partial_3 \bPhi - \partial_3 \hat{\bPhi}\|_{\mL^\infty(\Omega_T \times [-\bar{\rho},\bar{\rho}])} \\
    &+ n W_n \|\partial_3 \bPhi - \partial_3 \hat{\bPhi}\|_{\mL^\infty(\Omega_T \times [-\bar{\rho},\bar{\rho}])} 
          \int_0^t s \, \e^{\kappa_{\bPhi} (t-s)} \|\nabla \diverg \bPhi(s, \cdot, \cdot)\|_{\mL^1(\R^n;\mL^\infty((-\bar{\rho},\bar{\rho})))} \dd s \\
    &+ \int_0^t \e^{\kappa_{\bPhi}(t-s)} \|\diverg \bPhi(s, \cdot, \cdot) - \diverg \hat{\bPhi}(s, \cdot, \cdot)\|_{\mL^1(\R^n;\mL^\infty((-\bar{\rho},\bar{\rho})))} \dd s,
    \end{split}
    \end{equation}
    \[
    \kappa_{\bPhi} \coloneqq \max\Big\{ \|\partial_3 \bPhi - \partial_3 \hat{\bPhi}\|_{\mL^\infty(\Omega_T \times [-\bar{\rho},\bar{\rho}])}, \, 
    (2n+1) \|\nabla \partial_3 \bPhi\|_{\mL^\infty(\Omega_T \times (-\bar{\rho},\bar{\rho}))} \Big\}.
    \]
\end{description}
\end{theorem}

\begin{proof}
We establish existence by an approximation argument, exploiting 
\cref{theo:existence_uniqueness_smooth}. That result provides, for each 
$\delta>0$, a unique Kru\v{z}kov entropy solution
\[
\rho_{\delta}\in \mC\big([0,T^{*}];\mL^{1}_{\textnormal{loc}}(\R^n)\big)
\]
that satisfies, for all $k\in\R$ and all 
$\phi\in \mC^{1}_{\mathrm{c}}\big(\Omega_T;\R_{\geq0}\big)$,
\begin{equation}\label{E:Kruzkov-rho-delta}
\begin{split}
    &\int_{0}^{T}\!\!\!\!\int_{\R^n}\!\!\!
        |\rho_{\delta}-k|\,\partial_{t}\phi(t,\bx)
        +\sign(\rho_{\delta}-k)
        \big(\bPhi_{\delta}(t,\bx,\rho_{\delta})-\bPhi_{\delta}(t,\bx,k)\big)
        \circ \nabla\phi(t,\bx)\,\dd\bx\dd t\\
    &\hspace{1in}
        -\int_{0}^{T}\!\!\!\int_{\R^n}
        \sign(\rho_{\delta}-k)\,
        \diverg\bPhi_{\delta}(t,\bx,k)\,\phi(t,\bx)\,\dd \bx\dd t
    \geq 0,
\end{split}
\end{equation}
together with $\rho_{\delta}(0,\cdot)\equiv \rho_{\circ}$ on $\R^n$.
Our goal is to pass to the limit $\delta\to 0$ in each term of 
\cref{E:Kruzkov-rho-delta}.

\smallskip
\noindent\textit{First term.}
Since $f\mapsto |f-k|$ is Lipschitz continuous with constant $1$, and 
$\phi$ has compact support in the spatial variable $\bx$, we have
\begin{align*}
    &\bigg|\int_{0}^{T}\!\!\!\int_{\R^n}
        \big(|\rho_{\delta}-k|-|\rho_{*}-k|\big)
        \partial_{t}\phi\,\dd\bx\dd t\bigg|
    \leq \|\partial_t\phi\|_{\mL^\infty(\Omega_T)}
        \|\rho_{\delta}-\rho_{*}\|_{\mL^{1}(\supp\phi)}.
\end{align*}
By \cref{lem:compactness_rho_delta}, $\rho_\delta\to\rho_*$ strongly in 
$\mL^{1}_{\mathrm{loc}}(\Omega_T)$, so the right-hand side 
vanishes as $\delta\to 0$, giving
\[
\lim_{\delta\to 0}
\int_{0}^{T}\!\!\!\int_{\R^n}
    |\rho_{\delta}-k|\,\partial_{t}\phi\,\dd\bx\dd t
=\int_{0}^{T}\!\!\!\int_{\R^n}
    |\rho_{*}-k|\,\partial_{t}\phi\,\dd\bx\dd t.
\]

\smallskip
\noindent\textit{Second term.}
We decompose the error via a triangle inequality:
\begin{align*}
    &\bigg|\iint \sign(\rho_\delta - k)
        \big(\bPhi_\delta(t,\bx,\rho_\delta)-\bPhi_\delta(t,\bx,k)\big)
        \circ\nabla\phi\,\dd\bx\dd t \\
    &\hspace{0.5in} -\iint \sign(\rho_* - k)
        \big(\bPhi(t,\bx,\rho_*)-\bPhi(t,\bx,k)\big)
        \circ\nabla\phi\,\dd\bx\dd t\bigg| \leq I_1 + I_2,
\end{align*}
where
\begin{align*}
    I_1 &\coloneqq \bigg|\iint
        \Big[
        \sign(\rho_\delta-k)\big(\bPhi_\delta(t,\bx,\rho_\delta)
            -\bPhi_\delta(t,\bx,k)\big)
        -\sign(\rho_*-k)\big(\bPhi_\delta(t,\bx,\rho_*)
            -\bPhi_\delta(t,\bx,k)\big)
        \Big]\circ\nabla\phi\,\dd\bx\dd t\bigg|,\\
    I_2 &\coloneqq \bigg|\iint
        \sign(\rho_*-k)
        \Big[
        \big(\bPhi_\delta(t,\bx,\rho_*)-\bPhi_\delta(t,\bx,k)\big)
        -\big(\bPhi(t,\bx,\rho_*)-\bPhi(t,\bx,k)\big)
        \Big]\circ\nabla\phi\,\dd\bx\dd t\bigg|.
\end{align*}
For $I_1$, we have that the mapping $u\mapsto\sign(u-k)\big(\bPhi_\delta(t,\bx,u)-\bPhi_\delta(t,\bx,k)\big)$ 
is Lipschitz in $u$, uniformly in $\delta$, with Lipschitz constant 
$\|\partial_3\bPhi_\delta\|_{\mL^\infty}$, which is uniformly bounded 
by \cref{E:du-phi-uniform}. Together with the compact support of $\phi$ 
and the strong convergence $\rho_\delta\to\rho_*$ in 
$\mL^1_\mathrm{loc}$ from \cref{lem:compactness_rho_delta}, we obtain 
$I_1\to 0$ as $\delta\to 0$.
Concerning $I_2$, since $|\sign(\cdot)|\leq 1$ and $\phi$ has compact support,
\[
I_2\leq \|\nabla\phi\|_{\mL^\infty}
    \|\bPhi_\delta(\cdot,\cdot,\rho_*)-\bPhi(\cdot,\cdot,\rho_*)\|_{\mL^1(\supp\phi)}
    +\|\nabla\phi\|_{\mL^\infty}
    \|\bPhi_\delta(\cdot,\cdot,k)-\bPhi(\cdot,\cdot,k)\|_{\mL^1(\supp\phi)},
\]
which tends to zero as $\delta\to 0$ by \cref{eq:uniform_convergence_bphi}.
Hence, the second term converges to the claimed limit.

\noindent\textit{Third term.}
The convergence of the last term in \cref{E:Kruzkov-rho-delta} 
is more delicate, since strong $\mL^1$ convergence of $\rho_\delta$ 
does not directly imply the convergence of $\sign(\rho_\delta - k)$.
We argue as follows.

Let $K\subset\R^n$ be any compact set containing the spatial
support of $\phi$. Since $\rho_\delta\to\rho_*$ in 
$\mC([0,T^*];\mL^1(K))$ by \cref{lem:compactness_rho_delta},
there exists a subsequence $(\delta_n)_{n\in\NN}$ along which 
$\rho_{\delta_n}\to\rho_*$ pointwise a.e.\ on $[0,T^*]\times K$.
At any point $(t,\bx)$ where this convergence holds and 
$\rho_*(t,\bx)\neq k$, the sign function is continuous and we 
immediately obtain
\begin{equation}\label{E:sign-convergence}
    \sign(\rho_{\delta_n}(t,\bx)-k)
    \to\sign(\rho_*(t,\bx)-k),
    \quad\text{a.e.\ on }
    ([0,T^*]\times K)\setminus\mathcal{R}_k,
\end{equation}
where we have set
\[
\mathcal{R}_k
\coloneqq
\big\{(t,\bx)\in[0,T^*]\times K:\rho_*(t,\bx)=k\big\}.
\]
The question is, therefore, for which values of $k$ the 
exceptional set $\mathcal{R}_k$ has measure zero, so that 
\cref{E:sign-convergence} holds a.e.\ on the full 
product $[0,T^*]\times K$.

To answer this, consider for each $m\in\NN$ the set
\[
\mathcal{Q}_m
\coloneqq
\big\{k\in\R:\mathrm{meas}(\mathcal{R}_k)\geq\tfrac{1}{m}\big\}.
\]
Since the sets $\{\mathcal{R}_k\}_{k\in\mathcal{Q}_m}$ are 
pairwise disjoint measurable subsets of the bounded domain 
$[0,T^*]\times K$, each of measure at least $1/m$,
the set $\mathcal{Q}_m$ must be finite. Consequently,
\[
\Xi\coloneqq\bigcup_{m\in\NN}\mathcal{Q}_m
\]
is at most countable and thus has Lebesgue measure 
zero. For every $k\notin\Xi$, the set $\mathcal{R}_k$ has 
measure zero, so \cref{E:sign-convergence} holds a.e.\ on 
$[0,T^*]\times K$, that is,
\begin{equation}\label{E:sign-convergence_measurezero}
    \sign(\rho_{\delta_n}(t,\bx)-k)
    \to\sign(\rho_*(t,\bx)-k),
    \quad\text{a.e.\ on }[0,T^*]\times K,
    \quad\forall\,k\notin\Xi.
\end{equation}

It remains to handle the exceptional values $k\in\Xi$.
Fix any such $k$. Since $\Xi$ is countable, we can find 
sequences $k_m\nearrow k$ and $\eta_m\searrow k$ with 
$k_m,\eta_m\notin\Xi$ for all $m$. By definition, at any point 
$(t,\bx)\notin\mathcal{R}_k$, we have $\rho_*(t,\bx)\neq k$,
so the sign function is locally constant near $k$, and we 
conclude that
\begin{align}
    \lim_{m\to\infty}\sign(\rho_*(t,\bx)-k_m)
    &=\sign(\rho_*(t,\bx)-k),
    \quad\text{a.e.\ on }
    ([0,T^*]\times K)\setminus\mathcal{R}_k,
    \label{E:kappa-sequence-from-below}\\
    \lim_{m\to\infty}\sign(\rho_*(t,\bx)-\eta_m)
    &=\sign(\rho_*(t,\bx)-k),
    \quad\text{a.e.\ on }
    ([0,T^*]\times K)\setminus\mathcal{R}_k.
    \label{E:kappa-sequence-from-above}
\end{align}
These two approximations from below and above will allow us 
to pass to the limit for all $k\in\Xi$ in the next step.

By combining \cref{E:sign-convergence_measurezero} with the convergence 
results for the first and second terms, and applying dominated 
convergence (justified by the uniform $\mL^\infty$ bounds on 
$\rho_*$, $\bPhi$, and $\phi$), we may pass to the limit 
$\delta\to 0$ along the subsequence $(\delta_n)$ in 
\cref{E:Kruzkov-rho-delta}. For every $k\notin\Xi$, this yields
\begin{equation}\label{E:Kruzkov-rho-star}
\begin{split}
    &\int_{0}^{T}\!\!\!\int_{\R^n}
        |\rho_*-k|\,\partial_{t}\phi(t,\bx)
        +\sign(\rho_*-k)
        \big(\bPhi(t,\bx,\rho_*)-\bPhi(t,\bx,k)\big)
        \circ\nabla\phi(t,\bx)\,\dd\bx\dd t\\
    &\hspace{1in}
        -\int_{0}^{T}\!\!\!\int_{\R^n}
        \sign(\rho_*-k)\,
        \diverg\bPhi(t,\bx,k)\,\phi(t,\bx)\,\dd\bx\dd t
    \geq 0.
\end{split}
\end{equation}
To bring this into a convenient form, define
\[
L(k)\coloneqq
\int_{0}^{T}\!\!\!\int_{\R^n}
    \sign(\rho_*(t,\bx)-k)\,
    \diverg\bPhi(t,\bx,k)\,\phi(t,\bx)\,\dd\bx\dd t,
\]
and let $\zeta(k)$ denote the sum of the remaining two integrals 
in \cref{E:Kruzkov-rho-star}. Then \cref{E:Kruzkov-rho-star} 
reads \[L(k)\leq\zeta(k),\] for all $k\notin\Xi$. The function 
$\zeta$ is continuous in $k$, which follows from applying dominated 
convergence to each integral, using the regularity of 
$\bPhi$ from \cref{Ass:General} and the $\mL^\infty$ bound on $\rho_*$.

It remains to extend the inequality $L(k)\leq\zeta(k)$ to every 
$k\in\Xi$. Fix $k\in\Xi$ and let $k_m\nearrow k$ and 
$\eta_m\searrow k$ be the sequences constructed above, both 
taking values outside $\Xi$, so that \[L(k_m)\leq\zeta(k_m) \quad \text{ and } \quad L(\eta_m)\leq\zeta(\eta_m)\] for every $m$.

To pass the limit $m\to \infty$, we need to properly adjust the integral domain. More precisely, using the fact that $k_m < k$ for all $m \in \NN$, we have
\begin{equation}\label{E:limit-left}
    \begin{split}
    &\lim_{m \to \infty} \bigg(\iint_{\Omega_T \setminus \mathcal R_{k}}\sign(\rho_*(t,\bx)-k_m)\diverg\bPhi(t,\bx,k_m)\phi(t,\bx)\dd \bx\dd t \\
    & \hspace{1in} + \iint_{\mathcal R_{k}}\sign(\rho_*(t,\bx)-k_m)\diverg\bPhi(t,\bx,k_m)\phi(t,\bx)\dd \bx\dd t \bigg) \\
    & = \iint_{\Omega_T \setminus \mathcal R_{k}}\sign(\rho_*(t,\bx)-k )\diverg\bPhi(t,\bx,k)\phi(t,\bx)\dd \bx\dd t \\
     & \hspace{1in} + \iint_{\mathcal R_{k}}\diverg\bPhi(t,\bx,k)\phi(t,\bx)\dd \bx\dd t \le \zeta(k),
    \end{split}
\end{equation}
where we are using the continuity of $\zeta$.
By the same token, for $\eta_m$, we obtain
\begin{equation}\label{E:right-limit}
    \begin{split}
    &\lim_{m \to \infty} \bigg(\iint_{\Omega_T \setminus \mathcal R_{k}}\sign(\rho_*(t,\bx)-\eta_m)\diverg\bPhi(t,\bx,\eta_m)\phi(t,\bx)\dd \bx\dd t \\
    & \hspace{1in} + \iint_{\mathcal R_{k}}\sign(\rho_*(t,\bx)-\eta_m)\diverg\bPhi(t,\bx,\eta_m)\phi(t,\bx)\dd \bx\dd t\bigg) \\
    & = \iint_{\Omega_T \setminus \mathcal R_{k}}\sign(\rho_*(t,\bx)-k )\diverg\bPhi(t,\bx,k)\phi(t,\bx)\dd \bx\dd t \\
     & \hspace{1in} - \iint_{\mathcal R_{k}}\diverg\bPhi(t,\bx,k)\phi(t,\bx)\dd \bx\dd t \le \zeta(k).
    \end{split}
\end{equation}
If we add \cref{E:limit-left} and \cref{E:right-limit} and divide by \(2\), the $\mathcal{R}_k$ contributions cancel, and we obtain
\begin{equation}\begin{split}
    \tfrac{1}{2} \lim_{m \to \infty} \big(L(k_m) + L(\eta_m)\big)&=  \iint_{\Omega_T \setminus \mathcal R_{k}}\sign(\rho_*(t,\bx)-k )\diverg\bPhi(t,\bx,k)\phi(t,\bx)\dd \bx\dd t \\
    & \le \zeta(k).
\end{split}\end{equation}
On the other hand, over $\mathcal R_{k}$, we have $\sign(\rho_*(t, \bm x) - k) = 0$. Therefore, 
\begin{equation*}
    \tfrac{1}{2} \lim_{m \to \infty} \big( L(k_m) +  L(\eta_m)\big)= \iint_{\Omega_T}\sign(\rho_*(t,\bx)-k )\diverg\bPhi(t,\bx,k)\phi(t,\bx)\dd \bx\dd t = L(k) \le \zeta (k).
\end{equation*}
Putting these together, we can see that \cref{E:Kruzkov-rho-star} is satisfied for any $k \in \R$.

It follows from the proof of \cref{theo:existence_uniqueness_smooth} for the flux function $\bPhi$ that there exists a sufficiently small $T^*$ such that over $[0, T^*]$, 
\begin{equation*}
    \norm{\rho_*(t, \cdot)}_{\mL^\infty(\R^n)} <\infty, 
\end{equation*}
which proves \cref{E:rho_star_bound}. Furthermore, let us recall \cref{E:TV-bound-delta} and in particular for any domain $K \subset \R^n$,
\begin{equation*}
    \abs{\rho_\delta (t, \cdot)}_{\mTV(K)} <\infty, \quad \delta \in \R_{>0},
\end{equation*}
and the bound is uniform in $\delta$. In addition, by \cref{lem:compactness_rho_delta}, we have that $\rho_\delta(t, \cdot) \to \rho_*(t, \cdot)$ in $\mL^1(K)$. Therefore, by the lower semicontinuity of the total variation \cite{leoni2024first},
\[
\abs{\rho_*(t, \cdot)}_{\mTV(K)} \le \liminf_{\delta \to 0} \abs{\rho_\delta(t, \cdot)}_{\mTV(K)}, \quad t \in [0, T^*].
\]
Furthermore, since the result holds for any domain $K$, it holds on $\R^n$. By considering \cref{E:TV-bound-delta} and passing $\delta \to 0$, the \cref{E:rho_star_TVbound} directly follows.

Regarding the time continuity, we consider
\begin{equation*}
    \norm{\rho_*(t, \cdot) - \rho_*(s, \cdot)}_{\mL^1(\R^n)} \le \norm{\rho_*(t , \cdot) - \rho_\delta(t, \cdot)}_{\mL^1(\R^n)} + \norm{\rho_*(s , \cdot) - \rho_\delta(s, \cdot)}_{\mL^1(\R^n)} + \norm{\rho_\delta(t, \cdot) - \rho_\delta(s, \cdot)}_{\mL^1(\R^n)}.
\end{equation*}
The first two terms on the right-hand side vanish as $\delta \to 0$ by \cref{lem:compactness_rho_delta}. Therefore, \cref{E:rho_star-timecontinuity} follows immediately by letting $\delta \to 0$ from \cref{eq:time_continuity-delta}. 

The general stability result of \cref{E:rho_star_general_statbility} follows in a similar way, by considering \cref{E:rho_delta_general_statbility} and letting $\delta \to 0$. Collecting all the results, the proof is completed.
\end{proof}

We now adapt the previous result to a specific choice of the flux $\bPhi$. In particular, let $ w =  w(t,\bx):\Omega_T\to\R$ be a scalar function. We consider 
\[\bPhi(t,\bx,\rho)\coloneqq \bPsi(t,\bx, w(t,\bx),\rho),\]
where $\bPsi:(0,T)\times\R^n\times\R\times\R\rightarrow\R$. This structure will be central to the subsequent investigation of the well-posedness of nonlocal conservation laws. Specifically, we look at the equation 
\begin{equation}
\label{eq:tailored}
\begin{aligned}
        \partial_{t}\rho(t,\bx) + \diverg\big(\bPsi(t,\bx, w(t,\bx),\rho)\big)&=0, && (t,\bx)\in\Omega_{T},\\
        \rho(0,\cdot)&\equiv \rho_{\circ}, &&\text{ on } \R^n.
    \end{aligned}
\end{equation}
We then have the following result.
\begin{corollary}\label{cor:specific_tailored}
    Let $ w\in \mL^{\infty}\big((0,T);\mathcal{W}^{1, \infty}(\R^n;\R)\cap\mathcal{W}^{1, 1}(\R^n;\R)\big)$ be given with $\nabla w\in \mL^{\infty}((0,T);\mTV(\R^n;\R^n))$. Let $\rho_{\circ}\in\mL^{\infty}(\R^n)\cap\mBV(\R^n)$ and $\bPsi:(0,T)\times\R^n\times\R\times\R\rightarrow\R$ satisfy, for all open and bounded $\mU\ \subset\R$,
    \begin{equation}
    \begin{aligned}
        \bPsi&\in\mL^{\infty}((0,T);\mL^{\infty}(\R^n\times\mU^{2})),\\
        \nabla\bPsi,\partial_{3}\bPsi&\in\mL^{\infty}(\Omega_T\times\mU;\mathcal{W}^{1, \infty}(\mU)),\\
\partial_{3}\nabla\bPsi&\in\mL^{\infty}(\Omega_T\times\mU^{2}),\\
\partial_{3}^{2}\bPsi&\in \mL^{\infty}(\Omega_T\times\mU^{2}),
\end{aligned}
\label{eq:assumptions_bPsi_specifically_tailord}
\end{equation}
as well as 
\begin{align}
    \diverg\bPsi\in\mL^{1}(\Omega_T;\mL^{\infty}(\mU^2))\ \text{ or }\
    \diverg\bPsi(\cdot,\ast,0,\star)\equiv 0\label{eq:flux_spatial_dependency_TV}
\end{align}
and
\begin{equation}
\begin{split}
   &\nabla\diverg\bPsi\in\mL^{1}(\Omega_T;\mL^{\infty}(\mU^2)) \,\, or \,\,
    \nabla\diverg\bPsi(\cdot,\ast,0,\star)\equiv 0 \wedge 
    \partial_{3}\nabla\diverg\bPsi\in\mL^{\infty}(\Omega_T\times\mU^{2}).
\end{split}
\label{eq:flux_spatial_dependency_TV_2}
\end{equation}
Then, there exists a unique entropy solution $\rho$, in the sense of \cref{defi:entropy_solution_local}, to the Cauchy problem \eqref{eq:tailored}. Furthermore, the entropy solution satisfies 
\begin{equation}
\label{E:rho_star_w}
\begin{aligned}
\|\rho(t,\cdot)\|_{\mL^{\infty}(\R^n)}
    &\leq 42\|\rho_{\circ}\|_{\mL^{\infty}(\R^n)}\\
|\rho(t,\cdot)|_{\mTV(\R^n)}
    &\leq \abs{\rho_{\circ}}_{\mTV(\R^n)}\,\mathrm{e}^{\kappa t}
    + nW_n\mathrm{e}^{\kappa t}C^{1,\infty},\\
\|\rho(t,\cdot)-\rho(s,\cdot)\|_{\mL^{1}(\R^n)}
    &\leq 
    C^{1,\infty}(t,s)+ |t-s|C^{\infty,\infty}\,
    |\rho|_{\mL^{\infty}((0,T);\mTV(\R^n))},
\end{aligned}
\end{equation}
where the corresponding constants will be defined later in the proof in \cref{eq:C_1_infty,eq:C_infty_infty} and particularly \(C^{1,\infty}(t,s)\) for \((t,s)\in[0,T]^{2}\) vanishes for \(s\rightarrow t\).
\end{corollary}

\begin{proof}
We approximate $ w$ by smooth functions via convolution. 
Specifically, for each $\delta\in\R_{>0}$, set
\[
 w_{\delta}\coloneqq \Theta_{\delta}\ast w\in \mC^{\infty}(\Omega_T),
\]
where $\Theta_\delta$ is the---properly to two dimensions adjusted---mollifier of \cref{defi:bPhi_delta}.
Arguing as in the proof of \cref{lem:properties_mollified_flux}, 
we obtain the convergences 
\[
\lim_{\delta\to 0}
    \| w_{\delta}- w\|_{\mL^{1}((0,T);\mL^{\infty}\cap\mL^{1}(\R^n))}
    = 0,
\qquad
\lim_{\delta\to 0}
    \|\nabla w_{\delta}-\nabla w\|_{\mL^{1}(\Omega_T;\R^n)}
    = 0,
\]
together with the following uniform bounds, valid for all $\delta\in\R_{>0}$:
\begin{equation}
\begin{aligned}
    \| w_{\delta}\|_{\mL^{\infty}(\Omega_T)}
        &\leq \| w\|_{\mL^{\infty}(\Omega_T)}
        \eqqcolon  w_{\infty},\\
    \| w_{\delta}\|_{\mL^{\infty}((0,T);\mL^{1}(\R^n))}
        &\leq \| w\|_{\mL^{\infty}((0,T);\mL^{1}(\R^n))}
        \eqqcolon  w_{1},\\
    \|\nabla w_{\delta}\|_{\mL^{\infty}(\Omega_T)}
        &\leq \|\nabla w\|_{\mL^{\infty}(\Omega_T)}
        \eqqcolon  w'_{\infty},\\
    \|\nabla w_{\delta}\|_{\mL^{\infty}((0,T);\mL^{1}(\R^n))}
        &\leq \|\nabla w\|_{\mL^{\infty}((0,T);\mL^{1}(\R^n))}
        \eqqcolon  w'_{1},\\
    \|\nabla\nabla w_{\delta}\|_{\mL^{\infty}((0,T);\mL^{1}(\R^n;\R^n))}
        &\leq |\nabla w|_{\mL^{\infty}((0,T);\mTV(\R^n;\R^n))}
        \eqqcolon  w'_{\mTV}.
\end{aligned}
\label{eq:omega_delta_bounds_uniform}
\end{equation}
Since $\| w_\delta\|_{\mL^\infty(\Omega_T)}\leq w_\infty$, 
the values of $ w_\delta$ are confined to the interval
\[
\mathcal{I}_ w \coloneqq (- w_\infty, w_\infty)\subset\R.
\]
We also define the approximating flux
\[
\bPhi_{\delta}(t,\bx,\rho)
\coloneqq \bPsi\big(t,\bx, w_{\delta}(t,\bx),\rho\big),
\quad (t,\bx,\rho)\in\Omega_T\times\R,
\]
and verify that $\bPhi_\delta$ satisfies each part of \cref{Ass:General} 
uniformly in $\delta$:
\begin{description}
\item[\cref{E:Lipschitz-data}:]
By the chain rule, \cref{eq:assumptions_bPsi_specifically_tailord}, and \cref{eq:omega_delta_bounds_uniform},
\begin{align}
    \|\bPhi_{\delta}\|_{\mL^{\infty}(\Omega_T;\mL^\infty(\mU))}
        &\leq \|\bPsi\|_{\mL^{\infty}(\Omega_T\times\mathcal{I}_ w\times\mU)},\notag\\
    \|\nabla_{(\bx,\rho)}\bPhi_{\delta}\|_{\mL^{\infty}(\Omega_T;\mL^\infty(\mU))}
        &\leq \|\nabla\bPsi\|_{\mL^{\infty}(\Omega_T\times\mathcal{I}_ w\times\mU)}
        + w'_{\infty}
            \|\partial_{3}\bPsi\|_{\mL^{\infty}(\Omega_T\times\mathcal{I}_ w\times\mU)}
        +\|\partial_{4}\bPsi\|_{\mL^{\infty}(\Omega_T\times\mathcal{I}_ w\times\mU)}\notag\\
        &\eqqcolon C^{\infty,\infty},\label{eq:C_infty_infty}
\end{align}
which is finite and uniform in $\delta$ given the assumptions on $\bPsi$.

\item[\cref{E:D2-Phi-Linf-Lipschitz}:]
Similarly,
\begin{align*}
    \|\nabla\bPhi_{\delta}\|_{\mL^{\infty}(\Omega_T;\mathcal{W}^{1,\infty}(\mU;\R^n))}
        &\leq \|\nabla\bPsi\|_{\mL^{\infty}(\Omega_T\times\mathcal{I}_ w;
            \mathcal{W}^{1,\infty}(\mU;\R^n))}
        + w'_{\infty}
            \|\partial_{3}\bPsi\|_{\mL^{\infty}(\Omega_T\times\mathcal{I}_ w;
            \mathcal{W}^{1,\infty}(\mU;\R^n))}\\
        &\eqqcolon C^{1,\infty,\infty},
\end{align*}
which is again uniform in $\delta$.

\item[\cref{E:div_Phi_L1_uniform}:]
Under the first alternative of \cref{eq:flux_spatial_dependency_TV},
\begin{align*}
    \|\diverg\bPhi_{\delta}\|_{\mL^{1}(\Omega_T;\mL^{\infty}(\mU))}
        &\leq \|\diverg\bPsi\|_{\mL^{1}(\Omega_T;\mL^{\infty}(\mathcal{I}_ w\times\mU))}
        +T w'_{1}
            \|\partial_{3}\bPsi\|_{\mL^{\infty}(\Omega_T\times\mathcal{I}_ w\times\mU)}.
\end{align*}
Under the second alternative (i.e., $\diverg\bPsi(\cdot,\ast,0,\star)\equiv 0$),
\begin{align*}
    \|\diverg\bPhi_{\delta}\|_{\mL^{1}(\Omega_T;\mL^{\infty}(\mU))}
        &\leq T w_{1}
            \|\partial_{3}\diverg\bPsi\|_{\mL^{\infty}(\Omega_T\times\mathcal{I}_ w\times\mU)}
        +T w'_{1}
            \|\partial_{3}\bPsi\|_{\mL^{\infty}(\Omega_T\times\mathcal{I}_ w\times\mU)}.
\end{align*}
Both bounds are finite and uniform in $\delta$
\item[\cref{E:integrable-TV-norm}:]
Under the first alternative of \cref{eq:flux_spatial_dependency_TV_2},
\begin{equation}
\label{eq:C_1_infty}
\begin{aligned}
    \|\nabla\diverg\bPhi_{\delta}\|_{\mL^{1}(\Omega_T;\mL^{\infty}(\mU;\R^n))}
        &\leq \|\nabla\diverg\bPsi\|_{\mL^{1}(\Omega_T;\mL^{\infty}(\mathcal{I}_ w\times\mU))}
        +T w'_{1}
            \|\partial_{3}\nabla\bPsi\|_{\mL^{\infty}(\Omega_T\times\mathcal{I}_ w\times\mU)}\\
        &\quad +T w'_{1} w'_{\infty}
            \|\partial_{3}^{2}\bPsi\|_{\mL^{\infty}(\Omega_T\times\mathcal{I}_ w\times\mU)}
        +T w'_{\mTV}
            \|\partial_{3}\bPsi\|_{\mL^{\infty}(\Omega_T\times\mathcal{I}_ w\times\mU)}\\
        &\eqqcolon C^{1,\infty},
\end{aligned}
\end{equation}
where $C^{1, \infty} = C^{1, \infty}(T)$ which vanishes as $T \to 0$. 

Under the second alternative, the same estimate holds with 
$\|\nabla\diverg\bPsi\|_{\mL^1(\Omega_T;\mL^\infty(\mathcal{I}_ w\times\mU))}$ 
replaced by 
$ w_{1}\|\partial_{3}\nabla\diverg\bPsi\|_{\mL^{\infty}(\Omega_T\times\mathcal{I}_ w\times\mU)}$;
we denote the resulting bound by the same symbol $C^{1,\infty}$.
\end{description}

Having verified all conditions of \cref{Ass:General} uniformly 
in $\delta$, \cref{theo:existence_uniqueness_less_regular_flux} 
provides, for each $\delta\in\R_{>0}$, a unique weak entropy solution 
on some short time horizon $T^*\in(0,T]$,
\[
\rho_{\delta}
\in \mC\big([0,T^*];\mL^{1}_{\mathrm{loc}}(\R^n)\big)
    \cap\mL^{\infty}\big((0,T^*);\mL^{\infty}(\R^n)\cap\mTV(\R^n)\big),
\]
that satisfies, for all $t\in[0,T^*]$, the uniform bounds
\begin{align}
\|\rho_{\delta}(t,\cdot)\|_{\mL^{\infty}(\R^n)}
    &\leq 42\|\rho_{\circ}\|_{\mL^{\infty}(\R^n)}, \label{E:rho_w_delta_bound}\\
|\rho_{\delta}(t,\cdot)|_{\mTV(\R^n)}
    &\leq \abs{\rho_{\circ}}_{\mTV(\R^n)}\,\mathrm{e}^{\eta t}
    + nW_n\mathrm{e}^{\eta t}C^{1,\infty},
    \label{eq:TV_bound_tailored}\\
\|\rho_{\delta}(t,\cdot)-\rho_{\delta}(s,\cdot)\|_{\mL^{1}(\R^n)}
    &\leq C^{1,\infty}(t,s)
    +|t-s|\,C^{\infty,\infty}\,
    |\rho_{\delta}|_{\mL^{\infty}((0,T);\mTV(\R^n))}, \label{E:rho_w_delta_timecontinuity}
\end{align}
where $C^{1, \infty}(t,s)$ vanishes as $t - s\to 0$ and  $\eta\coloneqq(2n+1)C^{\infty,\infty}$.
Since all bounds are uniform in $\delta$, the compactness argument 
of \cref{lem:compactness_rho_delta} applies, and any limit point 
$\rho^*$ is a Kru\v{z}kov entropy solution by the same 
passage-to-the-limit argument as in 
\cref{theo:existence_uniqueness_less_regular_flux}. 

In addition, following the proof of \cref{theo:existence_uniqueness_less_regular_flux}, the same bounds \crefrange{E:rho_w_delta_bound}{E:rho_w_delta_timecontinuity} hold true for $\rho_{*}$ which proves \cref{E:rho_star_general_statbility}.

Furthermore, the general stability result of \cref{E:rho_star_general_statbility} will be satisfied in this case. In particular, for any two 
solutions corresponding to initial data, we have
$\rho_{\circ},\tilde{\rho}_{\circ}\in\mL^{\infty}(\R^n)\cap\mBV(\R^n)$,
\[
\|\rho(t,\cdot)-\tilde{\rho}(t,\cdot)\|_{\mL^{1}(\R^n)}
\leq \mathrm{e}^{(2n+1)C^{1,\infty,\infty}\,t}
    \|\rho_{\circ}-\tilde{\rho}_{\circ}\|_{\mL^{1}(\R^n)},
\]
which provides the uniqueness of $\rho_*$. Collecting all together, the claims follow. 
\end{proof}

\begin{remark}[reduced spatial regularity on $w$ in \cref{cor:specific_tailored}]
Compared with \cref{theo:existence_uniqueness_less_regular_flux}, the structured 
dependence of the flux on $w$ in \cref{cor:specific_tailored} allows us to work 
under weaker spatial regularity. Indeed, it suffices to assume 
$w \in \mL^\infty((0,T);\mathcal{W}^{1,1}(\R^n))$ with 
$\nabla w \in \mL^\infty((0,T);\mBV(\R^n;\R^n))$, rather than requiring 
$\nabla w \in \mL^\infty((0,T);\mathcal{W}^{1,1}(\R^n;\R^n))$ as would be needed 
to apply \cref{theo:existence_uniqueness_less_regular_flux} directly. 
Although this gain may appear marginal in isolation, it turns out to be essential in
\cref{sec:nonlocal_conservation_laws}, where the nonlocal operator $\mathcal{W}$ 
naturally produces functions of this regularity class when the kernel 
$\gamma$ is only assumed to be in $\mBV(\R^n)$.
\end{remark}

\begin{remark}[concrete instances of \cref{cor:specific_tailored}]
\label{rem:specifically_structured_flux}
The assumptions on $\bPsi$ in \cref{cor:specific_tailored} may be difficult to 
verify in concrete situations. We therefore discuss three simplistic and representative cases:

\begin{itemize}

\item \textbf{Lack of explicit space and time dependence.}
Assume $\partial_1\bPsi = 0$ and $\nabla\bPsi = 0$, so that the flux depends 
on $(t,\bx)$ only through $w$. Then there exists a
$\tilde{\bPsi}:\R\times\R\to\R^n$ such that
\begin{equation}
\bPsi(t,\bx,w,\rho) = \tilde{\bPsi}(w,\rho)
\qquad\forall\,(t,\bx,w,\rho)\in\Omega_T\times\R^2,\label{eq:autonomous_flux}
\end{equation}
and the assumptions of \cref{cor:specific_tailored} reduce, for every compact 
$\mU\subset\R$, to
\begin{itemize}
    \item $\tilde{\bPsi}\in\mathcal{W}^{1,\infty}(\mU^2)$,
    \item $\partial_1\tilde{\bPsi}\in\mL^\infty(\mU;\mathcal{W}^{1,\infty}(\mU))$,
    \item $\partial_1^2\tilde{\bPsi}\in\mL^\infty(\mU^2)$.
\end{itemize}
In particular, \cref{eq:flux_spatial_dependency_TV,eq:flux_spatial_dependency_TV_2} 
are satisfied trivially, since the divergence in $\bx$ vanishes identically. 
Any $\tilde{\bPsi}\in\mC^\infty(\R^2;\R^n)$ satisfies these conditions.

\item \textbf{Product structure in the spatial variable.}
Assume
\[
\bPsi(t,\bx,w,\rho) = v(t,\bx)\,\tilde{\bPsi}(w,\rho),
\quad (t,\bx,w,\rho)\in\Omega_T\times\R^2,
\]
with $v:\Omega_T\to\R$ scalar and $\tilde{\bPsi}$ as in the previous case.
The assumptions of \cref{cor:specific_tailored} are satisfied if
\begin{itemize}
    \item $v\in\mL^\infty\big((0,T);\mL^1\cap\mL^\infty(\R^n)\big)$,
    \item $\nabla v\in\mL^1(\Omega_T)$ or $\tilde{\bPsi}(0,\cdot)\equiv 0$,
    \item $\nabla^2 v\in\mL^1(\Omega_T)$ or
          $\nabla^2 v\in\mL^\infty(\Omega_T)$ and $\tilde{\bPsi}(0,\cdot)\equiv 0$.
\end{itemize}
The conditions on $v$ amount to spatial Lipschitz continuity together with 
$\mL^1$-integrability of the second spatial derivatives, or to a cancellation 
condition on $\tilde{\bPsi}$ at zero; neither is particularly restrictive.

\item \textbf{Vector-valued spatial factor.}
This case is somewhat symmetrical to the previous one. Assume
\[
\bPsi(t,\bx,w,\rho) = \bv(t,\bx)\,\tilde{\Psi}(w,\rho),
\quad (t,\bx,w,\rho)\in\Omega_T\times\R^2,
\]
with $\bv:\Omega_T\to\R^n$ and $\tilde{\Psi}:\R^2\to\R$ scalar. Here, the 
spatial vector structure is carried entirely by $\bv$, while $\tilde{\Psi}$ 
encodes the dependence on the local and nonlocal variables. The assumptions 
of \cref{cor:specific_tailored} are satisfied if
\begin{itemize}
    \item $\tilde{\Psi}\in\mathcal{W}^{1,\infty}(\mU^2)$, with 
          $\partial_1\tilde{\Psi}\in\mL^\infty(\mU;\mathcal{W}^{1,\infty}(\mU))$ and 
          $\partial_1^2\tilde{\Psi}\in\mL^\infty(\mU^2)$,
    \item $\bv\in\mL^\infty\big((0,T);\mathcal{W}^{1,\infty}(\R^n;\R^n)\big)$,
    \item $\diverg\bv\in\mL^1(\Omega_T)$ or $\tilde{\Psi}(0,\cdot)=0$,
    \item $\nabla\diverg\bv\in\mL^1(\Omega_T)$ or $\tilde{\Psi}(0,\cdot)=0$.
\end{itemize}
These are mild regularity conditions on both $\bv$ and $\tilde{\Psi}$.
\end{itemize}
\end{remark}


\section{Application to nonlocal nonlinear conservation laws}
\label{sec:nonlocal_conservation_laws}

We now turn to the core part of this work: the analysis of well-posedness in the nonlocal setting. In particular, we focus on the main equations \eqref{E:model_nonlocal} and \eqref{E:model_nonlocal_memory}. For the reader's convenience, we recall the nonlocal operator in the case of \cref{E:model_nonlocal},
\[\mW[J(q),\gamma](t,\bx)=\int_{\R^n}\gamma(\bx-\by)J(q(t,\by))\dd\by, \quad (t,\bx)\in (0,T)\times\R^n,\] 
and in the  case of the model with memory, \cref{E:model_nonlocal_memory}
\[\mW[J(q),\kappa](t,\bx)=\int_{\R_{<t}}\int_{\R^n}\kappa(t-s,\bx-\by)J(q(s,\by))\dd\by\dd s, \quad (t,\bx)\in (0,T)\times\R^n.\] 
Note that, in the latter integral, the solution $q(s,\by)$ takes values for $s<0$ as well. This is standard in the treatment of equations with memory: the solution prior to $t=0$ is prescribed as a \emph{historical 
datum} $q_{\circ}(t,\bx)$ for $(t,\bx)\in\R_{\leq 0} \times \R^n$.

We have the following definition of entropy solution, depending on the model under consideration.

\begin{definition}[entropy solution for nonlinear nonlocal conservation laws]
\label{defi:entropy_solution}
Let $T>0$. A function
\[
q \in \mC\big([0,T];\mL^1_{\loc}(\R^n)\big)
   \cap \mL^\infty\big((0,T);\mL^\infty(\R^n)\big)
\]
is called a Kru\v{z}kov entropy solution if the following conditions hold:
\begin{description}[leftmargin=10pt]

\item[Case of \cref{E:model_nonlocal}:]
For every $\phi \in \mC_c^{1}(\Omega_T;\R_{\ge 0})$ and every $\alpha \in \R$,
\begin{align*}
&\int_0^T \!\!\int_{\R^n}
    |q-\alpha|\,\partial_t \phi
    + \sign(q-\alpha)
      \Big(
        \bF(t,\bx,\mW[J(q),\gamma],q)
        - \bF(t,\bx,\mW[J(q),\gamma],\alpha)
      \Big)
      \cdot \nabla \phi
    \, d\bx\, dt \\
&\quad
- \int_0^T \!\!\int_{\R^n}
    \sign(q-\alpha)\,
    \diverg \bF(t,\bx,\mW[J(q),\gamma],\alpha)
    \,\phi
    \, d\bx\, dt \\
&\quad
- \int_0^T \!\!\int_{\R^n}
    \sign(q-\alpha)\,
    \partial_3 \bF(t,\bx,\mW[J(q),\gamma],\alpha)
    \circ \nabla \mW[J(q),\gamma]
    \,\phi
    \, d\bx\, dt
\;\ge 0 .
\end{align*}
Here, we have omitted the explicit dependence of $q, \mW[J(q),\gamma]$, and $\phi$ on $(t,\bx)$ to simplify the notation.
Furthermore, the initial condition
\[
q(0,\cdot) \equiv q_\circ \quad \text{on } \R^n
\]
is satisfied, for a given initial datum $q_\circ \in \mL^\infty(\R^n)$.

\item[Case of \cref{E:model_nonlocal_memory}:]
For every $\phi \in \mC_c^{1}(\Omega_T;\R_{\ge 0})$ and every $\alpha \in \R$,
\begin{align*}
&\int_0^T \!\!\int_{\R^n}
    |q-\alpha|\,\partial_t \phi
    + \sign(q-\alpha)
      \Big(
        \bF(t,\bx,\mW[J(q),\kappa],q)
        - \bF(t,\bx,\mW[J(q),\kappa],\alpha)
      \Big)
      \cdot \nabla \phi
    \, d\bx\, dt \\
&\quad
- \int_0^T \!\!\int_{\R^n}
    \sign(q-\alpha)\,
    \diverg \bF(t,\bx,\mW[J(q),\kappa],\alpha)
    \,\phi
    \, d\bx\, dt \\
&\quad
- \int_0^T \!\!\int_{\R^n}
    \sign(q-\alpha)\,
    \partial_3 \bF(t,\bx,\mW[J(q),\kappa],\alpha)
    \circ \nabla \mW[J(q),\kappa]
    \,\phi
    \, d\bx\, dt
\;\ge 0 .
\end{align*}
Here, we have omitted the explicit dependence of $q, \mW[J(q),\gamma]$, and $\phi$ on $(t,\bx)$ to simplify the notation.
Furthermore, the historical datum is prescribed by
\[
q(t,\bx) \equiv q_\circ(t,\bx),
\qquad (t,\bx)\in \R_{\le 0}\times\R^n,
\]
for a given function
\[
q_\circ \in \mL^\infty(\R_{<0};\mL^\infty(\R^n)).
\]

\end{description}
\end{definition}

In both cases, we will construct solutions employing the well-posedness and stability results obtained in the previous section. In particular, we consider the conservation law with a given space- and time-dependent term $w$,
\begin{equation}
    \begin{aligned}
    \partial_{t}q &=-\diverg\big(\bbF\big(t,\bx,w(t,\bx), q\big)\big), &&(t, \bx) \in (0, T) \times \R^n,\\
    q(0,\cdot)&=
    \begin{cases}
        q_{\circ}, &\text{ on }\R^n \qquad\text{ (nonlocal in space)}, \\
        q_{\circ}(0,\cdot) &\text{ on }\R^n \qquad\text{ (nonlocal in space and time; memory)}.
    \end{cases}
    \end{aligned}
    \label{eq:nonlocal_explicit_w}
\end{equation}
Observe that, in the memory case, the initial condition at $t=0$ is simply the trace of the historical 
datum $q_{\circ}(t,\bx)$, prescribed for all $(t,\bx)\in\R_{\leq 0}\times\R^n$, rather 
than a standalone datum as in the spatially nonlocal case.

A key ingredient of the fixed-point 
argument developed in each subsection below is the solution operator associated 
with \cref{eq:nonlocal_explicit_w}, which maps a given $w$ to the corresponding entropy solution. 
To introduce it, we first state our standing assumptions on the flux $\bF$.

\begin{assumption}[the flux \(\bF\)]
\label{ass:flux_nonlocal}
    Let $\mU\subset\R$ be an open and bounded set. We assume that the flux $\bF$ satisfies the following regularity properties:
        \begin{enumerate}
          \item $\bF\in\mL^{\infty}((0,T);\mL^{\infty}(\R^n\times\mU^{2}))$,\label{item:1_flux}
          \item $\nabla\bF\in\mL^{\infty}((0,T)\times\R^n\times\mU;\mathcal{W}^{1,\infty}(\mU))$,\label{item:2_flux}
          \item $\partial_{3}\bF\in\mL^{\infty}((0,T)\times\R^n\times\mU;\mathcal{W}^{1,\infty}(\mU))$,\label{item:3_flux}
          \item $\partial_{3}\nabla\bF\in\mL^{\infty}((0,T)\times\R^n\times\mU^{2})$,\label{item:4_flux}
          \item $\partial_{3}^{2}\bF\in\mL^{\infty}((0,T)\times\R^n\times\mU^{2})$,\label{item:5_flux}
          \item $\diverg\bF\in\mL^{1}((0,T)\times\R^n;\mL^{\infty}(\mU^{2}))\ or\ 
          \diverg\bF(\cdot,\ast,0,\star)\equiv 0$,\label{item:6_flux}
          \item $\nabla\diverg\bF\in\mL^{1}((0,T)\times\R^n;\mL^{\infty}(\mU^{2}))$
          or
          $\nabla\diverg\bF(\cdot,\ast,0,\star)\equiv 0 \wedge
          \partial_{3}\nabla\diverg\bF\in\mL^{\infty}((0,T)\times\R^n\times\mU^{2})$.
          \label{item:7_flux}
        \end{enumerate}
\end{assumption}

We can now define the solution operator.

\begin{definition}[solution operator]\label{defi:solution_operator_fixed_w}
    Let \cref{ass:flux_nonlocal} hold and let $w\in \mL^{\infty}\big((0,T);\mathcal{W}^{1, \infty}(\R^n;\R)\big)$ be given. Consider the conservation law \eqref{eq:nonlocal_explicit_w}. Let \[\mZ(T)\coloneqq\Big\{w\in\mL^{\infty}\big((0,T);\mathcal{W}^{1, \infty}\cap\mathcal{W}^{1, 1}(\R^n;\R)\big): |\nabla w|_{\mL^{\infty}((0,T);\mTV(\R^n;\R^n))}<\infty\Big\}.\] 
    Then, we define the solution operator 
    \[
    \mQ:\begin{cases}\mZ(T)
    &\rightarrow \mC\big([0,T];\mL^{1}(\R^n)\big)\cap \mL^{\infty}((0,T);\mL^{\infty}(\R^n)\cap \mTV(\R^n)),\ \\
    w&\mapsto \text{entropy solution of \cref{eq:nonlocal_explicit_w}},
    \end{cases}
    \]
    which maps $w$ to the corresponding weak entropy solution of \eqref{eq:nonlocal_explicit_w}.
\end{definition}

\begin{lemma}[well-definedness of $\mQ$]\label{T:Q-welldefined} \label{lem:well-posed-Q}
Let \cref{ass:general_nonlocal} hold. Then the solution operator $\mQ$ in 
\cref{defi:solution_operator_fixed_w} is well-defined; that is, for every $w \in \mZ(T)$, 
there exists a unique weak entropy solution of \cref{eq:nonlocal_explicit_w} in the 
sense of \cref{defi:entropy_solution_local}.
\end{lemma}
\begin{proof}
For any $w \in \mZ(T)$, the flux $\bPhi(t,\bx,q) \coloneqq \bF(t,\bx,w(t,\bx),q)$ 
satisfies the assumptions of \cref{cor:specific_tailored}: the regularity of $w$ matches 
that required of $\omega$, and the assumptions on $\bF$ in \cref{ass:general_nonlocal} 
coincide with \crefrange{eq:assumptions_bPsi_specifically_tailord}{eq:flux_spatial_dependency_TV_2}. 
The conclusion follows directly from \cref{cor:specific_tailored}.
\end{proof}

We are now equipped with all the basic tools needed to address the well-posedness of the two nonlocal models. We separate our analysis into two cases.

\subsection{Nonlocal in space}

We begin by considering the spatially nonlocal nonlinear conservation law \cref{E:model_nonlocal}. In addition to \cref{ass:flux_nonlocal}, we make the following requirements on the initial datum and nonlocal kernel.

\begin{assumption}[nonlocal kernel, nonlocal nonlinearity and initial datum]
\label{ass:general_nonlocal}
    We assume the following:
    \begin{description}
            \item[Nonlocal nonlinearity:] \ $J\in \mathcal{W}^{1, \infty}_{\textnormal{loc}}(\R),\quad J(0)=0$.
            \item[Kernel:] $\gamma\in \mBV(\R^n)\cap\mL^{\infty}(\R^n): \|\gamma\|_{\mL^{1}(\R^n)}=1$.
        \item[Initial datum:] $q_{\circ}\in \mBV(\R^n)\cap \mL^{\infty}(\R^n)$.
    \end{description}
\end{assumption}

\begin{remark}[required regularity on the datum]
Note that the assumptions on the flux $\bF$ in \cref{ass:flux_nonlocal} are tailored exactly to fit the requirements in \cref{Ass:General} and are identical to what was postulated in \cref{cor:specific_tailored}. Furthermore, the conditions on the nonlocal kernel and the nonlinearity are standard. The assumption of $q_{\circ}\in \mBV(\R^n)$, however, entails in particular $\mL^{1}(\R^n)$, which we require for uniform $\mTV$ bounds in the nonlocal equation. It seems that this requirement can only be avoided when postulating significantly stronger conditions on the flux function $\bF$, which we chose not to do.
 An illustration of how ``weak'' the assumptions on $\bF$ are was outlined in \cref{rem:specifically_structured_flux}.
\end{remark}

The first step is to make precise how $\mW$ inherits regularity from $q$.

\begin{lemma}[bounds on the nonlocal operator]
\label{lem:bounds_nonlocal_space}
Let \cref{ass:general_nonlocal} hold and let $T\in\R_{>0}$. Define
\begin{equation}\label{E:solution-space}
\mX \coloneqq \mC\big([0,T];\mL^{1}(\R^n)\big)\cap \mL^{\infty}\big((0,T);\mL^{\infty}(\R^n)\cap \mTV(\R^n)\big).
\end{equation}
Then, the nonlocal operator
\[
\mW:\mX \times \mBV(\R^n) \rightarrow \mC\big([0,T];\mL^{1}(\R^n)\big),
\]
defined by
\begin{equation}
\mW[J(q),\gamma](t,\bx) \coloneqq \int_{\R^n}\gamma(\bx-\by)J(q(t,\by))\dd\by,
\quad (t,\bx)\in (0,T)\times \R^n,
\label{defi:nonlocal_operator}
\end{equation}
satisfies the following bounds for every $t\in[0,T]$:
\begin{align}
\|\mW[J(q),\gamma](t,\cdot)\|_{\mL^{1}(\R^n)}
    &\leq \|J'\|_{\mL^{\infty}(Q(t))}\|q(t,\cdot)\|_{\mL^{1}(\R^n)},
    \label{E:WL1}\\
\|\mW[J(q),\gamma](t,\cdot)\|_{\mL^{\infty}(\R^n)}
    &\leq \|J'\|_{\mL^{\infty}(Q(t))}\|q(t,\cdot)\|_{\mL^{\infty}(\R^n)},
    \label{E:WLinfty}\\
\|\nabla \mW[J(q),\gamma](t,\cdot)\|_{\mL^{1}(\R^n;\R^n)}
    &\leq |\gamma|_{\mTV(\R^n)}\|J'\|_{\mL^{\infty}(Q(t))}\|q(t,\cdot)\|_{\mL^{1}(\R^n)},
    \label{E:DWL1}\\
\|\nabla \mW[J(q),\gamma](t,\cdot)\|_{\mL^{\infty}(\R^n;\R^n)}
    &\leq |\gamma|_{\mTV(\R^n)}\|J'\|_{\mL^{\infty}(Q(t))}\|q(t,\cdot)\|_{\mL^{\infty}(\R^n)},
    \label{E:DWLinfty}\\
|\nabla \mW[J(q),\gamma](t,\cdot)|_{\mTV(\R^n;\R^{n\times n})}
    &\leq |\gamma|_{\mTV(\R^n)}\|J'\|_{\mL^{\infty}(Q(t))}|q(t,\cdot)|_{\mTV(\R^n)},
    \label{E:DWTV}
\end{align}
where $Q(t)\coloneqq \big( -\|q(t,\cdot)\|_{\mL^{\infty}(\R^n)},\|q(t,\cdot)\|_{\mL^{\infty}(\R^n)}\big)$.
\end{lemma}

\begin{proof}
Inequalities \eqref{E:WL1} and \eqref{E:WLinfty} are an immediate consequence of \cref{ass:general_nonlocal} on $\gamma$ and $J$. To prove \eqref{E:DWL1} and \eqref{E:DWLinfty}, let $\bv\in\R^n$ be any unit-norm vector. Then, for any $h>0$ small enough we have
\begin{equation}\label{E:W-Lipschitz}
\begin{split}
&\abs[\big]{\mW[J(q), \gamma](t, \bx + h\bv)-\mW[J(q, \gamma](t, \bx))}\\
&\leq\Big| \int_{\R}\!\!\gamma(\bx+h\bv-\by)J(q(t,\by))\dd y-\!\! \int_{\R}\!\!\gamma(\bx-\by)J(q(t,\by))\dd y\Big|\\
&\leq  \int_{\R}|\gamma(\bx+h\bv-\by)-\gamma(\bx-\by)||J(q(t,\by))|\dd y\\
&\leq \|J(q(t,\cdot))\|_{\mL^{\infty}(\R^n)} \int_{\R}|\gamma(\bx+h\bv-\by)-\gamma(\bx-\by)|\dd y\\
&\leq \|J(q(t,\cdot))\|_{\mL^{\infty}(\R)} \abs{h} |\gamma|_{\mTV(\R^n)}\\
&\leq \|J'\|_{\mL^{\infty}(Q(t))}\|q(t,\cdot)\|_{\mL^{\infty}(\R^n)}\abs{h} |\gamma|_{\mTV(\R^n)},
\end{split}\end{equation}
and dividing by $\abs{h}$ and taking the limit as $h \to 0$, we obtain 
\begin{equation*}
    \mW[J(q), \gamma](t, \cdot) \in \mathcal W^{1, \infty}(\R^n; \R), \quad t \in [0, T],
\end{equation*}
proving \eqref{E:DWLinfty}. Next, keeping \cref{E:W-Lipschitz} in mind, we can calculate
\begin{equation*}
    \begin{split}
        \int_{\R^n} |\nabla \mW[J(q), \gamma](t, \bm x) |\dd \bm x& \le \|J(q(t,\cdot))\|_{L^{1}(\R^n)} |\gamma|_{\mTV(\R^n)}\\
        &\leq \|J'\|_{\mL^{\infty}(Q(t))}\|q(t,\cdot)\|_{L^{1}(\R^n)}  |\gamma|_{\mTV(\R^n)},
    \end{split}
\end{equation*}
which proves \cref{E:DWL1}. Finally, for \cref{E:DWTV}, we consider \(\bv,\bv'\in\R^n,\) of unitary norm. We are interested in bounding
\begin{align*}
       \mathbb{I}_{\nabla \mW} \coloneqq & \tfrac{1}{\abs{h}}\int_{\R^n} \abs[\big]{\nabla\mW[J(q), \gamma](t, \bm x + \bv h) - \nabla \mW[J(q), \gamma](t,\bm x)} \dd \bx  \\
        & \quad =\tfrac{1}{\abs{h}} \lim_{h'\to 0} \bigg(\int_{\R^n} \abs[\Big]{\tfrac{1}{\abs{h'}} \abs[\big]{\mW[J(q), \gamma](t, \bm x + h\bv + h'\bv') - \mW[J(q), \gamma](t, \bm x + h\bv)}\\
        &\qquad\qquad\qquad- \tfrac{1}{\abs{h'}} \abs[\big]{\mW[J(q), \gamma](t, \bm x+ h'\bv') - \mW[J(q), \gamma](t,\bm x)}} \dd \bx \bigg) .
\end{align*}
To this end, we observe that
\begin{equation*}
    \begin{split}
       &\mW[J(q), \gamma](t, \bm x + h \bm v + h' \bm v') -\mW[J(q), \gamma](t, \bm x + h \bm v) \\
        & =  \int_{\R^n}\gamma(\bm y) \set[\big]{J(q(t, \bm x + h \bm v + h' \bm v'- \bm y)) -  J(q(t, \bm x + h \bm v - \bm y))} \dd \bm y.
    \end{split}
\end{equation*}
This implies, in particular, that
\begin{equation*}
    \begin{split}
        \mathbb I_{\nabla \mW} & = \tfrac{1}{\abs{h}} \int_{\R^n} \lim_{h' \to 0} \tfrac{1}{\abs{h'}}\set[\bigg]{\int_{\R^n} \abs{J(q(t , \bm x + h \bm v + h' \bm v' - \bm y)) - J(q(t, \bm x + h \bm v - \bm y))} \abs{\gamma(\bm y)}\dd \bm y \\
        & \hspace{1.5in} \int_{\R^n} \abs{J(q(t, \bm x + h' \bm v' - \bm y)) - J(q(t, \bm x  - \bm y))} \abs{\gamma(\bm y)}\dd \bm y} \dd \bm x. \\
        \intertext{By the change of variables $\bm z = \bm x + h \bm v - \bm y$ and $\zeta = \bm x - \bm y$, we have } \\
         &=  \tfrac{1}{\abs{h}} \int_{\R^n} \lim_{h' \to 0} \tfrac{1}{\abs{h'}}\set[\bigg]{\int_{\R^n} \abs{J(q(t , \bm z +h' \bm v')) - J(q(t , \bm z))} \abs{\gamma(\bm x + h \bm v - \bm z)}\dd \bm z \\
        & \hspace{1.5in} -  \int_{\R^n} \abs{J(q(\tau , \bm z + h' \bm v')) - J(q(\tau , \bm z))} \abs{\gamma(\bm x - \zeta)}\dd \zeta} \dd \bm x .
\end{split}
\end{equation*}
Applying Fubini's theorem to $\dd \bm z \dd \bm x = \dd \bm x \dd \bm z$, we obtain
\begin{equation*}
\begin{split}
        \mathbb I_{\nabla \mW} & \le \lim_{h' \to 0} \tfrac{1}{\abs{h'}}\tfrac{1}{\abs{h}} \norm{J'}_{\mL^\infty(Q(T))} \int_{\R^n} \abs{q(t, \bm z + h' \bm v') - q(t, \bm z)} \dd \bm z \\
        & \hspace{2in} \times \int_{\R^n} \abs{\gamma(t, \bm x + h \bm v - \bm z) - \gamma(\bm x - \bm z)} \dd \bm x  \dd \bm z\\
        & \le \norm{J'}_{\mL^\infty(Q(T))} \norm{q(t, \cdot)}_{\mTV(\R^n))}  \norm{\gamma}_{\mTV(\R^n))}.
    \end{split}
\end{equation*}
This completes the proof.
\end{proof}

We now formulate the well-posedness of \cref{E:model_nonlocal} as a fixed-point problem. 
To this end, we introduce the admissible set and the fixed-point mapping.

\begin{definition}[admissible set $\mY(T)$]\label{defi:admissible_set}
Let \cref{ass:general_nonlocal} hold, let $T \in \R_{>0}$, and set
\[
B^{\infty} \coloneqq 42\|q_{\circ}\|_{\mL^{\infty}(\R^n)}, \qquad 
B^{\mTV} \coloneqq 42|q_{\circ}|_{\mTV(\R^n)}.
\]
Recalling \cref{E:solution-space}, we define the admissible set
\[
\mY(T) \coloneqq \left\{ q \in \mathcal X \;:\; 
\begin{array}{l}
    \|q\|_{\mL^{\infty}((0,T);\mL^{\infty}(\R^n))} \leq B^{\infty},\\[4pt]
    |q|_{\mL^{\infty}((0,T);\mTV(\R^n))} \leq B^{\mTV},\\[4pt]
    \|q(t,\cdot)\|_{\mL^{1}(\R^n)} \leq 42 \|q_{\circ}\|_{\mL^{1}(\R^n)} 
    \quad \forall\, t\in[0,T]
\end{array}
\right\}.
\]
\end{definition}

\begin{definition}[fixed-point mapping $\mF$]\label{defi:fixed_point}
Let \cref{ass:general_nonlocal} hold, let $T \in \R_{>0}$, and suppose that $\mQ$, $\mW$, and $\mY(T)$ 
are as in \cref{defi:solution_operator_fixed_w}, \cref{defi:nonlocal_operator}, and 
\cref{defi:admissible_set}, respectively. We define the fixed-point mapping
\[
\mF : \mY(T) \rightarrow \mathcal X, 
\qquad
\mF[q] \coloneqq \mQ\big[\mW[J(q),\gamma]\big],
\]
where, for each $q \in \mY(T)$, the function $\mW[J(q),\gamma]\in\mZ(T)$ is the nonlocal 
term evaluated along $q$.
\end{definition}

\begin{lemma}[self-mapping property of $\mF$]\label{lem:self_mapping}
Let \cref{ass:general_nonlocal} hold, and let $\mF$ and $\mY(T)$ be as in 
\cref{defi:fixed_point} and \cref{defi:admissible_set}. Then there exists a $T \in \R_{>0}$ sufficiently small such that
\[
\mF\big[\mY(T)\big] \subset \mY(T).
\]
\end{lemma}

\begin{proof}
For conciseness, we define the intervals
\begin{equation}
\label{intervals}
\mI_\rho(B^\infty) \coloneqq (-B^\infty, B^\infty) \subset \R, \qquad
\mI_{\omega,\rho}(B^\infty) \coloneqq \big(-\|J'\|_{\mL^\infty(\mI_\rho)}B^\infty,\,\|J'\|_{\mL^\infty(\mI_\rho)}B^\infty\big) \times \mI_\rho \subset \R^2.
\end{equation}
Additionally, we omit the dependence on $B^\infty$ throughout the proof. We also recall that $\Omega_T \coloneqq (0, T) \times \R^n$. Let $T\in\R_{>0}$ be fixed and take $\tilde{q}\in\mY(T)$. Considering \crefrange{E:WL1}{E:DWTV}, 
the nonlocal term satisfies
\begin{equation}
\begin{aligned}
    \|\mW[J(\tilde{q}),\gamma]\|_{\mL^{\infty}((0,T);\mL^{1}(\R^n))}
        &\leq \|J'\|_{\mL^{\infty}(\mI_\rho)}\|q_{\circ}\|_{\mL^{1}(\R^n)},\\
    \|\mW[J(\tilde{q}),\gamma]\|_{\mL^{\infty}(\Omega_T)}
        &\leq \|J'\|_{\mL^{\infty}(\mI_\rho)}B^{\infty},\\
    \|\nabla\mW[J(\tilde{q}),\gamma]\|_{\mL^{\infty}((0,T);\mL^{1}(\R^n))}
        &\leq |\gamma|_{\mTV(\R^n)}\|J'\|_{\mL^{\infty}(\mI_\rho)}\|q_{\circ}\|_{\mL^{1}(\R^n)},\\
    \|\nabla\mW[J(\tilde{q}),\gamma]\|_{\mL^{\infty}(\Omega_T)}
        &\leq |\gamma|_{\mTV(\R^n)}\|J'\|_{\mL^{\infty}(\mI_\rho)}B^{\infty},\\
    |\nabla\mW[J(\tilde{q}),\gamma]|_{\mL^{\infty}((0,T);\mTV(\R^n))}
        &\leq |\gamma|_{\mTV(\R^n)}\|J'\|_{\mL^{\infty}(\mI_\rho)}B^{\mTV}.
\end{aligned}
\label{eq:nonlocal_estimates_fixed_point}
\end{equation}
Thus, by \cref{T:Q-welldefined}, there exists a unique weak entropy solution 
$q\coloneqq\mF[\tilde{q}] \in \mathcal X$ of
\begin{align*}
    \partial_{t}q+\diverg_{\bx}\big(\bF(t,\bx,\mW[J(\tilde{q}),\gamma](t,\bx),q)\big)
        &=0, && (t,\bx)\in\Omega_T,\\
    q(0,\cdot)&=q_{\circ}, && \text{on }\R^n.
\end{align*}
It remains to show that $q\in \mY(T)$ for $T$ small enough, which we establish in 
three steps:
\begin{description}
\item[The $\mL^{\infty}$ bound:]
Set $\bPhi(t,\bx,q)\coloneqq \bF\big(t,\bx,\mW[J(\tilde{q}),\gamma](t,\bx),q\big)$ 
for $(t,\bx,q)\in\Omega_T\times\R$. Proceeding as in \cref{theo:existence_uniqueness_smooth}, 
we note that the $\mL^\infty$-norm of $q \coloneqq \mF[\tilde{q}]$ satisfies the differential inequality
\begin{align*}
\partial_{t}\|q(t,\cdot)\|_{\mL^{\infty}(\R^n)}
    &\leq\esssup_{\bz\in\R^n}\Big|\diverg\bF\big(t,\bz,\mW[J(\tilde{q}),\gamma](t,\bz),
        \|q(t,\cdot)\|_{\mL^{\infty}(\R^n)}\big)\Big|\\
    &\quad +\esssup_{\bz\in\R^n}\Big|\partial_{3}\bF\big(t,\bz,\mW[J(\tilde{q}),\gamma](t,\bz),
        \|q(t,\cdot)\|_{\mL^{\infty}(\R^n)}\big)
        \cdot\nabla\mW[J(\tilde{q}),\gamma](t,\bz)\Big|.
\end{align*}
Considering \cref{E:rho_star_w} and \cref{eq:nonlocal_estimates_fixed_point}, we have that
\[\big(\mW[J(\tilde{q}),\gamma](t,\bx),\|q(t,\cdot)\|_{\mL^\infty}\big) \in \mI_{\omega,\rho},\] 
and hence
\begin{align*}
\partial_{t}\|q(t,\cdot)\|_{\mL^{\infty}(\R^n)}
    &\leq \|\diverg\bF(t,\cdot,\ast,\star)\|_{\mL^{\infty}(\R^n\times\mI_{\omega,\rho})}
    + \|\partial_{3}\bF(t,\cdot,\ast,\star)\|_{\mL^{\infty}(\R^n\times\mI_{\omega,\rho})}
      |\gamma|_{\mTV(\R^n)}\|J'\|_{\mL^{\infty}(\mI_\rho)}B^{\infty}.
\end{align*}
Integrating in time from $0$ to $t$ gives
\begin{align*}
\|q(t,\cdot)&\|_{\mL^{\infty}(\R^n)} \\
    &\leq \|q_{\circ}\|_{\mL^{\infty}(\R^n)}
    + t\Big(\|\diverg\bF\|_{\mL^{\infty}(\Omega_T\times\R^n\times\mI_{\omega,\rho})}
    + \|\partial_{3}\bF\|_{\mL^{\infty}(\Omega_T\times\R^n\times\mI_{\omega,\rho})}
      |\gamma|_{\mTV(\R^n)}\|J'\|_{\mL^{\infty}(\mI_\rho)}B^{\infty}\Big).
\end{align*}
Defining
\[
T_{1}\coloneqq\frac{41\|q_\circ\|_{\mL^\infty(\R^n)}}{
    \|\diverg\bF\|_{\mL^{\infty}(\Omega_T\times\R^n\times\mI_{\omega,\rho})}
    +\|\partial_{3}\bF\|_{\mL^{\infty}(\Omega_T\times\R^n\times\mI_{\omega,\rho})}
    |\gamma|_{\mTV(\R^n)}\|J'\|_{\mL^{\infty}(\mI_\rho)}B^{\infty}},
\]
which is finite by \cref{item:2_flux,item:3_flux} of \cref{ass:flux_nonlocal}, implies that for all $t\in[0,T_{1}]$,
\[
\|q(t,\cdot)\|_{\mL^{\infty}(\R^n)}
    \leq  42\|q_\circ\|_{\mL^\infty(\R^n)} = B^{\infty}.
\]
which confirms the a priori assumption and establishes the claimed $\mL^\infty$ bound on $[0,T_1]$.

\item[The $\mTV$ bound:]
Next, we establish a $\mTV$ bound on $q = \mF[\tilde{q}]$ for small time horizons. Applying the $\mTV$ bound from \cref{cor:specific_tailored} to $q$ 
and substituting the estimates \cref{eq:nonlocal_estimates_fixed_point} for the nonlocal 
term, we obtain from \cref{eq:TV_bound_tailored} that, for all $t\in[0,T_{1}]$,
\begin{align}
    |\mF[\tilde{q}](t,\cdot)|_{\mTV(\R^n)}
    \leq \abs{q_{\circ}}_{\mTV(\R^n)} \mathrm{e}^{\eta(B^{\infty}) t} 
    + n W_n \mathrm{e}^{\eta(B^{\infty}) t}C\big(t,B^{\infty},B^{\mTV}\big),
\label{eq:TV_bound_fixed_point}
\end{align}
where $\eta(B^\infty)$ is given by
\begin{equation}\label{E:eta-B}
\begin{split}
    \eta(B^{\infty})\coloneqq (2n+1)\Big(
        &\|\nabla\bF\|_{\mL^{\infty}(\Omega_T\times\R^n\times\mI_{\omega,\rho})}
        +\|\partial_{4}\bF\|_{\mL^{\infty}(\Omega_T\times\R^n\times\mI_{\omega,\rho})}\\
        &+ |\gamma|_{\mTV(\R^n)}\|J'\|_{\mL^{\infty}(\mI_\rho)}B^{\infty}
        \|\partial_{3}\bF\|_{\mL^{\infty}(\Omega_T\times\R^n\times\mI_{\omega,\rho})}
    \Big),
\end{split}
\end{equation}
which is finite by \cref{item:2_flux,item:3_flux} of \cref{ass:flux_nonlocal}. Here, $C\big(t,B^{\infty},B^{\mTV}\big)$ 
is given by
\begin{equation}
\begin{aligned}
C\big(t,B^{\infty},B^{\mTV}\big)\coloneqq\,
    &\|\nabla\diverg\bF\|_{\mL^{1}(\Omega_{T};\mL^{\infty}(\mI_{\omega,\rho}))}\\
    &+t\|J'\|_{\mL^{\infty}(\mI_\rho)}\Big(
        |\gamma|_{\mTV(\R^n)}\|q_{\circ}\|_{\mL^{1}(\R^n)}
        \|\partial_{3}\diverg\bF\|_{\mL^{\infty}(\Omega_T\times\R^n\times\mI_{\omega,\rho})}\\
    &\quad +|\gamma|_{\mTV(\R^n)}^{2}\|q_{\circ}\|_{\mL^{1}(\R^n)}
        \|J'\|_{\mL^{\infty}(\mI_\rho)}B^{\infty}
        \|\partial_{3}^{2}\bF\|_{\mL^{\infty}(\Omega_T\times\R^n\times\mI_{\omega,\rho})}\\
    &\quad +|\gamma|_{\mTV(\R^n)}B^{\mTV}
        \|\partial_{3}\bF\|_{\mL^{\infty}(\Omega_T\times\R^n\times\mI_{\omega,\rho})}
    \Big)
\end{aligned}
\label{eq:C_t_B_infty_TV_1}
\end{equation}
under the first part of \cref{item:7_flux} or
\begin{equation}
\begin{aligned}
C\big(t,B^{\infty},B^{\mTV}\big)\coloneqq\,
    &t\|J'\|_{\mL^{\infty}(\mI_\rho)}\Big(
        \|q_{\circ}\|_{\mL^{1}(\R^n)}
        \|\partial_{3}\nabla\diverg\bF\|_{\mL^{\infty}(\Omega_T\times\R^n\times\mI_{\omega,\rho})}\\
    &\quad +|\gamma|_{\mTV(\R^n)}\|q_{\circ}\|_{\mL^{1}(\R^n)}
        \|\partial_{3}\diverg\bF\|_{\mL^{\infty}(\Omega_T\times\R^n\times\mI_{\omega,\rho})}\\
    &\quad +|\gamma|_{\mTV(\R^n)}^{2}\|q_{\circ}\|_{\mL^{1}(\R^n)}
        \|J'\|_{\mL^{\infty}(\mI_\rho)}B^{\infty}
        \|\partial_{3}^{2}\bF\|_{\mL^{\infty}(\Omega_T\times\R^n\times\mI_{\omega,\rho})}\\
    &\quad +|\gamma|_{\mTV(\R^n)}B^{\mTV}
        \|\partial_{3}\bF\|_{\mL^{\infty}(\Omega_T\times\R^n\times\mI_{\omega,\rho})}
    \Big)
\end{aligned}
\label{eq:C_t_B_infty_TV_2}
\end{equation}
under the second part of \cref{item:7_flux}. In both cases, $C\big(t,B^{\infty},B^{\mTV}\big)$ 
is finite by \cref{item:3_flux,item:4_flux,item:5_flux,item:7_flux}. Since 
$C\big(t,B^{\infty},B^{\mTV}\big)\to 0$ as $t\to 0$, we may choose $T_{2}\in(0,T_{1}]$ sufficiently
small such that for all $t\in[0,T_{2}]$,
\[
\abs{q_{\circ}}_{\mTV(\R^n)} \mathrm{e}^{\eta(B^{\infty}) t} 
+ n W_n \mathrm{e}^{\eta(B^{\infty}) t}C\big(t,B^{\infty},B^{\mTV}\big)
\leq 42|q_{\circ}|_{\mTV(\R^n)} = B^{\mTV},
\]
which gives $|\mF[\tilde{q}](t,\cdot)|_{\mTV(\R^n)}\leq B^{\mTV}$ for all $t\in[0,T_{2}]$.

\item[The \(\mL^{1}\)-bound:] This is a consequence of the fact that we are dealing with a conservation law. In particular, let us consider the $\mL^1$-time continuity result of \cref{theo:existence_uniqueness_less_regular_flux}. 
Letting $\bPhi(t, \bx, q) = \bF(t, \bx, w(t, \bx), q)$, we have
\begin{equation}\label{E:conservation-bound}
    \begin{split}
        \|\mF[\tilde q](t,\cdot)\|_{\mL^1(\R^n)}&\leq \norm{q_0}_{\mL^1(\R^n)} + \|\diverg\bF\|_{\mL^1(\Omega_T; \mL^{\infty}(\mI_{w, \rho})))} + t \, D(t, B^\infty, B^{TV})
    \end{split}
\end{equation}
where, 
\begin{equation}
    \begin{split}
        D(T_1, B^\infty, B^{TV}) &\coloneqq \norm{J'}_{\mL^\infty(Q(T_1))} \abs{\gamma}_{\mTV(\R^n)}  \Big( \abs{\gamma}_{\mTV(\R^n)} \norm{J'}_{\mL^\infty(\R^n)} \norm{q_\circ}_{\mL^1(\R^n)} B^\infty \norm{\partial_3^2 F}_{\mL^\infty(\Omega_T \times \mathcal U^2)}\\
        & \qquad +\norm{q_\circ}_{\mL^\infty(\R^n)} + B^{\mTV} \norm{\partial_3 F}_{\mL^\infty(\Omega_T \times \mathcal I_{\rho}^2)}\Big) +  \norm{\nabla F}_{\mL^\infty(\Omega_T \times \mathcal I_{\rho}^2)} \\
        &\qquad+ \abs{\gamma}_{\mTV(\R^n)} \norm{J'}_{\mL^\infty(Q(T_1))} B^\infty  \norm{\partial_3 F}_{\mL^\infty(\Omega_T \times \mathcal I_\rho^2)} + \norm{\partial_4 F}_{\mL^\infty(\Omega_T \times \mathcal I_\rho^2)}.
    \end{split}
\end{equation}
Thanks to \cref{ass:flux_nonlocal}, the constant $D(t, B^\infty, B^{\mTV})$ is bounded and hence for sufficiently small $T$, we can set 
\begin{equation*}
     \norm{q_0}_{\mL^1(\R^n)} + \|\diverg\bF\|_{\mL^1(\Omega_T; \mL^{\infty}(\mathcal U^2))} + t \, D(T_1, B^\infty, B^{TV}) \le 42 \norm{q_\circ}_{\mL^1(\R^n)}
\end{equation*}
In particular, one can choose $T_2\in(0,T_1]$ as 
\[
0 < T_2 \le  \frac{41 \norm{q_\circ}_{\mL^1(\R^n)} - \|\diverg\bF\|_{\mL^1(\Omega_{T_2}; \mL^{\infty}(\mI_{w, \rho}))}} {D(T_1, B^\infty, B^{\mTV})},
\]
 which is well-defined $\norm{\diverg \bF}_{\mL^1(\Omega_T; \mL^\infty(\mathcal U^2))}$ is monotonically decreasing as $T \to 0$.  
Therefore, 
\[\norm{\mF[\tilde q](t, \cdot)}_{\mL^1(\R^n)} \le 42\norm{q_\circ}_{\mL^1(\R^n)}\] for all $t \in [0, T_2]$. 
\end{description}
Collecting all of the above estimates, the self-mapping property of $\mF$ follows. 
\end{proof}

\begin{remark}[$\mL^1$ continuity in time]
    The $\mL^1$ time continuity of $\mF[q]$ for $q \in \mathcal Y(T)$, as in \cref{theo:existence_uniqueness_less_regular_flux}, holds true immediately, following from the bound defined at \cref{E:conservation-bound}. 
\end{remark}

\begin{lemma}[contraction of $\mF$ in $\mL^{1}(\R^n)$]
\label{lem:contraction_mapping_L_1}
Let \cref{ass:general_nonlocal} hold, and let $\mF$ be as in \cref{defi:fixed_point}. 
Then there exists a $T\in\R_{>0}$ such that for any $q,\tilde{q}\in \mY(T)$,
\[
\|\mF[q]-\mF[\tilde{q}]\|_{\mC([0,T];\mL^{1}(\R^n))}
\leq \tfrac{1}{2}\|q-\tilde{q}\|_{\mC([0,T];\mL^{1}(\R^n))};
\]
that is, $\mF$ is a contraction on $\mC\big([0,T];\mL^{1}(\R^n)\big)$.
\end{lemma}

\begin{proof}
By \cref{lem:self_mapping}, there exists a $T_{2}\in\R_{>0}$ such that 
$\mF[\mY(T_{2})]\subset\mY(T_{2})$, so we may take $q,\tilde{q}\in \mY(T_{2})$ 
throughout. Define the fluxes 
\begin{equation}
\label{fluxes}
    \bPhi(t,\bx, \rho)\coloneqq \bF\big(t,\bx,\mW[J(q),\gamma](t,\bx),\rho\big),
\qquad
\hat{\bPhi}(t,\bx,\rho)\coloneqq \bF\big(t,\bx,\mW[J(\tilde{q}),\gamma](t,\bx),\rho\big),
\end{equation}
for $(t,\bx,\rho)\in\Omega_{T_2}\times\R$, and set
\[
\kappa_{\bF}\coloneqq \max\Big\{
    \|\partial_{3}\bPhi-\partial_{3}\hat{\bPhi}\|_{\mL^{\infty}(\Omega_T\times\R^n\times\mI_\rho)},\ 
    (2n+1)\|\nabla\partial_{3}\bPhi\|_{\mL^{\infty}(\Omega_T\times\R^n\times\mI_\rho)}
\Big\}.
\]
Applying the stability estimate of \cref{theo:existence_uniqueness_less_regular_flux} 
to $\bPhi$ and $\hat{\bPhi}$ yields
\begin{equation}
\begin{split}
    &\|\mF[q](t,\cdot)-\mF[\tilde{q}](t,\cdot)\|_{\mL^{1}(\R^n)}\\
    &\leq t\e^{\kappa_{\bF}t}|q_{\circ}|_{\mTV(\R^n)}
        \esssup_{(t,\bx,\rho)\in\Omega_T\times\mI_\rho}
        \big|\partial_{4}\bPhi(t,\bx,\rho)
            -\partial_{4}\hat{\bPhi}(t,\bx,\rho)\big|\\
    &\quad +nW_{n}
        \esssup_{(t,\bx,\rho)\in\Omega_T\times\mI_\rho}
        \big|\partial_{4}\bPhi(t,\bx,\rho)
            -\partial_{4}\hat{\bPhi}(t,\bx,\rho)\big|
        \int_{0}^{t}s\e^{\kappa_{\bF}(t-s)}
        \|\nabla\diverg\bPhi(s,\cdot,\ast)\|_{\mL^{1}(\R^n;\mL^{\infty}(\mI_\rho))}\dd s\\
    &\quad +\int_{0}^{t}\e^{\kappa_{\bF}(t-s)}
        \|\diverg\bPhi(s,\ast,\cdot)
            -\diverg\hat{\bPhi}(s,\ast,\cdot)\|_{\mL^{1}(\R^n;\mL^{\infty}(\mI_\rho))}\dd s.
\end{split}
\label{eq:contraction_mapping_the_final_frontier}
\end{equation}
It remains to bound each term on the right-hand side in terms of 
$\|q-\tilde{q}\|_{\mC([0,t];\mL^{1}(\R^n))}$, which we do in the following steps.

\begin{description}
\item[Bound on $\esssup|\partial_4\bPhi-\partial_4\hat{\bPhi}|$.]
Using the Lipschitz continuity of $\partial_4\bF$ in its third argument and the linearity 
of $\mW$, we estimate
\begin{align*}
&\esssup_{(t,\bx,\rho)\in\Omega_{T}\times\mI_\rho}
    \big|\partial_{4}\bPhi(t,\bx,\rho)-\partial_{4}\hat{\bPhi}(t,\bx,\rho)\big|\\
& \hspace{1.5in} \leq \|\partial_{3}\partial_{4}\bF\|_{\mL^{\infty}(\Omega_T\times\R^n\times\mI_{\omega,\rho})}
    \|\mW[J(q)-J(\tilde{q}),\gamma]\|_{\mL^{\infty}(\Omega_{T_{2}})}\\
&\hspace{1.5in} \leq \underbrace{
    \|\partial_{3}\partial_{4}\bF\|_{\mL^{\infty}(\Omega_T\times\R^n\times\mI_{\omega,\rho})}
    \|\gamma\|_{\mL^{\infty}(\R^n)}
    }_{=:\,C_{1}(B^{\infty})}
    \|q(t,\cdot)-\tilde{q}(t,\cdot)\|_{\mL^{1}(\R^n)},
\end{align*}
where the last step uses \cref{E:WLinfty} and $J(0)=0$.

\item [Bound on $\|\nabla\diverg\bPhi\|$]
For $s\in(0,t)$, since $q\in\mY(T_2)$, the bound from \cref{lem:self_mapping} gives
\[
\|\nabla\diverg\bPhi(s,\cdot,\ast)\|_{\mL^{1}(\R^n;\mL^{\infty}(\mI_\rho))}
\leq C\big(s,B^{\infty},B^{\mTV}\big),
\]
with $C\big(s,B^{\infty},B^{\mTV}\big)$ as in \cref{eq:C_t_B_infty_TV_1} or \cref{eq:C_t_B_infty_TV_2}, which is finite 
by \cref{ass:general_nonlocal}.

\item [Bound on $\|\diverg\bPhi-\diverg\hat{\bPhi}\|$.]
Expanding the divergence and using the product rule, for $(t,\bx)\in(0,T_2)\times\R^n$, leads to
\begin{align*}
    &\|\diverg\bPhi(t,\bx,\cdot)-\diverg\hat{\bPhi}(t,\bx,\cdot)\|_{\mL^{\infty}(\mI_\rho)}\\
    &\leq\|\diverg\bF(t,\bx,\mW[J(q),\gamma](t,\bx),\cdot)
        -\diverg\bF(t,\bx,\mW[J(\tilde{q}),\gamma](t,\bx),\cdot)\|_{\mL^{\infty}(\mI_\rho)}\\
    &\quad +\big\|\partial_{3}\bF(t,\bx,\mW[J(q),\gamma](t,\bx),\cdot)
            \nabla\mW[J(q),\gamma](t,\bx)\\
            &\qquad\qquad\qquad\qquad\qquad\qquad
            -\partial_{3}\bF(t,\bx,\mW[J(\tilde{q}),\gamma](t,\bx),\cdot)
            \nabla\mW[J(\tilde{q}),\gamma](t,\bx)\big\|_{\mL^{\infty}(\mI_\rho)}\\
    &\leq \|\partial_{3}\diverg\bF\|_{\mL^{\infty}(\Omega_T\times\R^n\times\mI_{\omega,\rho})}
        \|J'\|_{\mL^{\infty}(\mI_\rho)}\mW[|q-\tilde{q}|,\gamma](t,\bx)\\
    &\quad +\|\partial_{3}\bF\|_{\mL^{\infty}(\Omega_T\times\R^n\times\mI_{\omega,\rho})}
        |\nabla\mW[J(q)-J(\tilde{q}),\gamma](t,\bx)|\\
    &\quad +\|\nabla\mW[J(\tilde{q}),\gamma]\|_{\mL^{\infty}(\Omega_{T_{2}})}
        \|\partial_{3}^{2}\bF\|_{\mL^{\infty}(\Omega_T\times\R^n\times\mI_{\omega,\rho})}
        \|J'\|_{\mL^{\infty}(\mI_\rho)}\mW[|q-\tilde{q}|,\gamma](t,\bx).
\end{align*}
Integrating over $\R^n$ and applying \cref{E:WL1}, \cref{E:DWL1}, and 
\cref{eq:nonlocal_estimates_fixed_point} yields
\begin{align*}
    &\|\diverg\bPhi(t,\ast,\cdot)-\diverg\hat{\bPhi}(t,\ast,\cdot)\|_{\mL^{1}(\R^n;\mL^{\infty}(\mI_\rho))}\\
    &\leq \|J'\|_{\mL^{\infty}(\mI_\rho)}\Big(
        \|\partial_{3}\diverg\bF\|_{\mL^{\infty}(\Omega_T\times\R^n\times\mI_{\omega,\rho})}
        +|\gamma|_{\mTV(\R^n)}B^{\infty}
        \|\partial_{3}^{2}\bF\|_{\mL^{\infty}(\Omega_T\times\R^n\times\mI_{\omega,\rho})}
        \|J'\|_{\mL^{\infty}(\mI_\rho)} \\
    &\quad+|\gamma|_{\mTV(\R^n)}
        \|\partial_{3}\bF\|_{\mL^{\infty}(\Omega_T\times\R^n\times\mI_{\omega,\rho})}
    \Big)
    \|q(t,\cdot)-\tilde{q}(t,\cdot)\|_{\mL^{1}(\R^n)}\\
    &\eqqcolon C_{2}(B^{\infty})\|q(t,\cdot)-\tilde{q}(t,\cdot)\|_{\mL^{1}(\R^n)}.
\end{align*}

\smallskip
\noindent
\item[Expression for $\kappa_{\bF}$.]
Since both $\bPhi$ and $\hat{\bPhi}$ have arguments lying in $\mI_{\omega,\rho}$, 
$\kappa_{\bF}$ reduces to
\[
\begin{split}
\kappa_{\bF}(B^{\infty})&\coloneqq \max\Big\{
    \|\partial_{4}\bF\|_{\mL^{\infty}(\Omega_T\times\R^n\times\mI_{\omega,\rho})},\\
    &\qquad\qquad\qquad
    (2n+1)\Big(
        \|\nabla\partial_{4}\bF\|_{\mL^{\infty}(\Omega_T\times\R^n\times\mI_{\omega,\rho})}
        +\|\partial_{3}\partial_{4}\bF\|_{\mL^{\infty}(\Omega_T\times\R^n\times\mI_{\omega,\rho})}B^{\infty}
    \Big)
\Big\},
\end{split}
\]
which is finite by \cref{ass:general_nonlocal}. Consequently, $C_1$, $C_2$, $C$, and 
$\kappa_{\bF}$ are all finite.

Substituting the bounds on $C_1$, $C_2$, $C$, and $\kappa_{\bF}$ back into 
\cref{eq:contraction_mapping_the_final_frontier} yields
\begin{align*}
    \|\mF[q](t,\cdot)-\mF[\tilde{q}](t,\cdot)\|_{\mL^{1}(\R^n)}
    &\leq \Big(t\e^{\kappa_{\bF}(B^{\infty})t}|q_{\circ}|_{\mTV(\R^n)}C_{1}(B^{\infty})\\
    &\quad +nW_{n}C_{1}(B^{\infty})\!\!\int_{0}^{t}\!\!\!s\e^{\kappa_{\bF}(B^{\infty})(t-s)}
        C(s,B^{\infty},B^{\mTV})\dd s\Big)
    \|q(t,\cdot)-\tilde{q}(t,\cdot)\|_{\mL^{1}(\R^n)}\\
    &\quad + C_{2}(B^{\infty})\int_{0}^{t}\e^{\kappa_{\bF}(B^{\infty})(t-s)}
        \|q(s,\cdot)-\tilde{q}(s,\cdot)\|_{\mL^{1}(\R^n)}\dd s.
\end{align*}
Taking the supremum over $[0,t]$ gives
\begin{align*}
\|\mF[q]-\mF[\tilde{q}]\|_{\mC([0,t];\mL^{1}(\R^n))}
    &\leq \Lambda(t)\|q-\tilde{q}\|_{\mC([0,t];\mL^{1}(\R^n))},
\end{align*}
where
\[
\Lambda(t)\coloneqq t\e^{\kappa_{\bF}t}|q_{\circ}|_{\mTV(\R^n)}C_{1}(B^{\infty})
+nW_{n}C_{1}(B^{\infty})\int_{0}^{t}\!\!s\e^{\kappa_{\bF}(t-s)}C(s,B^{\infty},B^{\mTV})\dd s
+ C_{2}(B^{\infty})\int_{0}^{t}\!\!\e^{\kappa_{\bF}(t-s)}\dd s.
\]
Since $\Lambda(t)\to 0$ as $t\to 0$, we may choose $T_{3}\in(0,T_{2}]$ small enough 
that $\Lambda(t)\leq\tfrac{1}{2}$ for all $t\in[0,T_{3}]$, which gives
\[
\|\mF[q]-\mF[\tilde{q}]\|_{\mC([0,T_{3}];\mL^{1}(\R^n))}
\leq \tfrac{1}{2}\|q-\tilde{q}\|_{\mC([0,T_{3}];\mL^{1}(\R^n))};
\]
that is, $\mF$ is a contraction on $\mC\big([0,T_{3}];\mL^{1}(\R^n)\big)$.\qedhere
\end{description}
\end{proof}

\begin{lemma}[existence and uniqueness of a fixed point of $\mF$]
\label{lem:existence_uniqueness_fixed-point}
Let \cref{ass:general_nonlocal} hold, and let $\mF$ be as in \cref{defi:fixed_point}. 
Then, there exists a time horizon $T\in\R_{>0}$ and a unique
\[
q^*\in \mC\big([0,T];\mL^{1}(\R^n)\big)\cap\mL^{\infty}\big((0,T);\mL^{\infty}(\R^n)\cap\mTV(\R^n)\big)
\]
such that $\mF[q^*]\equiv q^*$ in $\mC\big([0,T];\mL^{1}(\R^n)\big)$. 
Moreover, $q^*\in \mY(T)$.
\end{lemma}

\begin{proof}
By \cref{defi:admissible_set}, the set $\mY(T)$ is closed in $\mC\big([0,T];\mL^{1}(\R^n)\big)$, 
which is a Banach space. The result then follows from \cref{lem:self_mapping,lem:contraction_mapping_L_1} by the Banach fixed-point theorem.
\end{proof}
Collecting all the results in \cref{lem:self_mapping}---specifically the proven bounds \cref{eq:TV_bound_fixed_point}, \cref{E:conservation-bound} and the general stability result of \cref{eq:contraction_mapping_the_final_frontier}---as well as \cref{lem:existence_uniqueness_fixed-point} we have proven the following result.
\begin{theorem}[existence and uniqueness for \cref{E:model_nonlocal} on a small time horizon]
\label{theo:existence_uniqueness_nonlocal_small_time_horizon}
Let \cref{ass:general_nonlocal} hold. Then there exists a time horizon $T\in\R_{>0}$ 
and a unique weak entropy solution
\[
q\in\mC\big([0,T];\mL^{1}(\R^n)\big)\cap\mL^{\infty}\big((0,T);\mL^{\infty}(\R^n)\cap\mTV(\R^n)\big)
\]
of \cref{E:model_nonlocal} in the sense of \cref{defi:entropy_solution}. Furthermore, 
\begin{equation}
\label{E:q_fixed_point_solution}
\begin{aligned}
\|q(t,\cdot)\|_{\mL^{\infty}(\R^n)}
    &\leq 42\|q_{\circ}\|_{\mL^{\infty}(\R^n)}\\
|q(t,\cdot)|_{\mTV(\R^n)}
    &\leq \abs{q_{\circ}}_{\mTV(\R^n)} \e^{\eta(B^{\infty}) t} + n W_n\mathrm{e}^{\eta(B^{\infty}) t}C\big(t,B^{\infty},B^{\mTV}\big),\\
\|q(t,\cdot)-q(s,\cdot)\|_{\mL^{1}(\R^n)}
    &\leq C^{\infty}(t,s)) + \abs{t -s}\, D(T_1, B^\infty, B^{\mTV})
\end{aligned}
\end{equation}
where $C^{\infty}(t,s)\coloneqq \int_{\min\{s,t\}}^{\max\{t,s\}}\int_{\R^n} \|\diverg\bF(\tau, \bx, *)\|_{\mL^{\infty}(\mI(B^\infty)))} \dd \bx \dd \tau $ for \((t,s)\in[0,T]^{2}\) which vanishes as $\abs{t -s } \to 0$, and $\eta(B^\infty)$ is defined as in \cref{E:eta-B}.

Furthermore, let $q$ and $\tilde q$ be solutions \cref{E:model_nonlocal} with respect to the initial conditions $q_0$ and $\tilde{q}_0$, respectively. In addition, let $\bPhi$, $\hat{\bPhi}$ be defined as in \cref{fluxes}. Then, we have the following general stability estimate
\begin{equation}
\begin{split}
    &\|q(t,\cdot)-\tilde{q}(t,\cdot)\|_{\mL^{1}(\R^n)}\\
    &\leq t\e^{\kappa_{\bF}t}|q_{\circ}|_{\mTV(\R^n)}
        \esssup_{(t,\bx,q)\in\Omega_T\times\mI_\rho}
        \big|\partial_{4}\bPhi(t,\bx,q)
            -\partial_{4}\hat{\bPhi}(t,\bx,q)\big|\\
    &\quad +nW_{n}
        \esssup_{(t,\bx,q)\in\Omega_T\times\mI_\rho}
        \big|\partial_{4}\bPhi(t,\bx,q)
            -\partial_{4}\hat{\bPhi}(t,\bx,q)\big|
        \int_{0}^{t}s\e^{\kappa_{\bF}(t-s)}
        \|\nabla\diverg\bPhi(s,\cdot,\ast)\|_{\mL^{1}(\R^n;\mL^{\infty}(\mI_\rho))}\dd s\\
    &\quad +\int_{0}^{t}\e^{\kappa_{\bF}(t-s)}
        \|\diverg\bPhi(s,\ast,\cdot)
            -\diverg\hat{\bPhi}(s,\ast,\cdot)\|_{\mL^{1}(\R^n;\mL^{\infty}(\mI_\rho))}\dd s.
\end{split}
\end{equation}
\end{theorem}


We conclude with the following corollary, which establishes global well-posedness 
under the additional assumption that the flux vanishes at the extremal values of the 
initial datum, yielding a maximum principle.

\begin{corollary}[global well-posedness under a maximum principle condition]
\label{cor:large_time_horizon}
Let \cref{ass:general_nonlocal} hold, and assume that there exist $\R\ni q_{\min} \leq  q_{\max}\in\R$ 
such that
\begin{equation}
\bF(\cdot,\ast,\star,q_{\min})=\boldsymbol{0}=\bF(\cdot,\ast,\star,q_{\max}).
\label{eq:F_maximum_principle}
\end{equation}
Let $q_{\circ}\in \mBV(\R^n)\cap\mL^{\infty}(\R^n)$ satisfy
\[
q_{\min}\leq q_{\circ}\leq q_{\max} \qquad\text{a.e. on } \R^n.
\]
Then, for any $T\in\R_{>0}$, there exists a unique weak entropy solution of \cref{E:model_nonlocal} 
in the sense of \cref{defi:entropy_solution}, and it satisfies
\[
q_{\min}\leq q(t,\bx)\leq q_{\max} \qquad\text{a.e. on } (0,T)\times\R^n.
\]
\end{corollary}

\begin{proof}
By \cref{theo:existence_uniqueness_nonlocal_small_time_horizon}, there exists a time 
horizon $T\in\R_{>0}$ on which a unique weak entropy solution to \cref{E:model_nonlocal} 
exists. To extend it to any finite time horizon, we establish a uniform $\mL^{\infty}(\R^n)$ 
bound and a uniform $\mTV$ bound on the solution, which then allows us to apply a time 
clustering argument.

\begin{description}
\item[$\mL^{\infty}(\R^n)$ bound:]
Proceeding as in \cref{theo:existence_uniqueness_smooth}, we note that the $\mL^{\infty}(\R^n)$-norm 
of the solution satisfies
\begin{align*}
\tfrac{\dd}{\dd t} \|q(t,\cdot)\|_{\mL^{\infty}(\R^n)}
    &\leq\big\|\diverg \bF\big(t,\ast,\mW[J(q),\gamma](t,\ast),\|q(t,\cdot)\|_{\mL^{\infty}(\R^n)}\big)\big\|_{\mL^{\infty}(\R^n)}\\
    &\quad +\big\|\partial_{3}\bF\big(t,\ast,\mW[J(q),\gamma](t,\ast),\|q(t,\cdot)\|_{\mL^{\infty}(\R^n)}\big)\circ\nabla\mW[J(q),\gamma](t,\ast)\big\|_{\mL^{\infty}(\R^n)}.
\end{align*}
By \cref{eq:F_maximum_principle}, the right-hand side vanishes when 
$\|q(t,\cdot)\|_{\mL^{\infty}(\R^n)}$ reaches $q_{\max}$, so the ODE comparison 
principle yields $\|q(t,\cdot)\|_{\mL^{\infty}(\R^n)}\leq q_{\max}$ for all $t\in[0,T]$. 
An analogous argument gives the lower bound $q_{\min}\leq q(t,\cdot)$, establishing a 
uniform (in time) $\mL^{\infty}(\R^n)$ bound.

\item[$\mTV(\R^n)$ bound:]
Rather than invoking \cref{eq:TV_bound_fixed_point} directly, we use the finer $\mTV$ 
estimate from \cref{theo:existence_uniqueness_less_regular_flux}, which can be exploited for a 
Grönwall argument. Set
\[
\mathcal{Q}\coloneqq (q_{\min},q_{\max})\subset\R,\]
\[J\mathcal{Q}\coloneqq \big(-\|J'\|_{\mL^{\infty}(\mathcal{Q})}\max\{|q_{\min}|,|q_{\max}|\},\,
\|J'\|_{\mL^{\infty}(\mathcal{Q})}\max\{|q_{\min}|,|q_{\max}|\}\big)\subset\R,
\]
and $\bPhi(t,\bx,q)\coloneqq \bF(t,\bx,\mW[J(q),\gamma](t,\bx),q)$ for 
$(t,\bx,q)\in\Omega_T\times\mathcal{Q}$. The estimate reads, for all $t\in[0,T]$,
\begin{align}
    |q(t,\cdot)|_{\mTV(\R^n)}&\leq |q_{\circ}|_{\mTV(\R^n)} \e^{\kappa t}
    + n W_n \int_0^t \e^{\kappa(t-s)}
    \int_{\R^n} 
        \|\nabla\diverg \bPhi(s, \bx, \cdot)\|_{\mL^{\infty}(\mathcal{Q})} 
    \dd \bx \dd s,\label{eq:TV_Gronwall}
\end{align}
with $\kappa\coloneqq (2n+1)\|\diverg\partial_{3}\bPhi\|_{\mL^{\infty}(\Omega_T\times\R^n\times\mathcal{Q})}$.
Since $q(t,\cdot)\in\mathcal{Q}$ by the $\mL^\infty$ bound, the nonlocal term satisfies
\[
\|\mW[J(q),\gamma]\|_{\mL^{\infty}(\Omega_T)}\leq \|J'\|_{\mL^{\infty}(\mathcal{Q})}\max\{|q_{\min}|,|q_{\max}|\},
\]
so the arguments of $\bF(t,\bx,*,\cdot)$ lie in $J\mathcal{Q}\times\mathcal{Q}$. Converting $\kappa$ 
into an expression involving $\bF$ and $\mW$, analogously to \cref{lem:self_mapping}, gives
\begin{align*}
\kappa=(2n+1)\Big(
    &\|\diverg\partial_{4}\bF\|_{\mL^{\infty}(\Omega_T\times\R^n\times J\mathcal{Q}\times\mathcal{Q})}\\
    &+\|\partial_{3}\partial_{4}\bF\|_{\mL^{\infty}(\Omega_T\times\R^n\times J\mathcal{Q}\times\mathcal{Q})}
    |\gamma|_{\mTV(\R^n)}\|J'\|_{\mL^{\infty}(\mathcal{Q})}\max\{|q_{\min}|,|q_{\max}|\}
\Big),
\end{align*}
which is finite and uniform in the solution by \cref{ass:general_nonlocal}. For the 
integrand in \cref{eq:TV_Gronwall}, under the first part of \cref{item:7_flux} of \cref{ass:flux_nonlocal},
\begin{align*}
    \|\nabla\diverg\bPhi(s,\ast,\cdot)\|_{\mL^{1}(\R^n;\mL^{\infty}(\mathcal{Q}))}
    &\leq \|\nabla\diverg\bF(s,\star,\ast,\cdot)\|_{\mL^{1}(\R^n;\mL^{\infty}(J\mathcal{Q}\times\mathcal{Q}))}\\
    &\quad +\|\partial_{3}^{2}\bF\|_{\mL^{\infty}(\Omega_{T}\times J\mathcal{Q}\times\mathcal{Q})}
        \|J'\|_{\mL^{\infty}(\mathcal{Q})}|\gamma|_{\mTV(\R^n)}\|q_{\circ}\|_{\mL^{1}(\R^n)}\\
    &\quad +\|\partial_{3}\diverg\bF\|_{\mL^{\infty}(\Omega_{T}\times J\mathcal{Q}\times\mathcal{Q})}
        \|J'\|_{\mL^{\infty}(\mathcal{Q})}|\gamma|_{\mTV(\R^n)}\|q_{\circ}\|_{\mL^{1}(\R^n)}\\
    &\quad +\|\partial_{3}\bF\|_{\mL^{\infty}(\Omega_{T}\times J\mathcal{Q}\times\mathcal{Q})}
        \|J'\|_{\mL^{\infty}(\mathcal{Q})}|\gamma|_{\mTV(\R^n)}|q(s,\cdot)|_{\mTV(\R^n)},
\end{align*}
and under the second part of \cref{item:7_flux},
\begin{align*}
    \|\nabla\diverg\bPhi(s,\ast,\cdot)\|_{\mL^{1}(\R^n;\mL^{\infty}(\mathcal{Q}))}
    &\leq \|\partial_{3}\nabla\diverg\bF\|_{\mL^{\infty}(\Omega_{T}\times J\mathcal{Q}\times\mathcal{Q})}
        \|J'\|_{\mL^{\infty}(\mathcal{Q})}\|q_{\circ}\|_{\mL^{1}(\R^n)}\\
    &\quad +\|\partial_{3}^{2}\bF\|_{\mL^{\infty}(\Omega_{T}\times J\mathcal{Q}\times\mathcal{Q})}
        \|J'\|_{\mL^{\infty}(\mathcal{Q})}|\gamma|_{\mTV(\R^n)}\|q_{\circ}\|_{\mL^{1}(\R^n)}\\
    &\quad +\|\partial_{3}\diverg\bF\|_{\mL^{\infty}(\Omega_{T}\times J\mathcal{Q}\times\mathcal{Q})}
        \|J'\|_{\mL^{\infty}(\mathcal{Q})}|\gamma|_{\mTV(\R^n)}\|q_{\circ}\|_{\mL^{1}(\R^n)}\\
    &\quad +\|\partial_{3}\bF\|_{\mL^{\infty}(\Omega_{T}\times J\mathcal{Q}\times\mathcal{Q})}
        \|J'\|_{\mL^{\infty}(\mathcal{Q})}|\gamma|_{\mTV(\R^n)}|q(s,\cdot)|_{\mTV(\R^n)}.
\end{align*}

In both cases, the first three summands are uniformly bounded by \cref{ass:general_nonlocal} 
and the $\mL^{\infty}$ bound established above, whereas the fourth depends linearly on 
$|q(s,\cdot)|_{\mTV(\R^n)}$. Accordingly, in each case, for suitable constants $C_0, C_{\mTV}$,
\[
 \|\nabla\diverg\bPhi(s,\ast,\cdot)\|_{\mL^{1}(\R^n;\mL^{\infty}(\mathcal{Q}))} \leq C_0 + C_{\mTV} |q(s,\cdot)|_{\mTV(\R^n)}.
\]
Substituting this into \cref{eq:TV_Gronwall} gives
\[
|q(t,\cdot)|_{\mTV(\R^n)}
\leq \underbrace{|q_{\circ}|_{\mTV(\R^n)} \e^{\kappa t} 
+ C_0 \int_0^t \e^{\kappa(t-s)}\dd s}_{=:\,\alpha(t)}
+ C_{\mTV}\int_0^t \e^{\kappa(t-s)}|q(s,\cdot)|_{\mTV(\R^n)}\dd s.
\]
Thus, by Grönwall's lemma,
\[
|q(t,\cdot)|_{\mTV(\R^n)} \leq \alpha(t)\,\e^{C_{\mTV}, t},
\]
so the $\mTV$ seminorm of the solution grows at most exponentially in time.
\end{description}

In conclusion, with a uniform $\mL^{\infty}$ bound and an at-most-exponential (in time) $\mTV$ bound in hand, the 
solution remains in a bounded subset of $\mathcal X$ on any finite time interval. We can 
therefore apply the clustering argument of 
\cref{theo:existence_uniqueness_nonlocal_small_time_horizon}: cover $[0,T]$ by a finite 
sequence of subintervals, each short enough for the local existence result to apply, with 
the terminal value of each subproblem serving as the initial datum for the next. Since the 
$\mTV$ bound ensures that the required smallness condition does not degenerate, the time steps 
are bounded away from zero, and the argument exhausts any finite horizon. This concludes 
the proof.

\end{proof}

\begin{remark}[time-dependent nonlocal kernel]\label{rem:time_dependent_kernel}
The kernel \(\gamma\) in \cref{ass:general_nonlocal} may also depend on time, as considered in \cite{radici2023entropy}. Our theory extends to such kernels under the assumptions \(\gamma\in \mC([0,T];\mL^{1}(\R^n))\cap \mL^{\infty}((0,T);\mL^{\infty}\cap\mBV(\R^n))\) and \(\|\gamma(t,\cdot)\|_{\mL^{1}(\R^n)}\leq 1\) for all \(t\in[0,T]\). The proofs require only minor modifications, replacing total variation bounds of \(\gamma\) by \(\mL^{\infty}\) bounds in time.
\end{remark}
\begin{remark}[Admissible fluxes and applications to results in the literature]
The fluxes discussed in \cref{rem:specifically_structured_flux} are also applicable for the class of nonlocal conservation laws discussed in this section. Thus, the results apply for instance to the model classes considered in \cite{amorim2015numerical,ciaramaglia2024nonlocaltrafficflowmodels,DiFrancesco2019deterministic,hu2021shock,radici2023entropy} for the scalar case and \cite{colombo2011control,courcel2025well-posedness,goatin2025nonlocalmodelheterogeneousmaterial,Rossi2020} for the multi-dimensional case as already pointed out in \cref{sec:related-work}. They are also applicable (see particularly \cref{eq:autonomous_flux}) to nonlocal dynamics where the local dependency is linear, i.e., one has for \(\bV:\R\rightarrow\R^{n}\) a given velocity function
\[
\partial_{t}q +\diverg\big(\bV(\mW[J(q),\gamma](t,\bx))q\big)=0,
\]
with \(\mW\) the nonlocal operator introduced in \cref{defi:nonlocal_operator}.
Such dynamics has been studied intensively (sometimes with higher regularity of the nonlocal kernel) in \cite{crippa2012existence,keimer2017existence,colombo2026multi,coclite2022existence,keimer2018existence} and related results. However, when realizing that the nonlocal term and with such the velocity of the dynamics are Lipschitz-continuous in space, one can invoke the method of characteristics and obtains uniqueness of weak solutions. Thus, the type of entropy solutions \cref{defi:entropy_solution}  considered in this manuscript is not required for obtaining uniqueness, and as such one would work with the methods mentioned in the named papers.
\end{remark}


\subsection{Nonlocal in space and time \tup{memory}}\label{subsec:memory}
We now move to the analysis of the conservation law with memory \cref{E:model_nonlocal_memory}. 
Once again, we work under the \cref{ass:flux_nonlocal}. We recall the solution operator $\mQ$ defined as in \cref{defi:solution_operator_fixed_w} which is well-defined by \cref{lem:well-posed-Q} for any $w \in \mathcal Z(T)$.  \\
We assume the following statements about the nonlocality and the historical datum.

\begin{assumption}[nonlocal kernel, nonlocal nonlinearity and historical datum]
\label{ass:general_nonlocal_memory}
    We assume the following:
    \begin{description}
        \item[Nonlocal nonlinearity and kernel:] $J\in\mathcal{W}^{1,\infty}_{\mathrm{loc}}(\R), \quad J(0)=0$.
        \item[Kernel:] $\kappa\in\mL^{\infty}\cap\mL^{1}(\R_{>0};\mBV(\R^n)\cap\mL^{\infty}(\R^n))$ with $\|\kappa\|_{\mL^{1}(\R_{>0}\times\R^n)}=1$.
        \item[Historical datum:] $q_{\circ}\in\mL^{\infty}\cap\mL^{1}(\R_{\leq0};\mBV(\R^n)\cap\mL^{\infty}(\R^n))\cap\mC\big(\R_{\leq0};\mL^{1}(\R^n)\big)$.
    \end{description}
\end{assumption}

We begin by stating the counterpart of \cref{lem:bounds_nonlocal_space}.

\begin{lemma}[bounds on the nonlocal operator with memory]
\label{lem:bounds_nonlocal_memory}
Let \cref{ass:general_nonlocal_memory} hold, and let $T\in\R_{>0}$. We define
\begin{align}
\mathcal S &\coloneqq \mC\big(\R_{\leq T};\mL^{1}(\R^n)\big)
    \cap\mL^{\infty}\big(\R_{\leq T};\mL^{\infty}(\R^n)\cap\mTV(\R^n)\big), \label{E:solution_space_memory}\\
\mK &\coloneqq \big\{\kappa\in\mL^{\infty}\cap\mL^{1}(\R_{>0};\mBV(\R^n)\cap\mL^{\infty}(\R^n))
    :\ \|\kappa\|_{\mL^{1}(\R_{>0}\times\R^n)}=1\big\} \notag.
\end{align}
Then, the nonlocal operator
\[
\mW:\mathcal S \times\mK\rightarrow \mC\big([0,T];\mL^{1}(\R^n)\big),
\]
defined by
\begin{equation}
\mW[J(q),\kappa](t,\bx)
\coloneqq \int_{\R_{<t}}\int_{\R^n}\kappa(t-s,\bx-\by)J(q(s,\by))\dd\by\dd s,
\quad (t,\bx)\in[0,T]\times\R^n,
\label{defi:nonlocal_operator_memory}
\end{equation}
satisfies
\begin{align}
\|\mW[J(q),\kappa](t,\cdot)\|_{\mL^{1}(\R^n)}
    &\leq \|J'\|_{\mL^{\infty}(Q(t))}\|q\|_{\mL^{\infty}(\R_{\leq t};\mL^{1}(\R^n))},
    \label{E:WL1-memory}\\
\|\mW[J(q),\kappa](t,\cdot)\|_{\mL^{\infty}(\R^n)}
    &\leq \|J'\|_{\mL^{\infty}(Q(t))}\|q\|_{\mL^{\infty}(\R_{\leq t}\times\R^n)},
    \label{E:WLinfty-memory}\\
\|\nabla\mW[J(q),\kappa](t,\cdot)\|_{\mL^{1}(\R^n;\R^n)}
    &\leq \|\kappa\|_{\mL^{1}(\R_{>0};\mTV(\R^n))}\|J'\|_{\mL^{\infty}(Q(t))}
    \|q\|_{\mL^{\infty}(\R_{\leq t};\mL^{1}(\R^n))},
    \label{E:DWL1-memory}\\
\|\nabla\mW[J(q),\kappa](t,\cdot)\|_{\mL^{\infty}(\R^n;\R^n)}
    &\leq \|\kappa\|_{\mL^{1}(\R_{>0};\mTV(\R^n))}\|J'\|_{\mL^{\infty}(Q(t))}
    \|q\|_{\mL^{\infty}(\R_{\leq t}\times\R^n)},
    \label{E:DWLinfty-memory}\\
|\nabla\mW[J(q),\kappa](t,\cdot)|_{\mTV(\R^n;\R^{n\times n})}
    &\leq \|\kappa\|_{\mL^{1}(\R_{>0};\mTV(\R^n))}\|J'\|_{\mL^{\infty}(Q(t))}
    |q|_{\mL^{\infty}(\R_{\leq t};\mTV(\R^n))},
    \label{E:DWTV-memory}
\end{align}
where, $Q(t)\coloneqq\big(-\|q\|_{\mL^{\infty}(\R_{\leq t}\times\R^n)},\|q\|_{\mL^{\infty}(\R_{\leq t}\times\R^n)}\big)$.
\end{lemma}

\begin{proof}
To prove \cref{E:WL1-memory}, using the definition of $\mW$ and Fubini's theorem, we have
    \begin{equation*}
        \begin{split}
            \norm{\mW[J(q), \kappa](t, \cdot)}_{\mL^1(\R^n)}
            & \le \int_{\R^n} \int_{\R_{<t}}\int_{\R^n} \abs{\kappa(t- s, \bm x - \bm y) J(q(s, \bm y))} \dd \bm y \dd s \dd \bm x \\
            & \le \int_{\R^n} \norm{J(q(\cdot, \bm y))}_{\mL^\infty(\R_{<t})} \int_{\R_{<t}} \int_{\R^n} \abs{\kappa(t - s, \bm x - \bm y)} \dd \bm x \dd s \dd \bm y \\
            & \le \norm{J(q)}_{\mL^\infty(\R_{<t}; \mL^1(\R^n))} \\
            & \le \norm{J'}_{\mL^\infty(Q(t))} \norm{q}_{\mL^\infty(\R_{<t}; \mL^1(\R^n))},
        \end{split}
    \end{equation*}
where we used the fact that $\|\kappa\|_{\mL^1(\R_{>0}\times\R^n)}=1$. The proof of \cref{E:WLinfty-memory} follows similarly.

To prove \cref{E:DWL1-memory} and \cref{E:DWLinfty-memory}, let $\bv\in\R^n$ be any unit-norm vector. Then, for any $h>0$ small enough and $(t,\bx)\in(0,T)\times\R^n$, we have
\begin{equation}\label{E:W-Lipschitz-memory}
\begin{split}
&\abs[\big]{\mW[J(q), \kappa](t, \bx + h\bv)-\mW[J(q), \kappa](t, \bx)}\\
&\leq \int_{\R_{<t}} \int_{\R^n}|\kappa(t - s, \bx+h\bv-\by)-\kappa(t - s, \bx-\by)||J(q(s,\by))|\dd \bm y \dd s\\
&\leq \|J(q)\|_{\mL^{\infty}(\R_{<t} \times \R^n)} \int_{\R_{<t}}\int_{\R^n}|\kappa(t - s, \bx+h\bv-\by)-\kappa(t - s,\bx-\by)|\dd \bm y \dd s\\
&\leq \|J'\|_{\mL^{\infty}(Q(t))}\|q\|_{\mL^{\infty}(\R_{<t} \times\R^n)}\abs{h}\, \norm{\kappa}_{\mL^1(\R_{>0}; \mTV(\R^n))}.
\end{split}
\end{equation}
Dividing by $\abs{h}$ and taking the limit as $\abs{h} \to 0$, we conclude that
\begin{equation*}
    \mW[J(q), \kappa](t, \cdot) \in \mathcal{W}^{1, \infty}(\R^n; \R), \quad t \in [0, T].
\end{equation*}
Taking the supremum for $x \in \R^n$ yields \eqref{E:DWLinfty-memory}. 

In addition, keeping \cref{E:W-Lipschitz-memory} in mind and applying Fubini's theorem, we can calculate
\begin{equation*}
    \begin{split}
       \int_{\R^n} |\nabla\mW[J(q), \kappa]&(t, \bm x) |\dd \bm x \\
       & \le \int_{\R^n} \norm{J(q(\cdot, \bm y))}_{\mL^\infty(\R_{<t})} \int_{\R_{<t}} \int_{\R^n} |\kappa(t - s, \bx+h\bv-\by)-\kappa(t - s,\bx-\by)|\dd \bm x \dd s \dd \bm y\\
       & \le \norm{J(q)}_{\mL^{\infty}(\R_{<t}; \mL^1(\R^n))} \norm{\kappa}_{\mL^1(\R_{>0};\mTV(\R^n))}\\
        &\leq \|J'\|_{\mL^{\infty}(Q(t))}\|q\|_{\mL^{\infty}(\R_{<t}; \mL^1(\R^n))} \norm{\kappa}_{\mL^1(\R_{>0};\mTV(\R^n))},
    \end{split}
\end{equation*}
which proves \cref{E:DWL1-memory}.

Finally, for \cref{E:DWTV-memory} we consider $\bv,\bv'\in\R^n$ with unitary norm $\|\bv\|_{2}=\|\bv'\|_{2}=1$, and we bound
\begin{align*}
       \mathbb{I}_{\nabla \mW} \coloneqq & \tfrac{1}{\abs{h}}\int_{\R^n} \abs[\big]{\nabla\mW[J(q), \kappa](t, \bm x + \bv h) - \nabla \mW[J(q), \kappa](t,\bm x)} \dd \bx  \\
        & \quad =\tfrac{1}{\abs{h}} \lim_{h'\to 0} \bigg(\int_{\R^n} \abs[\Big]{\tfrac{1}{\abs{h'}} \abs[\big]{\mW[J(q), \kappa](t, \bm x + h\bv + h'\bv') - \mW[J(q), \kappa](t, \bm x + h\bv)}\\
        &\qquad\qquad\qquad- \tfrac{1}{\abs{h'}} \abs[\big]{\mW[J(q), \kappa](t, \bm x+ h'\bv') - \mW[J(q), \kappa](t,\bm x)}} \dd \bx \bigg).
\end{align*}
We have that
\begin{equation*}
    \begin{split}
       &\mW[J(q), \kappa](t, \bm x + h \bm v + h' \bm v') -\mW[J(q), \kappa](t, \bm x + h \bm v) \\
        & = \int_0^\infty \int_{\R^n}\kappa(s, \bm y) \set[\big]{J(q(t- s, \bm x + h \bm v + h' \bm v'- \bm y)) -  J(q(t- s, \bm x + h \bm v - \bm y))} \dd \bm y \dd s.
    \end{split}
\end{equation*}
This implies that
\begin{equation*}
    \begin{split}
        \mathbb I_{\nabla \mW} & = \tfrac{1}{\abs{h}} \int_{\R^n} \lim_{h' \to 0} \tfrac{1}{\abs{h'}}\set[\bigg]{\int_{0}^{\infty} \int_{\R^n} \abs{J(q(t -s , \bm x + h \bm v + h' \bm v' - \bm y)) - J(q(t -s , \bm x + h \bm v - \bm y))} \abs{\kappa(s, \bm y)}\dd \bm y \dd s \\
        & \hspace{1.5in} -\int_{0}^{\infty} \int_{\R^n} \abs{J(q(t -s , \bm x + h' \bm v' - \bm y)) - J(q(t -s , \bm x  - \bm y))} \abs{\kappa(s, \bm y)}\dd \bm y \dd s} \dd \bm x, \\
        \intertext{by the change of variables $\bm z = \bm x + h \bm v - \bm y$, $\zeta = \bm x - \bm y$, and $\tau = t -s$,}
         &=  \tfrac{1}{\abs{h}} \int_{\R^n} \lim_{h' \to 0} \tfrac{1}{\abs{h'}}\set[\bigg]{\int_{-\infty}^{t} \int_{\R^n} \abs{J(q(\tau , \bm z +h' \bm v')) - J(q(\tau , \bm z))} \abs{\kappa(t - \tau, \bm x + h \bm v - \bm z)}\dd \bm z \dd \tau \\
        & \hspace{1.5in} - \int_{-\infty}^{t} \int_{\R^n} \abs{J(q(\tau , \bm z + h' \bm v')) - J(q(\tau , \bm z))} \abs{\kappa(t - \tau, \bm x - \zeta)}\dd \zeta \dd \tau} \dd \bm x.
        \end{split}
        \end{equation*}
        Applying Fubini's theorem to exchange $\dd \bm z$ and $\dd \bm x$, we get
        \begin{equation*}
        \begin{split}
        \mathbb I_{\nabla \mW} & \le \lim_{h' \to 0} \tfrac{1}{\abs{h'}}\tfrac{1}{\abs{h}} \norm{J'}_{\mL^\infty(Q(t))} \int_{\R^n} \norm{q(\cdot, \bm z + h' \bm v') - q(\cdot, \bm z)}_{\mL^\infty(\R_{<t})} \\
        & \hspace{2in} \times \int_{-\infty}^{t} \int_{\R^n} \abs{\kappa(t - \tau, \bm x + h \bm v - \bm z) - \kappa(t- \tau, \bm x - \bm z)} \dd \bm x \dd \tau \dd \bm z\\
        & \le \norm{J'}_{\mL^\infty(Q(t))} \norm{q}_{\mL^\infty(\R_{<t}; \mTV(\R^n))}  \norm{\kappa}_{\mL^1(\R_{>0}; \mTV(\R^n))}.
    \end{split}
\end{equation*}
This completes the proof.
\end{proof}
As in \cref{defi:admissible_set}, we introduce an admissible set---tailored to
the memory structure of the nonlocal term---which requires the solution to be defined 
on the half-line $(-\infty, T)$.
\begin{definition}[admissible set $\mY(T)$]\label{defi:admissible_set_memory}
Let \cref{ass:general_nonlocal} hold, let $T \in \R_{>0}$, and set
\[
B^{\infty} \coloneqq 42\|q_{\circ}(0,\cdot)\|_{\mL^{\infty}(\R^n)}, \qquad
B^{\mTV} \coloneqq 42|q_{\circ}(0,\cdot)|_{\mTV(\R^n)}.
\]
Recalling the spaces $\mathcal X$, \cref{E:solution-space}, and $\mathcal S$, \cref{E:solution_space_memory}, we define the admissible set
\[
\mY(T) \coloneqq \left\{ q \in \mathcal S\;:\;
\begin{array}{l}
    q \equiv q_{\circ} \text{ on } (-\infty,0] \times \R^n,\\[4pt]
    \|q\|_{\mL^{\infty}((-\infty,T);\mL^{\infty}(\R^n))} \leq B^{\infty},\\[4pt]
    \|q\|_{\mL^{\infty}((-\infty,T);\mTV(\R^n))} \leq B^{\mTV},\\[4pt]
    \|q(t,\cdot)\|_{\mL^{1}(\R^n)} \leq 42\|q_{\circ}(0,\cdot)\|_{\mL^{1}(\R^n)}
    \quad \forall\, t\in(-\infty,T)
\end{array}
\right\}.
\]
\end{definition}
\begin{definition}[fixed-point mapping] \label{defn:fixedpoint_memory}
Let $\mQ$ be the solution operator as in \cref{defi:solution_operator_fixed_w}. We define the mapping
    \[
    \mF[\cdot]\coloneqq \mQ[\mW[J(\cdot),\kappa]]:\begin{cases}\mY(T)&\rightarrow \mC\big((-\infty,T];\mL^{1}(\R^{n})\big)\\
    q&\mapsto \bigg(\R_{\leq T}\times\R^n\ni (t,\bx)\mapsto \begin{cases}\mQ[\mW[J(q),\kappa]](t,\bx), &t\in[0,T]\\
    q_{\circ},& t\in\R_{<0}
    \end{cases}
    \bigg)
    \end{cases}
    \]
    Here, \(\mW\) is the nonlocal operator defined in \cref{defi:nonlocal_operator_memory}.
\end{definition}
\begin{lemma}[self-mapping of \(\mF\)]\label{lem:self_mapping_memory}
    Let \cref{ass:general_nonlocal_memory} holds, the flux function $\bF$ has the properties of \cref{ass:flux_nonlocal} and let the mapping \(\mF\) be as in \cref{defn:fixedpoint_memory}. Then, there exists a time horizon \(T\in\R_{>0}\) such that
    \[
        \mF[\mY(T)]\subset \mY(T).
    \]
    That is,\ \(\mF\) is a self-mapping.
\end{lemma}

\begin{proof} 
   We follow the template adopted in \cref{lem:self_mapping}. First, we recall the intervals $\mI_\rho(B^\infty)$ and $\mI_{\omega,\rho}(B^\infty)$ defined at \cref{intervals}. As in the previous lemma, we will omit the dependence on $B^\infty$ throughout this proof. 
   
   Next, we fix $T\in\R_{>0}$ and set \(\tilde{q}\in\mY(T)\). Then, we can take advantage of the estimates in \crefrange{E:WL1-memory}{E:DWTV-memory} to obtain their counterparts
    \begin{equation}
    \begin{aligned}
        \|\mW[J(\tilde{q}),\kappa]\|_{\mL^{\infty}((0,T);\mL^{1}(\R^n))}&\leq 42 \|J'\|_{\mL^{\infty}(\mI_\rho)}\|q_{\circ}(0, \cdot)\|_{\mL^{1}(\R^n)}\\
        \|\mW[J(\tilde{q}),\kappa]\|_{\mL^{\infty}(\Omega_T)}&\leq \|J'\|_{\mL^{\infty}(\mI_\rho)}B^{\infty}\\
        \|\nabla\mW[J(\tilde{q}),\kappa]\|_{\mL^{\infty}((0,T);\mL^{1}(\R^n))}&\leq 42\norm{\kappa}_{\mL^1(\R_{>0}; \mTV(\R^n;\R^n))}\|J'\|_{\mL^{\infty}(\mI_\rho)}\|q_{\circ}(0, \cdot)\|_{\mL^{1}(\R^n)}\\
    \|\nabla\mW[J(\tilde{q}),\kappa]\|_{\mL^{\infty}((0,T);\mL^{\infty}(\R^n))}&\leq \norm{\kappa}_{\mL^1(\R_{>0};\mTV(\R^n;\R^n))}\|J'\|_{\mL^{\infty}(\mI_\rho)}B^{\infty}\\
    \|\nabla\mW[J(\tilde{q}),\kappa]\|_{\mL^1((0,T);\mTV(\R^n))}&\leq \norm{\kappa}_{\mL^1(\R_{>0}; \mTV(\R^n;\R^n)}\|J'\|_{\mL^{\infty}(\mI_\rho)}B^{\mTV}.
    \end{aligned}
    \label{eq:nonlocal_estimates_fixed_point_memory}
    \end{equation}
    In particular, \cref{eq:nonlocal_estimates_fixed_point_memory} implies that nonlocal operator \(\mW[J(\tilde{q}),\kappa] \in \mathcal Z(T)\) and hence, by \cref{T:Q-welldefined} there exists a unique weak entropy solution to
    \begin{align*}
        \partial_{t}q+\diverg\big(\bF(t,\bx,\mW[J(\tilde{q}),\kappa](t,\bx),q\big)&=0, && (t,\bx)\in(0,T)\times\R^n,\\
        q(0,\cdot)&=q_{\circ}(0, \cdot), && \text{on }\R^n.
    \end{align*}
    which following \cref{defn:fixedpoint_memory} is presented by \(\mF[\tilde{q}]\). Therefore, it remains to show that \(\mF[\tilde{q}]\in \mY(T)\) on $(0, T) \times \R^n$ for \(T\in\R_{>0}\) sufficiently small.

\begin{description}
\item[The \(\mL^{\infty}\)-bound:] Recalling how the \(\mL^{\infty}\)-norm can grow over time in \cref{theo:existence_uniqueness_smooth}, we can estimate the corresponding differential inequality as follows. Starting with 
    \[\bPhi(t,\bx,q)\coloneqq \bF\big(t,\bx,\mW[J(\tilde q),\kappa](t,\bx),q\big),\ (t,\bx,q)\in(0,T)\times\R^n\times\R,\]
we have 
\begin{align*}
\partial_{t}\|q(t,\cdot)\|_{\mL^{\infty}(\R^n)}&\leq\esssup_{\bz\in\R^n}\Big|\diverg\bF\big(t,\bz,\mW[J(\tilde{q}),\kappa](t,\bz),\|q(t,\cdot)\|_{\mL^{\infty}(\R^n)}\big)\Big|\\
&\quad +\esssup_{\bz\in\R^n}\Big|\partial_{3}\bF\big(t,\bz,\mW[J(\tilde{q}),\kappa](t,\bz),\|q(t,\cdot)\|_{\mL^{\infty}(\R^n)}\big)\nabla\mW[J(\tilde{q}),\kappa](t,\bz)\Big|
\intertext{using the bounds \cref{eq:nonlocal_estimates_fixed_point_memory},}
&\leq \|\diverg\bF(t,\cdot,\ast,\star)\|_{\mL^{\infty}(\R^n \times \mI_{\omega,\rho})}\\
&\quad + \|\partial_{3}\bF(t,\cdot,\ast,\star)\|_{\mL^{\infty}(\R^n \times \mI_{\omega,\rho})}\norm{\kappa}_{\mL^1(\R_{>0};\mTV(\R^n;\R^n))}\|J'\|_{\mL^{\infty}(\mI_\rho)}B^{\infty}.
\end{align*}
Integrating in time over $[0, T^*]$, and considering the \cref{ass:flux_nonlocal} on the flux function $\bF$---in particular \cref{item:2_flux}, \cref{item:3_flux}---the previous computation yields
\begin{align*}
\|q(t,\cdot)\|_{\mL^{\infty}(\R^n)}&\leq \|q_{\circ}(t, \cdot)\|_{\mL^{\infty}(\R^n)}+t\|\diverg\bF\|_{\mL^{\infty}( \Omega_T \times \R^n \times \mI_{\omega,\rho})}\\
&\quad +t \|\partial_{3}\bF\|_{\mL^{\infty}( \Omega_T \times \R^n \times \mI_{\omega,\rho})}\norm{\kappa}_{\mL^1(\R_{>0};\mTV(\R^n;\R^n))}\|J'\|_{\mL^{\infty}(\mI_\rho)}B^{\infty} \\
&  \overset{!}{\le} 42 \norm{q_\circ(0,\cdot)}_{\mL^\infty(\R^n)}.
\end{align*}
This leads to 
\[
T_{1}\coloneqq \frac{41 \norm{q_\circ(0,\cdot)}_{\mL^\infty(\R^n)}}{\|\diverg\bF\|_{\mL^{\infty}( \Omega_T \times \R^n \times \mI_{\omega,\rho})}+\|\partial_{3}\bF\|_{\mL^{\infty}( \Omega_T \times \R^n \times \mI_{\omega,\rho})}\norm{\kappa}_{\mL^1(\R_{>0};\mTV(\R^n;\R^n))}\|J'\|_{\mL^{\infty}(\mI_\rho)}B^{\infty}}.
\]
In particular, we have
\[
\|q(t,\cdot)\|_{\mL^{\infty}(\R^n)}\leq B^{\infty}\quad, \forall t\in[0,T_{1}],
\]
and we have identified a time horizon on which the \(\mL^{\infty}\)-norm remains in the claimed box.

\item[The \(\mTV\)-bound:] We now turn to a \(\mTV\) bound on \(q\coloneqq \mF[\tilde{q}]\) for small time horizons. The \(\mTV\) estimate established in \cref{cor:specific_tailored}---specifically \cref{E:rho_star_w}---expresses such a bound in terms of the flux and of the function $w \in \mathcal Z(T)$. Substituting the estimates \cref{eq:nonlocal_estimates_fixed_point_memory} for the latter yields
\begin{align}
    \norm{\mF[\tilde q](t, \cdot)}_{\mTV(\R^n))}&\leq \abs{q_{\circ}(0, \cdot)}_{\mTV(\R^n)} \e^{\eta(B^{\infty}) t} + n W_n\mathrm{e}^{\eta(B^{\infty}) t}C\big(t,B^{\infty},B^{\mTV}\big),
    \label{eq:TV_bound_fixed_point_memory}
\end{align}
where $\eta(B^\infty)$ is given by
\begin{equation}\label{E:eta-B-infty}
\begin{split}
    \eta(B^{\infty})\coloneqq (2n+1)\Big(
        &\|\nabla\bF\|_{\mL^{\infty}( \Omega_T \times \R^n \times \mI_{\omega,\rho})}
        +\|\partial_{4}\bF\|_{\mL^{\infty}( \Omega_T \times \R^n \times \mI_{\omega,\rho})}\\
        &+ |\kappa|_{\mL^1(\R_{>0});\mTV(\R^n)}\|J'\|_{\mL^{\infty}(\mI_\rho)}B^{\infty}
        \|\partial_{3}\bF\|_{\mL^{\infty}( \Omega_T \times \R^n \times \mI_{\omega,\rho})}
    \Big),
\end{split}
\end{equation}
and, for any $t \in [0, T_1]$,
\begin{equation}
\begin{aligned}
C\big(t,B^{\infty},B^{\mTV}\big)\coloneqq&\|\nabla\diverg\bF\|_{\mL^{1}(\Omega_{T});\mL^{\infty}(\mI_\rho))}\\
        &\ +\Big(\norm{\kappa}_{\mL^1(\R_{\ge 0}; \mTV(\R^n;\R^n))}\|q_{\circ}(0, \cdot)\|_{\mL^{1}(\R^n)}\|\partial_{3}\diverg\bF\|_{\mL^{\infty}( \Omega_T \times \R^n \times \mI_{\omega,\rho})}\\
        &\ +\norm{\kappa}_{\mL^1(\R_{\ge0}; \mTV(\R^n;\R^n))}^{2}\|q_{\circ}(0, \cdot)\|_{\mL^{1}(\R^n)}\|J'\|_{\mL^{\infty}(\mI_\rho)}B^{\infty} \|\partial_{3}^{2}\bF\|_{\mL^{\infty}( \Omega_T \times \R^n \times \mI_{\omega,\rho})}\\
        &\ +\norm{\kappa}_{\mL^1(\R_{\ge 0}; \mTV(\R^n;\R^n))}B^{\mTV}\|\partial_{3}\bF\|_{\mL^{\infty}( \Omega_T \times \R^n \times \mI_{\omega,\rho})}\Big)\cdot  t  \|J'\|_{\mL^{\infty}(\mI_\rho)}
\end{aligned}
\label{eq:C_t_B_infty_TV_1_memory}
\end{equation}

if we assume \cref{item:7_flux}, first part or
\begin{equation}
\begin{aligned}
C\big(t,B^{\infty},B^{\mTV}\big)\coloneqq&\Big(\|q_{\circ}(0, \cdot)\|_{\mL^{1}(\R^n)}\|\partial_{3}\nabla\diverg\bF\|_{\mL^{\infty}( \Omega_T \times \R^n \times \mI_{\omega,\rho})}\\
        &\ +\Big(\norm{\kappa}_{\mL^1(\R_{\ge 0}; \mTV(\R^n;\R^n))}\|q_{\circ}(0, \cdot)\|_{\mL^{1}(\R^n)}\|\partial_{3}\diverg\bF\|_{\mL^{\infty}( \Omega_T \times \R^n \times \mI_{\omega,\rho})}\\
        &\ +\norm{\kappa}_{\mL^1(\R_{\ge 0}; \mTV(\R^n;\R^n))}^{2}\|q_{\circ}(0, \cdot)\|_{\mL^{1}(\R^n)}\|J'\|_{\mL^{\infty}(\mI_\rho)}B^{\infty} \|\partial_{3}^{2}\bF\|_{\mL^{\infty}( \Omega_T \times \R^n \times \mI_{\omega,\rho})}\\
        &\ +\norm{\kappa}_{\mL^1(\R_{\ge 0}; \mTV(\R^n;\R^n))}B^{\mTV}\|\partial_{3}\bF\|_{\mL^{\infty}( \Omega_T \times \R^n \times \mI_{\omega,\rho})}\Big)\cdot  t\|J'\|_{\mL^{\infty}(\mI_\rho)}
\end{aligned}
\label{eq:C_t_B_infty_TV_2_memory}
\end{equation}
in the case of \cref{item:7_flux}, second part.

By \cref{ass:general_nonlocal_memory,ass:flux_nonlocal}, the term $C\big(t,B^{\infty},B^{\mTV}\big)$ is finite by \cref{item:3_flux,item:4_flux,item:5_flux,item:7_flux}, and $\eta(B^{\infty})$ is finite by \cref{item:2_flux,item:3_flux}. Since $C\big(t,B^{\infty},B^{\mTV}\big)\to 0$ as $t\to 0$, we may choose $T_{2}\in(0,T_{1}]$ sufficiently small so that
\[
\abs{q_{\circ}(0, \cdot)}_{\mTV(\R^n)} \, \mathrm{e}^{\eta(B^{\infty}) t} + n W_n\mathrm{e}^{\eta(B^{\infty}) t}C\big(t,B^{\infty},B^{\mTV}\big)\leq 42|q_{\circ}(0, \cdot)|_{\mTV(\R^n;\R^n)}=B^{\mTV} \ \forall t\in[0,T_{2}]
\]
and hence 
\[
|\mF[\tilde{q}](t,\cdot)|_{\mTV(\R^n;\R^n)}\leq B^{\mTV}\ \forall t\in[0,T_{2}].
\]
\item[The \(\mL^{1}\)-bound:] This is a consequence of the fact that we deal with a conservation law. In particular, we consider the $\mL^1$ time continuity result in  \cref{theo:existence_uniqueness_less_regular_flux}.
Letting $\bPhi(t, \bx, q) = \bF(t, \bx, w(t, \bx), q)$, we have
\begin{equation}\label{E:conservation-bound_memory}
    \begin{split}
        \|\mF[\tilde q](t,\cdot)\|_{\mL^1(\R^n)}&\leq \norm{q_\circ(0, \cdot)}_{\mL^1(\R^n)} + \|\diverg\bF\|_{\mL^1(\Omega_T; \mL^{\infty}(\mathcal U^2))} + t \, D(T_1, B^\infty, B^{\mTV})
    \end{split}
\end{equation}
where, 
\begin{equation}
    \begin{split}
        D(T_1, B^\infty, B^{TV}) &\coloneqq \norm{J'}_{\mL^\infty(Q(T_1))} \abs{\kappa}_{\mL^1(\R_{>0};\mTV(\R^n))}  \Big( \norm{q_\circ(0,\cdot)}_{\mL^\infty(\R^n)} + B^{\mTV} \norm{\partial_3 F}_{\mL^\infty( \Omega_T \times \R^n \times \mI_{\omega,\rho})}\\
        & \qquad +\abs{\kappa}_{\mL^1(\R_{>0};\mTV(\R^n))} \norm{J'}_{\mL^\infty(\R^n)} \norm{q_\circ(0, \cdot)}_{\mL^1(\R^n)} B^\infty \norm{\partial_3^2 F}_{\mL^\infty( \Omega_T \times \R^n \times \mI_{\omega,\rho}})\Big)\\
        &\qquad+  \norm{\nabla F}_{\mL^\infty( \Omega_T \times \R^n \times \mI_{\omega,\rho})}) + \norm{\partial_4 F}_{\mL^\infty( \Omega_T \times \R^n \times \mI_{\omega,\rho})}\\
        &\qquad + \abs{\kappa}_{\mL^1(\R_{>0};\mTV(\R^n))} \norm{J'}_{\mL^\infty(Q(T_1))} B^\infty  \norm{\partial_3 F}_{\mL^\infty( \Omega_T \times \R^n \times \mI_{\omega,\rho})}.
    \end{split}
\end{equation}
Thanks to \cref{ass:flux_nonlocal}, the constant $D(t, B^\infty, B^{\mTV})$ is bounded and hence for sufficiently small $T$, we can set 
\begin{equation*}
     \norm{q_\circ(0,\cdot)}_{\mL^1(\R^n)} + \|\diverg\bF\|_{\mL^1(\Omega_T; \mL^{\infty}(\mI_{\omega,\rho}))} + t \, D(T_1, B^\infty, B^{TV}) \le 42 \norm{q_\circ(0, \cdot)}_{\mL^1(\R^n)}.
\end{equation*}
In particular, one can choose $T_2\in(0,T_1]$ as 
\[
0 < T_2 \le \frac{41 \norm{q_\circ(0, \cdot)}_{\mL^1(\R^n)} - \|\diverg\bF\|_{\mL^1(\Omega_{T_2}; \mL^{\infty}(\mI_{\omega,\rho}))}}{D(T_1, B^\infty, B^{\mTV})},
\]
which again is well-defined as $\|\diverg\bF\|_{\mL^1(\Omega_T; \mL^{\infty}(\mathcal U^2))}$ is monotonically decreasing as $T \to 0$. We have, therefore, shown that $\norm{\mF[\tilde q](t, \cdot)}_{\mL^1(\R^n)} \le 42\norm{q_\circ(0, \cdot)}_{\mL^1(\R^n)}$ over $[0, T_2]$. 
\end{description}
Altogether, we have shown that
\[
\mF[\mY(T_{2})]\subset\mY(T_{2}),
\]
meaning that \(\mF\) is a self-mapping. The proof is concluded.
\end{proof}


\begin{remark}[$\mL^1$ continuity in time.]
    The $\mL^1$-time continuity of $\mF[q]$ for and $q \in \mathcal Y(T)$, as in \cref{theo:existence_uniqueness_less_regular_flux} holds true immediately following the bound defined in \cref{E:conservation-bound}.
\end{remark}

\begin{lemma}[contraction of the fixed-point mapping in \(\mL^{1}(\R^n)\)]\label{lem:contraction_mapping_L_1_memory}
Let the mapping \(\mF\) be as in \cref{defi:fixed_point}. Let \cref{ass:general_nonlocal} hold. Then, there exists a \(T\in\R_{>0}\) such that for any \(q,\tilde{q}\in \mY(T)\) we have
 \[
    \|\mF[q]-\mF[\tilde{q}]\|_{\mC((-\infty,T];\mL^{1}(\R^n))}\leq \tfrac{1}{2}\|q-\tilde{q}\|_{\mC((-\infty,T];\mL^{1}(\R^n))}.
 \]
 In other words, \(\mF\) is a contraction in \(\mC\big([0,T];\mL^{1}(\R^n)\big)\).
\end{lemma}
\begin{proof}
    We first note that over $(-\infty, 0) \times \R^n$, by \cref{defn:fixedpoint_memory}, the claim trivially holds. Therefore, it is sufficient to prove the theorem on $[0, T] \times \R^n$. In this case, the proof directly follows from that of \cref{lem:contraction_mapping_L_1}.
\end{proof}
We can finally state the critical lemma of this section, establishing the existence of a fixed point.
\begin{lemma}[existence and uniqueness of a fixed point of $\mF$] \label{lem:well-posedness-fixedpoint} 
Let \cref{ass:flux_nonlocal} holds and let $\mF$ be defined as in \cref{defn:fixedpoint_memory}. Then there exists a time horizon $T \in \R_{>0}$ and a unique function 
\begin{equation*}
    q^* \in \mC\big(\R_{\le T}; \mL^1(\R^n)\big) \cap \mL^\infty(\R_{<T}; \mL^\infty(\R^n) \cap \mTV(\R^n))
\end{equation*}
such that $\mF[q^*] = q^*$ in $\mC\big((-\infty, T]; \mL^1(\R^n)\big)$ and $q^* \equiv q_\circ$ on $(t, \bx) \in (-\infty, 0) \times \R^n$.
\end{lemma}
\begin{proof}
    By \cref{defi:admissible_set_memory}, the set $\mathcal Y(T)$ is closed subset of the properly chosen Banach space under the $\mL^\infty$-norm topology. The result therefore, follows by \cref{lem:self_mapping_memory} and \cref{lem:contraction_mapping_L_1_memory}.    
\end{proof}
Now, collecting \cref{lem:well-posedness-fixedpoint}, together with all the results in \cref{lem:self_mapping_memory}, such as the proven bounds \cref{eq:TV_bound_fixed_point_memory}, \cref{E:conservation-bound_memory}, the general stability result of \cref{eq:contraction_mapping_the_final_frontier}, we have proven the following theorem.
\begin{theorem}
\label{theo:existence_uniqueness_nonlocal_memory_small_time_horizon}
    Let \cref{ass:flux_nonlocal} and \cref{ass:general_nonlocal_memory} holds. Then, there exists a time horizon $T \in \R_{>0}$ and a unique entropy solution  
    \begin{equation*}
        q \in \mC\big(\R_{\le T}; \mL^1(\R^n)\big) \cap \mL^\infty(\R_{<T}; \mL^\infty \cap \mTV(\R^n)), \quad q \equiv q_\circ \quad \textnormal{on $(-\infty, 0) \times \R^n$}
    \end{equation*}
    of \cref{E:model_nonlocal_memory} in the sense of \cref{defi:entropy_solution}. Furthermore, for any $t,s \in (-\infty, T]$, we have that 
\begin{equation}
\label{E:q_fixed_point_solution_memory}
\begin{aligned}
\|q(t,\cdot)\|_{\mL^{\infty}(\R^n)}
    &\leq 42\|q_{\circ}(0, \cdot)\|_{\mL^{\infty}(\R^n)}\\
|q(t,\cdot)|_{\mTV(\R^n)}
    &\leq \abs{q_{\circ}(0, \cdot)}_{\mTV(\R^n)} \e^{\eta(B^{\infty}) t} + n W_n\mathrm{e}^{\eta(B^{\infty}) t}C\big(t,B^{\infty},B^{\mTV}\big),\\
\|q(t,\cdot)-q(s,\cdot)\|_{\mL^{1}(\R^n)}
    &\leq C^{\infty}(t,s) + \abs{t -s}\, D(T_1, B^\infty, B^{\mTV})
\end{aligned}
\end{equation}
where $C^{\infty}(t,s)\coloneqq \int_{\min\{s,t\}}^{\max\{s,t\}}\int_{\R^n} \|\diverg\bF(\tau, \bx, *)\|_{\mL^{\infty}(\mI(B^\infty)))} \dd \bx \dd \tau $ for \((s,t)\in [0,T]^{2}\), which vanishes as $\abs{t -s } \to 0$, and $\eta(B^\infty)$ is defined as in \cref{E:eta-B-infty}. 

Furthermore, let $q$ and $\tilde q$ be solutions of \cref{E:model_nonlocal_memory} with respect to the initial conditions $q_0$ and $\tilde{q}_0$, respectively. In addition, let $\bPhi$, $\hat{\bPhi}$ be the fluxes defined at \cref{fluxes}. Then, we have the following general stability result
\begin{equation}
\begin{split}
    &\|q(t,\cdot)-\tilde{q}(t,\cdot)\|_{\mL^{1}(\R^n)}\\
    &\leq t\e^{\kappa_{\bF}t}|q_{\circ}|_{\mTV(\R^n)}
        \esssup_{(t,\bx,q)\in\Omega_T\times\mI_\rho}
        \big|\partial_{4}\bPhi(t,\bx,q)
            -\partial_{4}\hat{\bPhi}(t,\bx,q)\big|\\
    &\quad +nW_{n}
        \esssup_{(t,\bx,q)\in\Omega_T\times\mI_\rho}
        \big|\partial_{4}\bPhi(t,\bx,q)
            -\partial_{4}\hat{\bPhi}(t,\bx,q)\big|
        \int_{0}^{t}s\e^{\kappa_{\bF}(t-s)}
        \|\nabla\diverg\bPhi(s,\cdot,\ast)\|_{\mL^{1}(\R^n;\mL^{\infty}(\mI_\rho))}\dd s\\
    &\quad +\int_{0}^{t}\e^{\kappa_{\bF}(t-s)}
        \|\diverg\bPhi(s,\ast,\cdot)
            -\diverg\hat{\bPhi}(s,\ast,\cdot)\|_{\mL^{1}(\R^n;\mL^{\infty}(\mI_\rho))}\dd s.
\end{split}
\end{equation}
\end{theorem}

\subsection{Systems of nonlocal balance laws with nonlinear locality}
The previous results can readily be extended to systems in which the coupling is only in the nonlocal operator and not in the local nonlinearity. We devote this short section to state the existence and uniqueness theorem for such models, without being too detailed in the related proofs.

We begin by defining the model of interest. Let \(N\in\N_{\geq1}\) be given; we consider the following system of nonlocal conservation laws
\begin{align}
\partial_{t}\bq^{k}&=-\diverg\bbF^{k}\big(t,\bx,\mbW[\bJ^{k},\bq,\bgamma^{k}](t,\bx),\bq^{k}(t,\bx)\big), && k\in\{1,\ldots,N\},\ (t,\bx)\in (0,T)\times\R^n,\label{eq:system}\\
    \bq(0,\cdot)&=\bq_{\circ}(\cdot), && \text{ on } \R^n, \nonumber
    \intertext{supplemented with a vectorized nonlocal operator for \(k\in\{1,\ldots,n\}\),}
\Big(\mbW[\bJ,\bq,\bgamma^{k}]\Big)_{l}&\coloneqq\mW[\bJ^{k}_{l}(\bq_{l}),\bgamma^{k}_{l}]&& l\in\{1,\ldots,N\} \text{ on } (0,T)\times\R^n,\nonumber
\end{align}
with \(\mW\) as in \cref{E:model_nonlocal}. We observe that the equations are coupled together solely through the nonlocal operator, reason why we talk about \emph{local nonlinearity}. 

In line with the previous analysis, we make the following assumptions on the involved functions.

\begin{assumption}
\label{ass:system}
We assume the following.
    \begin{description}
    \item[Initial datum:] \(\bq_{\circ}\in\mBV(\R^n;\R^n)\cap\mL^{\infty}(\R^n;\R^n)\).
    \item[Nonlocal kernel:] \(\bgamma\in\mBV\big(\R^n;\R_{\geq0}^{N\times N}\big)\cap\mL^{\infty}\big(\R^n;\R^{N\times N}\big)\) so that \(\|\bgamma_{l}^{k}\|_{\mL^{1}(\R^n)}= 1\),\ \((k,l)\in\{1,\ldots,N\}^{2}\).
    \item[Nonlocal nonlinearity:] \(\bJ\in\mW^{1,\infty}_\textnormal{loc}\big(\R;\R^{N\times N}\big)\) so that \(\bJ(0)=\boldsymbol{0}\).
    \item[Flux:]\(\bF^{k}\) satisfies, 
     for all open and bounded \(\mU\subset\R\),
        \begin{enumerate}[leftmargin=-10pt]
          \item \(\bF^{k}\in\mL^{\infty}((0,T);\mL^{\infty}(\R^n\times\mU^{N+2}))\),
          \item \(\nabla\bF^{k}\in\mL^{\infty}((0,T)\times\R^n\times\mU^{N};\mathcal{W}^{1, \infty}(\mU))\),
         \item \(\partial_{\boldsymbol{3}}\bF^{k}\in \mL^{\infty}((0,T)\times\R^n\times\mU^{N};\mathcal{W}^{1, \infty}(\mU))\),
         \item \(\partial_{\boldsymbol{3}}\nabla\bF^{k}\in \mL^{\infty}((0,T)\times\R^n\times\mU^{N+2})\),
         \item \(\partial_{\boldsymbol{3}}^{2}\bF^{k}\in \mL^{\infty}((0,T)\times\R^n\times\mU^{N+2})\),
    \item \(\diverg\bF^{k}\in\mL^{1}((0,T)\times\R^n;\mL^{\infty}(\mU^{N+2}))\ \text{ or there is \(l\in\{1,\ldots,N\}\) such that}\
    \diverg\bF^{k}(\cdot,\ast,\ldots,0,\ldots,\star)\equiv 0\), where the \(0\) evaluation is meant at the \((1+n+l)\)-th component of \(\bF^{k}\),
\item \(\nabla\diverg\bF^{k}\in\mL^{1}((0,T)\times\R^n;\mL^{\infty}(\mU^{N+2}))\)
   or there is \(l\in\{1,\ldots,N\}\) such that
    \(\nabla\diverg\bF^{k}(\cdot,\ast,\ldots,0,\ldots,\star)\equiv 0 \wedge 
    \partial_{\boldsymbol{3}}\nabla\diverg\bF\in\mL^{\infty}((0,T)\times\R^n\times\mU^{N+2})\), where the \(0\) evaluation is at the \((1+n+l)\)-th component of \(\bF^{k}\).
    \end{enumerate}
\end{description}
\end{assumption}
\begin{remark}[similarity of \cref{eq:system} to the scalar case]
The system case considered in \cref{defi:bPhi_delta} is similar in its assumptions to \cref{ass:general_nonlocal}, in which the nonlocal dependency is now not on a scalar quantity but a vectorial one. It is, however, crucial in this class of models for different \(\bq^{k}\) to be coupled only via the nonlocal quantity, not the local one; otherwise, significant difficulties arise (see, for instance, \cite{Bressan2000}). This is also why the entropy condition for the scalar case in \cref{eq:nonlocal_explicit_w} can be adjusted to the system case, in which the entropy inequality acts on each \(\bq^{k},\ k\in\{1,\ldots,N\}\) separately. We do not detail this further; we leave the proper alignment of an entropy condition to the reader.
\end{remark}
With the assumptions and the model class stated, we can present the existence and uniqueness result for the system case.
\begin{theorem}[existence and uniqueness for \cref{eq:system} of weak entropy solutions]
Let \cref{ass:system} hold. 
Then, there exists a time horizon \(T\in\R_{>0}\) and a unique weak entropy solution
\[
\bq\in\mC\big([0,T];\mL^{1}(\R^n;\R^n)\big)\cap \mL^{\infty}\big((0,T);\mL^{\infty}\cap\mTV(\R^n;\R^n)\big).
\]
of \cref{eq:system}.
\end{theorem}
\begin{proof}
    This is a consequence of the proof of the scalar case in \cref{theo:existence_uniqueness_nonlocal_small_time_horizon} when considering the fixed-point mapping in the corresponding product space and the contraction in \(\mC([0,T];\mL^{1}(\R^n;\R^n))\).
\end{proof}

\begin{remark}[extension to longer time horizons]
There is an analogous result to \cref{cor:large_time_horizon} for systems and long time horizons. Adjusting the condition in \cref{eq:F_maximum_principle} to
\[
\exists \bq^{\min},\bq^{\max}\in\R^n: \bq^{\min}\leq \bq^{\max} 
\]
such that \[
\bq^{\min}\leq \bq_{\circ}\leq \bq^{\max}
\]
and assuming that for every \(k\in\{1,\ldots,N\}\), it holds that
\[
\bF^{k}(\cdot,\ast,\star,\bq^{\min}_{k})\equiv\boldsymbol{0}\equiv\bF^{k}(\cdot,\ast,\star,\bq^{\max}_{k}),
\]
then it is possible to obtain the existence and uniqueness of systems in \cref{eq:system} on any finite time horizon.   
\end{remark}

\subsection{Nonlocal conservation laws with time delay in the nonlocal quantity and nonlinearity in the local quantity}\label{subsec:delay}

Another straightforward extension of our work with local and nonlocal-in-time equations is the one to delayed nonlocal conservation laws.
Here, we briefly sketch how the previously obtained results, particularly in \cref{sec:stability}, can be applied to this setting.

The results of this subsection are directly applicable to the traffic flow models with delay, as in \cite{ciaramaglia2025multi,ciaramaglia2024nonlocaltrafficflowmodels} under weaker assumptions on the involved objects, particularly on the regularity of the nonlocal kernels. It is important to note that the delay term we consider is only present in the nonlocal quantity, and not in the local quantity. Indeed, the latter situation typically leads to equations that are not well-defined (see, e.g., \cite[Section 1.2]{Keimer2019}). 

In more detail, we consider the following equation 
\begin{equation}
\label{eq:delay}
\begin{aligned}
    \partial_{t}q +\diverg\bbF\big(t,\bx,\mW[J(q),\gamma](t-\delta,\bx),q\big)&=0 && (t,\bx)\in(0,T)\times\R^n\\
    q(t,\bx)&=q_{\circ}(t,\bx) && (t,\bx)\in (-\delta,0]\times\R^n
\end{aligned}
\end{equation}
In line with the previous analysis, the assumptions on the flux, nonlocality and initial datum read as follows.

\begin{assumption}
\label{ass:delay}
We make the following assumptions.
    \begin{description}
        \item[Initial datum:] \(q_{\circ}\in \mC\big([-\delta,0];\mL^{1}(\R^n)\big)\cap \mL^{\infty}\big((-\delta,0);\mBV\cap\mL^{\infty}(\R^n)\big)\).
        \item[Nonlocal kernel:] \(\gamma\in\mBV(\R^n)\cap\mL^{\infty}(\R^n)\).
        \item[Nonlocal nonlinearity:] \(J\in\mW^{1,\infty}_{\textnormal{loc}}(\R)\), so that \(J(0)=0\).
        \item[Flux:] The same assumptions about the flux as in \cref{ass:general_nonlocal_memory}.
    \end{description}
\end{assumption}

Given this, we can present our existence and uniqueness results. Once again, we begin with the wellposedness on a short time horizon. 
\begin{theorem}[existence and uniqueness for \cref{eq:delay} of weak entropy solutions]
\label{theo:existence_uniqueness_delay}
    Let \(\delta\in\R_{>0}\) be given, and let \cref{ass:delay} hold. Then there eists a time horizon $T \in \R_{>0}$ and a weak entropy solution    
    \[
q\in \mC\big([0,T^{*}];\mL^{1}(\R^n)\big)\cap \mL^{\infty}\big((0,T^{*});\mL^{\infty}\cap\mBV(\R^n)\big)
    \]
    of \cref{eq:delay}.
\end{theorem}

\begin{proof}
    We construct the solution by marching through time slices $[0,\delta]$, $[\delta, 2\delta]$, and so on, as long as the $\mL^\infty$-norm of the solution remains finite.

    On the first interval $t\in[0,\delta]$, the delay shifts the nonlocal term to $\mW[J(q_\circ),\gamma](t-\delta,\bx)$, which depends only on the prescribed initial datum $q_\circ$ and not on the unknown $q$. The problem therefore reduces to
    \begin{align*}
\partial_{t}q+\diverg\bbF\big(t,\bx,\mW[J(q_{\circ}),\gamma](t-\delta,\bx),q\big)&=0,&& (t,\bx)\in(0,\delta)\times\R^n,\\
q(0,\bx)&=q_{\circ}(0,\bx),&& \text{ on }\R^n,
    \end{align*}
    which is a conservation law with an explicitly given, space- and time-dependent coefficient. By the regularity assumption on $q_\circ$ and the properties of $\mW$, the hypotheses of \cref{cor:specific_tailored} are satisfied, yielding a unique weak entropy solution on some time horizon $T^* \leq \delta$.

    Provided the $\mL^\infty$-norm of $q$ has not blown up at $t = \delta$, the solution evaluated there is $\mL^1$, essentially bounded, and of bounded total variation. We may then set $\tilde{q}(t,\bx) \coloneqq q(t+\delta,\bx)$ and consider the shifted problem
    \begin{align*}
\partial_{t}\tilde{q}+\diverg\bbF\big(t+\delta,\bx,\mW[J(q),\gamma](t,\bx),\tilde{q}\big)&=0,&& (t,\bx)\in(0,\delta)\times\R^n,\\
\tilde{q}(0,\bx)&=q(\delta,\bx),&& \text{ on }\R^n,
    \end{align*}
    where now the nonlocal term depends on the solution $q$ already obtained on $[0,\delta]$, and is therefore again explicitly known. Applying \cref{cor:specific_tailored} once more yields a unique weak entropy solution on a further time horizon. This procedure is then iterated, extending the solution one slice at a time, until either the prescribed final time $T$ is reached or the $\mL^\infty$-norm of the solution blows up.
\end{proof}

\begin{remark}[extension to longer time horizons]
There is, once more, an analogous result to \cref{cor:large_time_horizon} that only takes advantage of the local properties of the flux. Assuming in addition to \cref{ass:delay} that
 there exist \(q_{\min},q_{\max}\in\R\) such that \(q_{\min}\leq q_{\max}\), 
    \begin{equation}
    \bF(\cdot,\ast,\star,q_{\min})=\boldsymbol{0}=\bF(\cdot,\ast,\star,q_{\max}),
    \end{equation}
    and 
    \[
        q_{\min}\leq q_{\circ}(0,\cdot)\leq q_{\max} \qquad\text{on } \R^n \text{ a.e.,}
    \]
    then the solution to the nonlocal conservation law with delay \(q\) exists on any finite time horizon \(T\in\R_{>0}\) and satisfies the following maximum principle:
    \[
    q_{\min}\leq q(t,\bx)\leq q_{\max} \quad \text{for }(t,\bx)\in (0,T)\times\R^n \text{ a.e.}
    \]
\end{remark}

\begin{remark}[memory and delay]
    In the memory setting of \cref{subsec:memory}, the nonlocal operator integrates up to the current time, whereas in the delay case it only involves the density strictly in the past. If the time-dependent kernel in \cref{subsec:memory} were supported away from the current time---an admissible choice within our framework---the existence and uniqueness argument would reduce to the slicing procedure described in \cref{theo:existence_uniqueness_delay}.
\end{remark}

\begin{remark}[existence and uniqueness of nonlocal balance laws with nonlinearity in the locality]\label{rem:balance_laws}
The theory developed in \cref{sec:nonlocal_conservation_laws} concerns conservation laws, but can be extended with minor modifications to balance laws, that is, to
\begin{equation}
    \begin{aligned}    
\partial_t q +\diverg_{\bx}\big(\bF\big(t, \bx,\mW[J(q),\gamma](t,\bx), q(t,\bx)\big)\big)&=h\big(t,\bx,q,\mW[J(q),\gamma](t,\bx)\big), &&(t, \bx) \in (0, T) \times \R^n,\\
\mW[J(q),\gamma](t,\bx)&=\int_{\R^n}\gamma(\bx-\by)J(q(t,\by))\dd\by,&& (t,\bx)\in (0,T)\times\R^n,\\
q(0, \cdot) &\equiv q_{\circ},&& \bx\in\R^n,
\end{aligned}
\end{equation}
with $h:(0,T)\times\R\times\R^{2}\rightarrow\R$ satisfying suitable regularity conditions. This includes the settings with memory and delay treated above. We refer the reader to \cite{Mercier2011improved} for the required estimates and assumptions on $h$.
\end{remark}

\section{Numerical simulations}\label{sec:numerics}

We conclude this work by presenting numerical simulations with two complementary goals. The first is to provide numerical evidence supporting the theoretical framework developed in the preceding sections. The second is to explore the qualitative behavior introduced by memory effects in nonlocal conservation laws.

We stress that a rigorous numerical analysis of the scheme, including convergence rates, efficiency, and error estimates, is beyond the scope of this paper and is deferred to future work. However, our contraction result provides theoretical grounding for the iterative numerical procedure we employ, and the simulations below confirm its practical viability. The code is publicly available at \url{https://github.com/ALive95/CLAWS}.

\subsection{Numerical scheme}

A key advantage of the fixed-point framework developed above is that it naturally decouples the nonlocal (and possibly memory) term from the local conservation law dynamics. At each iteration, the nonlocal operator $\mW$ is evaluated on the current iterate and then frozen, reducing the problem to a standard conservation law with a space- and time-dependent flux. This means that any classical numerical scheme for local conservation laws can be applied directly, without the need to design a solver tailored to handle memory. For concreteness, we employ a Lax--Friedrichs and a Godunov scheme, depending on the experiment.

We discretize the domain $[0,T]\times [x_L,x_R]$ into $N_t$ time steps of size $\Delta t$ and $N_x$ cells of size $\Delta x$, with cell centers $x_i$ and time levels $t^n$. The Courant--Friedrichs--Lewy (CFL) condition $\alpha\,\Delta t/\Delta x \leq 1$ is enforced throughout, where $\alpha$ bounds the wave speed of the frozen flux. At each time step $t^n$, the nonlocal term $\mW_i^n \approx \mW[J(q^n),\gamma](t^n,x_i)$ is approximated via numerical quadrature (direct convolution, or FFT-based circular convolution for periodic boundary conditions; see below for an in-depth explanation).

With $\mW_i^n$ frozen, the numerical scheme of choice defines a discrete solution operator $\mQ_h[\mW^n] \approx \mQ[\mW]$, which maps a frozen discrete nonlocal field $\mW^n = (\mW_i^n)_i$ to the updated density $(q_i^{n+1})_i$. The discrete counterpart of the fixed-point operator $\mF$ in \cref{defi:fixed_point} is then
\begin{equation}
    \mF_h[q^n] \coloneqq \mQ_h\bigl[\mW_h[J(q^n),\gamma]\bigr],
    \label{eq:discrete_fp}
\end{equation}
where $\mW_h$ denotes the discrete nonlocal operator.

The full nonlocal scheme is implemented as a Picard iteration. Given an iterate $q^{(k)}$, we compute $\mW[J(q^{(k)}),\gamma]$ and use the numerical scheme to obtain $q^{(k+1)}=\mF_h[q^{(k)}]$. Iteration is terminated when
\[
    \|q^{(k+1)}-q^{(k)}\|_{\infty} < \texttt{tol},
\]
with \texttt{tol}$=10^{-7}$ in all experiments below. Convergence to a unique fixed point is guaranteed by \cref{lem:contraction_mapping_L_1} on a sufficiently short time horizon, and the numerical experiments confirm rapid convergence (typically within $5$--$10$ iterations per time step).

\begin{remark}[on the memory kernel computation]

For conservation laws with memory \tup{see \cref{E:model_nonlocal_memory}}, the nonlocal term involves integration over the full history:
\[
    \mW[J(q),\kappa](t,x) = \int_0^t\int_{\R}\kappa(t-s,x-y)J(q(s,y))\dd y\dd s.
\]
This is discretized as a causal quadrature sum,
\[
    \mW_i^n = \sum_{k<n}\Delta t\sum_j\Delta x\,\kappa(t^n-t^k,x_i-x_j)J(q^k_j),
\]
where the summation is performed strictly over past time levels $k<n$ to preserve the causal structure of the fixed-point formulation. For each fixed lag $\tau = t^n - t^k$, the inner sum is a discrete spatial convolution in $i$, requiring $\mathcal{O}(N_x^2)$ operations in general. Under periodic boundary conditions, this reduces to $\mathcal{O}(N_x\log N_x)$ via FFT-based circular convolution \cite{HAIRER198887}.
Accumulating over all past time levels shows that the overall cost per time step is $\mathcal{O}(N_t N_x \log N_x)$, and $\mathcal{O}(N_t^2 N_x \log N_x)$ over the full simulation.

    For separable kernels that are exponential in time, namely,
\[
    \kappa(\tau,z)=\tfrac{1}{\tau_0}\e^{-\tau/\tau_0}\gamma(z),
\]
it is possible to exploit the separability to accelerate the computations. Indeed, applying the elementary identity $\e^{-(t^n-t^k)/\tau_0} = \e^{-\Delta t/\tau_0}\cdot \e^{-(t^{n-1}-t^k)/\tau_0}$ collapses the full temporal sum into the recursive update \cite[Chapter 10]{SimoHughes1998}
\[
\begin{dcases}
    S^{n+1}_i = \e^{-\Delta t/\tau_0}S^n_i + \Delta t\sum_j\Delta x\,\gamma(x_i-x_j)J(q^n_j),\\
    \mW_i^n = S^n_i,
\end{dcases}
\]
reducing the temporal cost per step from $\mathcal{O}(N_t)$ to $\mathcal{O}(1)$ and the total complexity to $\mathcal{O}(N_t N_x \log N_x)$.
\end{remark}

\subsection{Example: LWR model with nonlocal flux and memory}

We begin by considering the classical one-dimensional LWR-type traffic model
\begin{equation}
    \partial_t q + \partial_x\Big(q(1-q)V_{\max}\big(1-\mW[J(q),\gamma](t,x)\big)\Big) = 0,\quad (t,x)\in(0,T)\times(0,L),
    \label{eq:LWR_numerical}
\end{equation}
with $V_{\max}=1$, $J(q)=q$, $L=4$, $T=5$, and outflow boundary conditions. The initial datum consists of two rectangular plateaus,
\[
    q_0(x) = 0.8\,\mathbf{1}_{[0.5,1.2]}(x) + 0.7\,\mathbf{1}_{[2.2,2.9]}(x),
\]
generating two shocks whose interaction is influenced by the nonlocal term. The spatial kernel $\gamma$ is a one-sided raised cosine supported on $[-R,0]$ with $R=0.4$,
\[
    \gamma(z) = \tfrac{2}{R}\cos^2\!\left(\tfrac{\pi z}{2R}\right)\mathbf{1}_{[-R,0]}(z),
\]
reflecting the fact that drivers react to traffic ahead of them. We consider four configurations for the nonlocal operator $\mW$, namely, pure spatial nonlocality and three memory kernels---exponential, Erlang, and triangular in time---as follows:
\[
    K_{\mathrm{exp}}(\tau) = \e^{-\tau}, \qquad
    K_{\mathrm{Erl}}(\tau) = \tau\,\e^{-\tau}, \qquad
    K_{\mathrm{tri}}(\tau) = 2(1-\tau)\,\mathbf{1}_{[0,1]}(\tau).
\]
For the memory cases, the full nonlocal term reads
\[
    \mW[J(q),\kappa](t,x) = \int_0^t \int_{\R} K(t-s)\gamma(x-y)\,q(s,y)\dd y\dd s,
\]
and the historical datum is $q_{\mathrm{hist}}(t,x) = q_0(x)\,\e^{t}$ for $t\leq 0$, which decays to zero as $t\to-\infty$ and matches $q_0$ at $t=0$. 
The model components are visualized in \cref{fig:components}.

\begin{figure}[t!]
    \centering
    \includegraphics[width=\linewidth]{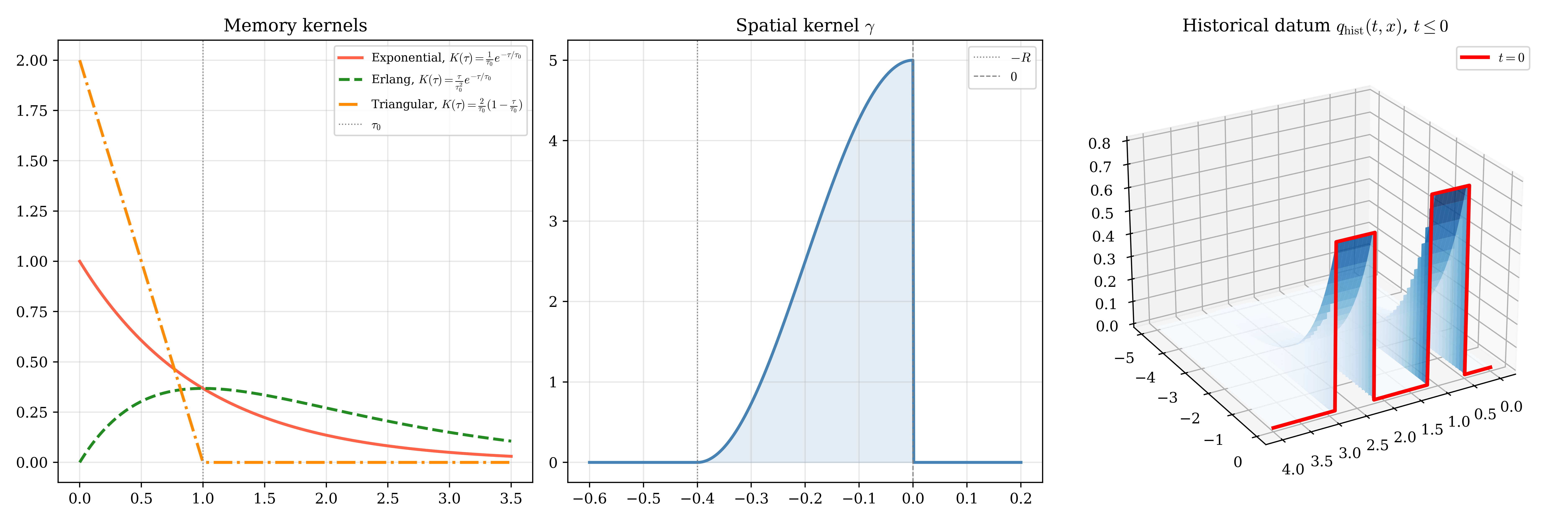}
   \caption{Model components for \eqref{eq:LWR_numerical}. Left: temporal memory kernels $K(\tau)$ (exponential, Erlang, triangular). Center: one-sided spatial kernel $\gamma(z)$ supported on $[-R,0]$, reflecting drivers' anticipation of traffic ahead. Right: historical datum $q_{\mathrm{hist}}(t,x)=q_0(x)\e^{t}$ for $t\leq 0$; the red profile marks $t=0$.}
\label{fig:components}
\end{figure}

For this experiment, we adopt the Godunov scheme, whose update reads
\begin{equation}
    q_i^{n+1} = q_i^n - \tfrac{\Delta t}{\Delta x}\Bigl(G_{i+1/2}^n - G_{i-1/2}^n\Bigr),
    \label{eq:Godunov}
\end{equation}
where the numerical flux at interface $x_{i+1/2}$ is given by the exact Riemann solution
\[
    G_{i+1/2}^n \coloneqq
    \begin{cases}
        \min_{u\in[q_i^n,\,q_{i+1}^n]} F(t^n,x_{i+1/2},\mW_{i+1/2}^n,u) & \text{ if } q_i^n \leq q_{i+1}^n,\\[4pt]
        \max_{u\in[q_{i+1}^n,\,q_i^n]} F(t^n,x_{i+1/2},\mW_{i+1/2}^n,u) & \text{ if } q_i^n > q_{i+1}^n,
    \end{cases}
\]
with $\mW_{i+1/2}^n = \tfrac{1}{2}(\mW_i^n + \mW_{i+1}^n)$. For the concave flux of \eqref{eq:LWR_numerical}, the optimizer is attained at $q^*=\tfrac{1}{2}$ regardless of the frozen values $(t,x,w)$, so each minimization reduces to at most three function evaluations.

\Cref{fig:basic_density} shows the evolution of the density $q$ on the whole space-time domain, while \Cref{fig:basic_snapshots} plots the density profile at several instants for all four cases. The memory kernels induce dynamics markedly different from those in the purely spatial case. One notable feature is that the shocks appear to propagate faster: this is likely due to the fact that $q_{\mathrm{hist}}(t,x) \leq q_0(x)$ for $t\leq 0$, so the memory integral $\mW$ is generally smaller than its purely spatial counterpart, resulting in a higher effective speed. A second distinctive feature is the formation of a kink between the two density maxima, which is absent in the spatial case: the passage of the first bump leaves a trace in $\mW$ that modifies the local speed in the intermediate region.

\begin{figure}[h!]
    \centering
    \includegraphics[width=0.95\textwidth]{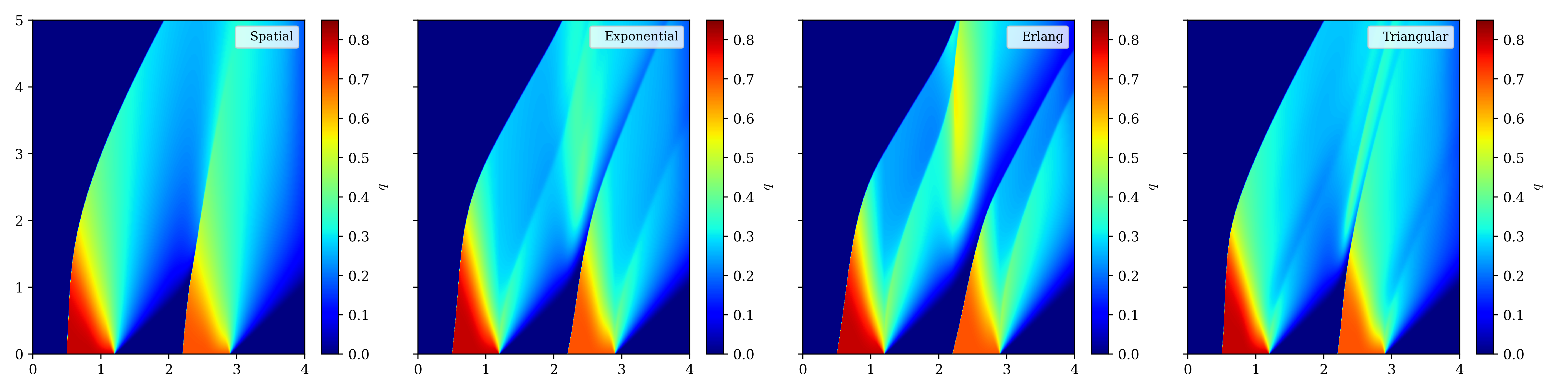}
    \caption{Full density evolution for the spatial case and the three memory kernels for \eqref{eq:LWR_numerical}.}
    \label{fig:basic_density}
\end{figure}

\begin{figure}[h!]
    \centering
    \includegraphics[width=0.8\textwidth]{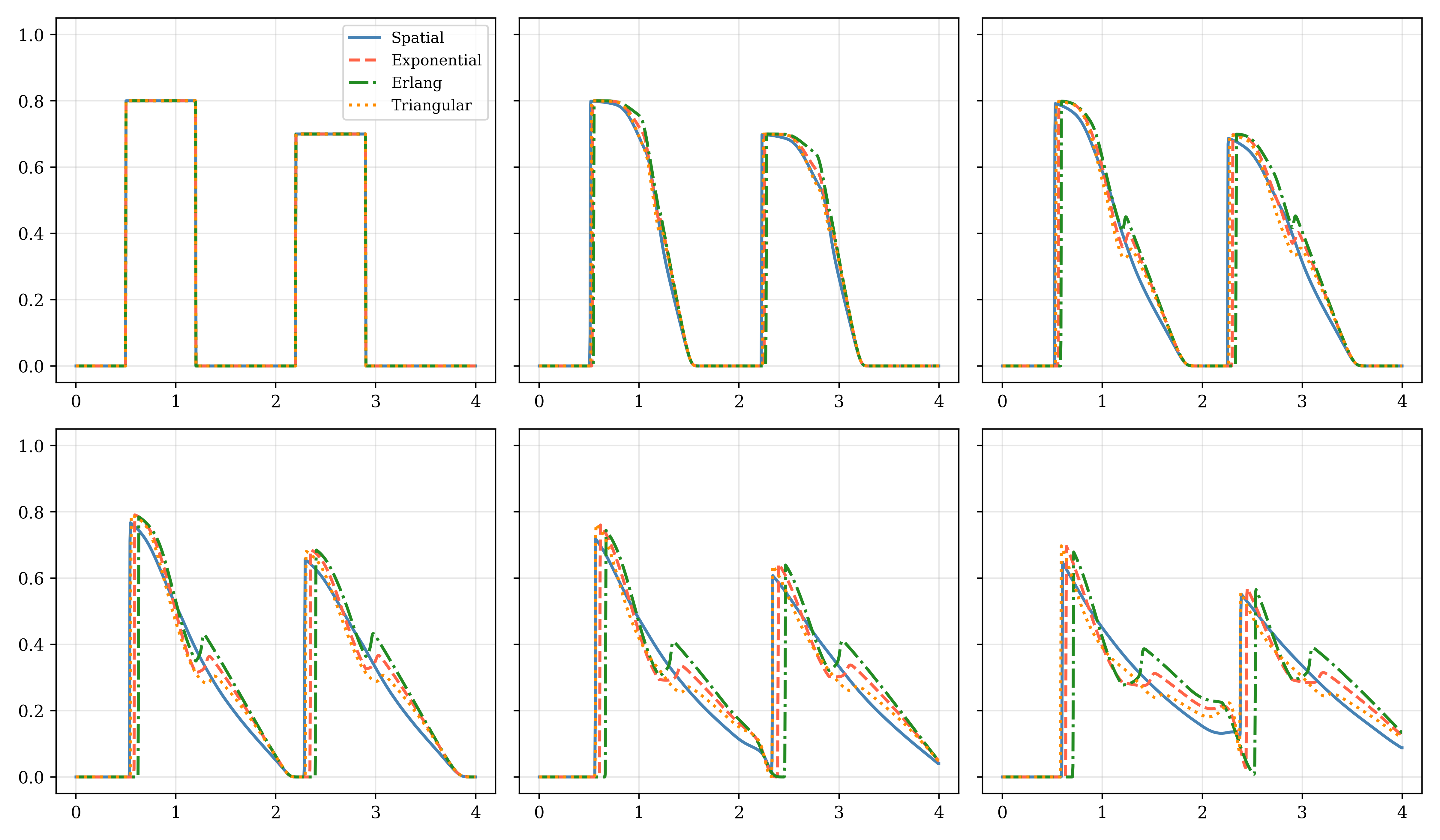}
    \caption{Density snapshots at times $t\in\{0,0.3,0.6,0.9,1.2,1.5\}$ for the spatial case and the three memory kernels for \eqref{eq:LWR_numerical}.}
    \label{fig:basic_snapshots}
\end{figure}

\subsection{Validation: Chiarello--Goatin experiment}

As a further demonstration, we replicate the numerical experiment from \cite[Section~4]{chiarello2018global}. The model in this experiment is a nonlocal LWR equation on the periodic domain $(-1,1)$ and $(0,0.5)$ in time, with constant initial datum $\rho_0=0.6$:
\begin{equation}
    \partial_t\rho + \partial_x\Big(V_{\max}(t,x)\,\rho(1-\rho)\,v\big(\mW[J(\rho),w_{\eta,\delta}](t,x)\big)\Big) = 0, \quad (t,x)\in(0,0.5)\times(-1,1),
    \label{eq:Goatin}
\end{equation}
where $J(\rho)=\rho$, $v(\rho)=(1-\rho)^{m-1}(1+\rho)^m$, and
\[
    \mW[J(\rho),w_{\eta,\delta}](t,x) = \int_{-1}^{1}w_{\eta,\delta}(x-y)\,\rho(t,y)\dd y.
\]
Here, $w_{\eta,\delta}$ is a quintic kernel of support radius $\eta$ shifted by $\delta$ (see \cite[eq.~(4.8)]{chiarello2018global}), and $V_{\max}(t,x)$ is a Gaussian-smoothed piecewise constant speed limit as defined in \cite[eq.~(4.6)]{chiarello2018global}. This example is of particular interest because the flux depends explicitly on both space and time through $V_{\max}(t,x)$, illustrating the full generality of the framework of \cref{sec:related-work}. We use $N_x=2000$ cells and $\mathrm{CFL}=0.9$. Following \citep{chiarello2018global}, in this case we adopt the Lax-Friedrichs scheme. Thus, the solution update reads
\begin{equation}
    q_i^{n+1} = \tfrac{q_{i+1}^n+q_{i-1}^n}{2} - \tfrac{\Delta t}{2\Delta x}\Big(F(t^n,x_{i+1/2},\mW_{i+1}^n,q_{i+1}^n)-F(t^n,x_{i-1/2},\mW_{i-1}^n,q_{i-1}^n)\Big).
    \label{eq:Lax--Friedrichs}
\end{equation}
We validate the fixed-point solver by comparing it against a direct explicit Lax--Friedrichs scheme for the parameter set $m=3$, $\eta=0.1$, $\delta=0.06$. The direct scheme is exactly the one described in \cite{chiarello2018global}: it advances the solution one step at a time without any Picard iteration, using the density at the current step to evaluate the nonlocal term. \Cref{fig:Goatin_fp_vs_direct_A} shows the space--time density produced by each method. The two solutions agree closely, confirming that the fixed-point iteration converges to the correct discrete solution.

\begin{figure}[h!]
    \centering
    \includegraphics[width=0.7\textwidth]{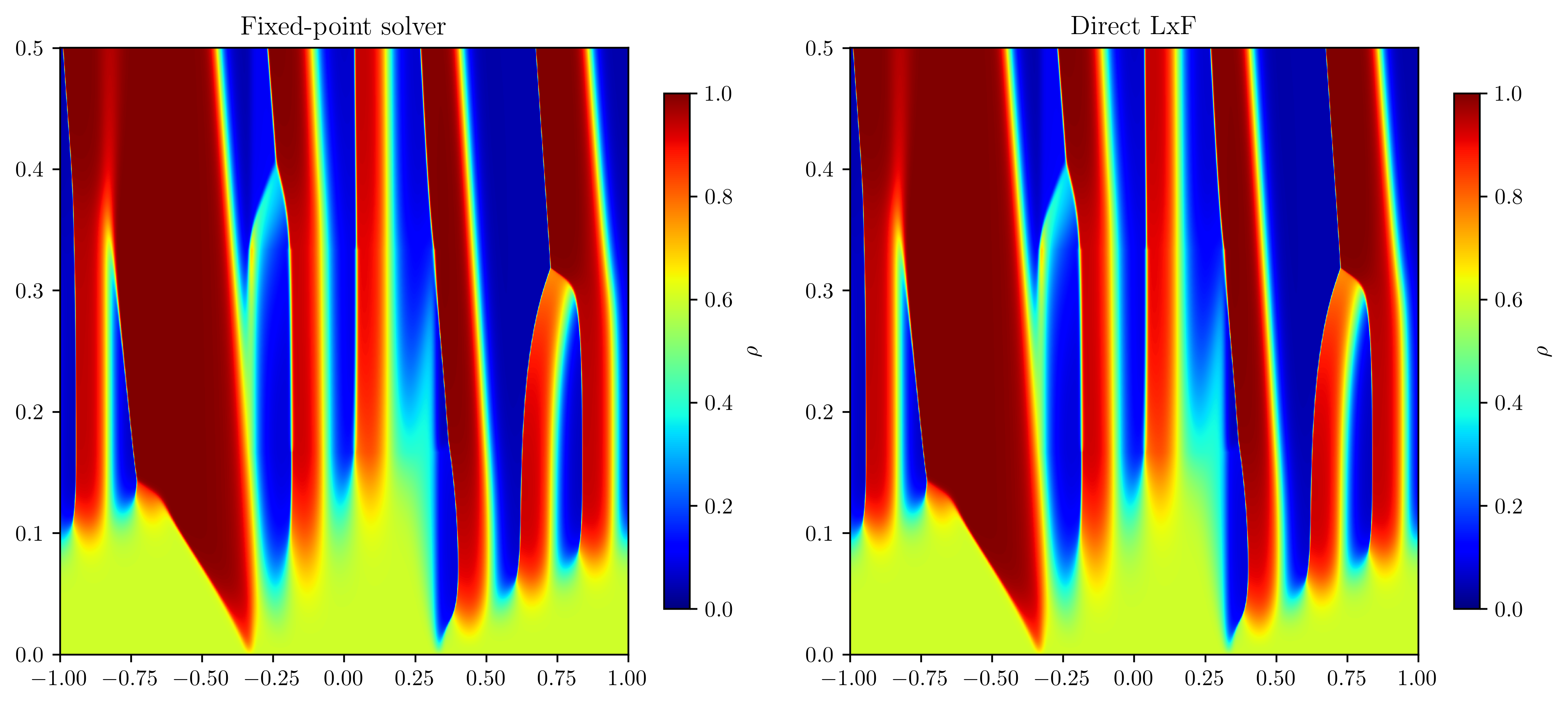}
    \caption{Fixed-point solver (left) vs.\ direct Lax--Friedrichs (center) and pointwise error (right) for \eqref{eq:Goatin} with $m=3$, $\eta=0.1$, $\delta=0.06$.}
    \label{fig:Goatin_fp_vs_direct_A}
\end{figure}

Having validated the scheme, we turn to the effect of memory. We augment \eqref{eq:Goatin} with a factorized memory kernel $\kappa(\tau,z)=K(\tau)\,w_{\eta,\delta}(z)$, considering the same three choices for the temporal part $K$ as in the LWR experiment: exponential, Erlang, and triangular, all with the same time scale $\tau_0$. The spatial kernel $w_{\eta,\delta}$ is unchanged. For the parameter set $m=3$, $\eta=1.0$, $\delta=0$, \cref{fig:Goatin_fp_vs_direct_B} compares the space--time density of the purely spatial model against those of the three memory variants. Memory regularizes the dynamics: the time-averaging effect of $K$ smooths out sharp features and slows the formation of high-density regions, with the degree of smoothing depending on how broadly $K$ distributes weight over the past.

The difference is further demonstrated in \cref{fig:Goatin_snapshots_B}, which overlays density profiles at selected instants for all four models. The memory kernels visibly delay and diffuse the shock structure relative to the purely spatial case, with the Erlang kernel producing the broadest profile owing to its heavier tail.

\begin{figure}[h!]
    \centering
    \includegraphics[width=\textwidth]{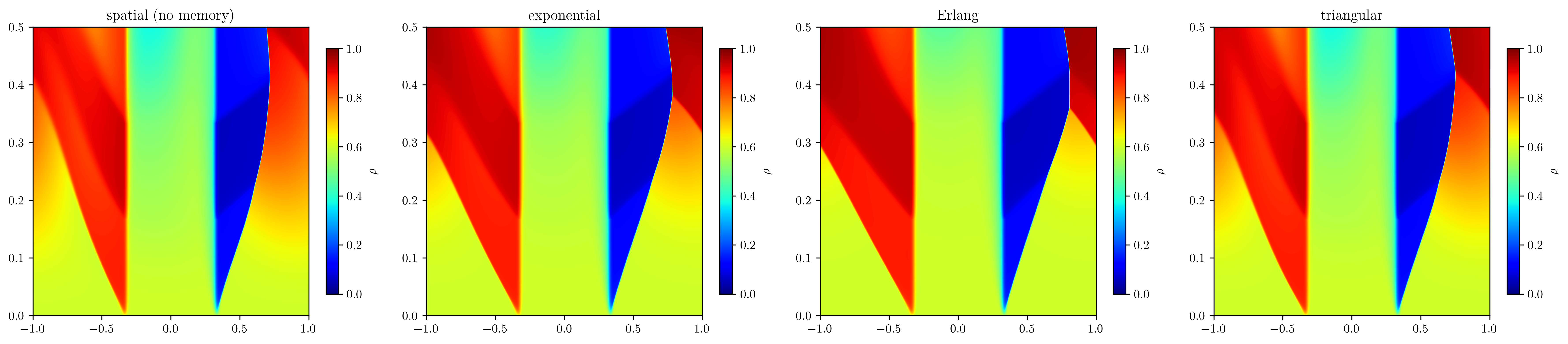}
    \caption{Space--time density for \eqref{eq:Goatin} with $m=3$, $\eta=1.0$, and $\delta=0$: from left to right, a spatial-only model and three memory kernels (exponential, Erlang, triangular). A comparison of the fixed-point solver and direct Lax--Friedrichs is shown for each, along with pointwise errors.}
    \label{fig:Goatin_fp_vs_direct_B}
\end{figure}

\begin{figure}[h!]
    \centering
    \includegraphics[width=\textwidth]{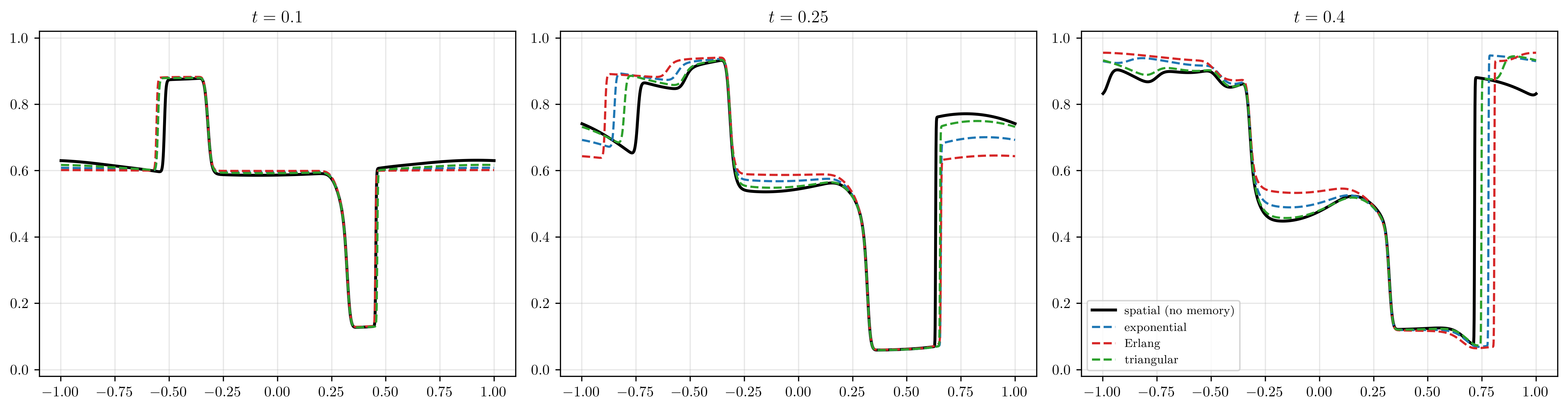}
    \caption{Density profiles at $t\in\{0.1,0.25,0.4\}$ for \eqref{eq:Goatin} with $m=3$, $\eta=1.0$, and $\delta=0$: spatial-only model (solid) and three memory kernels (dashed).}
    \label{fig:Goatin_snapshots_B}
\end{figure}

\section{Future work and open problems}
\label{s:future}

In this work, we have established existence and uniqueness of entropy solutions for a broad class of nonlocal conservation laws in $\R^n$, via a fixed-point approach. The same method extends naturally to equations where the nonlocal operator incorporates a time memory or a time delay, yielding analogous well-posedness results in both settings. Several directions remain open for future investigation.
\begin{enumerate}
    \item So far, we have only considered conservation laws, while in applications balance laws are often required (appearing, for example, in traffic flow modelling when lane-changing is allowed, see \cite{bayen2022modeling,chiarello2024singular,holden2019models}). A particular interesting case is that such right hand side might also depend on the nonlocal operator, in line with what we also point out in \cref{rem:balance_laws}. A comprehensive analysis of the minimal assumptions for the wellposedness of nonlocal balance laws, in the same spirit of this contribution, is an interesting future avenue for research.
    \item Another important thread is the investigation of the singular limit problem, i.e., whether one obtains the unique entropy solution of the corresponding ``local'' conservation law when the nonlocal kernel tends to a Dirac-distribution. 
    For nonlocal conservation laws this has been studied intensively over the last decade (see \cite{bressan2021entropy,coclite2022general,nitti2025singular,colombo2023nonlocal,keimer2019approximation} just to name a few related results). Notwithstanding, the key feature of the considered class of nonlocal conservation laws was that the local dependency of the flux be linear. For nonlinear locality, a singular limit result is not available so far.
    \item A natural question connecting for the models with memory and delay introduced in \cref{subsec:memory,subsec:delay} is whether solutions converge to those of the memoryless, delay-free equation as the temporal kernel concentrates to a Dirac distribution or the delay tends to zero. A positive answer would confirm the robustness of the model class with respect to these perturbations. In the delay case, first results in this direction are already available in \cite{ciaramaglia2024nonlocaltrafficflowmodels,ciaramaglia2025multi,Keimer2019}.
\end{enumerate}

\bibliographystyle{elsarticle-harv}
\bibliography{references}

\begin{thebibliography}{70}
\expandafter\ifx\csname natexlab\endcsname\relax\def\natexlab#1{#1}\fi
\providecommand{\url}[1]{\texttt{#1}}
\providecommand{\href}[2]{#2}
\providecommand{\path}[1]{#1}
\providecommand{\DOIprefix}{doi:}
\providecommand{\ArXivprefix}{arXiv:}
\providecommand{\URLprefix}{URL: }
\providecommand{\Pubmedprefix}{pmid:}
\providecommand{\doi}[1]{\href{http://dx.doi.org/#1}{\path{#1}}}
\providecommand{\Pubmed}[1]{\href{pmid:#1}{\path{#1}}}
\providecommand{\bibinfo}[2]{#2}
\ifx\xfnm\relax \def\xfnm[#1]{\unskip,\space#1}\fi
\bibitem[{Aggarwal et~al.(2015)Aggarwal, Colombo and Goatin}]{aggarwal2015nonlocal}
\bibinfo{author}{Aggarwal, A.}, \bibinfo{author}{Colombo, R.M.}, \bibinfo{author}{Goatin, P.}, \bibinfo{year}{2015}.
\newblock \bibinfo{title}{Nonlocal systems of conservation laws in several space dimensions}.
\newblock \bibinfo{journal}{SIAM Journal on Numerical Analysis} \bibinfo{volume}{53}, \bibinfo{pages}{963–983}.
\newblock \URLprefix \url{http://dx.doi.org/10.1137/140975255}, \DOIprefix\doi{10.1137/140975255}.
\bibitem[{Aggarwal et~al.(2025)Aggarwal, Holden and Vaidya}]{aggarwal2025error}
\bibinfo{author}{Aggarwal, A.}, \bibinfo{author}{Holden, H.}, \bibinfo{author}{Vaidya, G.}, \bibinfo{year}{2025}.
\newblock \bibinfo{title}{Error estimates for systems of nonlocal balance laws modelling dense multilane vehicular traffic}.
\newblock \bibinfo{journal}{Nonlinearity} \bibinfo{volume}{38}, \bibinfo{pages}{105007}.
\newblock \URLprefix \url{http://dx.doi.org/10.1088/1361-6544/ae0b24}, \DOIprefix\doi{10.1088/1361-6544/ae0b24}.
\bibitem[{Aggarwal and Vaidya(2024)}]{aggarwal2025convergence}
\bibinfo{author}{Aggarwal, A.}, \bibinfo{author}{Vaidya, G.}, \bibinfo{year}{2024}.
\newblock \bibinfo{title}{Convergence of the numerical approximations and well-posedness: {N}onlocal conservation laws with rough flux}.
\newblock \bibinfo{journal}{Mathematics of Computation} \URLprefix \url{http://dx.doi.org/10.1090/mcom/3976}, \DOIprefix\doi{10.1090/mcom/3976}.
\bibitem[{Amorim et~al.(2015)Amorim, Colombo and Teixeira}]{amorim2015numerical}
\bibinfo{author}{Amorim, P.}, \bibinfo{author}{Colombo, R.M.}, \bibinfo{author}{Teixeira, A.}, \bibinfo{year}{2015}.
\newblock \bibinfo{title}{On the numerical integration of scalar nonlocal conservation laws}.
\newblock \bibinfo{journal}{ESAIM: Mathematical Modelling and Numerical Analysis} \bibinfo{volume}{49}, \bibinfo{pages}{19–37}.
\newblock \URLprefix \url{http://dx.doi.org/10.1051/m2an/2014023}, \DOIprefix\doi{10.1051/m2an/2014023}.
\bibitem[{Aw and Rascle(2000)}]{aw2000resurrection}
\bibinfo{author}{Aw, A.}, \bibinfo{author}{Rascle, M.}, \bibinfo{year}{2000}.
\newblock \bibinfo{title}{Resurrection of ``second order'' models of traffic flow}.
\newblock \bibinfo{journal}{SIAM Journal on Applied Mathematics} \bibinfo{volume}{60}, \bibinfo{pages}{916–938}.
\newblock \URLprefix \url{http://dx.doi.org/10.1137/S0036139997332099}, \DOIprefix\doi{10.1137/s0036139997332099}.
\bibitem[{Bayen et~al.(2022)Bayen, Friedrich, Keimer, Pflug and Veeravalli}]{bayen2022modeling}
\bibinfo{author}{Bayen, A.}, \bibinfo{author}{Friedrich, J.}, \bibinfo{author}{Keimer, A.}, \bibinfo{author}{Pflug, L.}, \bibinfo{author}{Veeravalli, T.}, \bibinfo{year}{2022}.
\newblock \bibinfo{title}{Modeling multilane traffic with moving obstacles by nonlocal balance laws}.
\newblock \bibinfo{journal}{SIAM Journal on Applied Dynamical Systems} \bibinfo{volume}{21}, \bibinfo{pages}{1495–1538}.
\newblock \URLprefix \url{http://dx.doi.org/10.1137/20M1366654}, \DOIprefix\doi{10.1137/20m1366654}.
\bibitem[{Beckers and Friedrich(2026)}]{beckers2026monotone}
\bibinfo{author}{Beckers, A.}, \bibinfo{author}{Friedrich, J.}, \bibinfo{year}{2026}.
\newblock \bibinfo{title}{Monotone-based numerical schemes for two-dimensional systems of nonlocal conservation laws}.
\newblock \URLprefix \url{https://arxiv.org/abs/2601.20494}, \DOIprefix\doi{10.48550/ARXIV.2601.20494}.
\bibitem[{Bellomo et~al.(2002)Bellomo, Delitala and Coscia}]{bellomo2002mathematical}
\bibinfo{author}{Bellomo, N.}, \bibinfo{author}{Delitala, M.}, \bibinfo{author}{Coscia, V.}, \bibinfo{year}{2002}.
\newblock \bibinfo{title}{On the mathematical theory of vehicular traffic flow {I}: {F}luid dynamic and kinetic modelling}.
\newblock \bibinfo{journal}{Mathematical Models and Methods in Applied Sciences} \bibinfo{volume}{12}, \bibinfo{pages}{1801–1843}.
\newblock \URLprefix \url{http://dx.doi.org/10.1142/S0218202502002343}, \DOIprefix\doi{10.1142/s0218202502002343}.
\bibitem[{Bellomo and Dogbe(2011)}]{bellomo2011modeling}
\bibinfo{author}{Bellomo, N.}, \bibinfo{author}{Dogbe, C.}, \bibinfo{year}{2011}.
\newblock \bibinfo{title}{On the modeling of traffic and crowds: {A} survey of models, speculations, and perspectives}.
\newblock \bibinfo{journal}{SIAM Review} \bibinfo{volume}{53}, \bibinfo{pages}{409–463}.
\newblock \URLprefix \url{http://dx.doi.org/10.1137/090746677}, \DOIprefix\doi{10.1137/090746677}.
\bibitem[{Betancourt et~al.(2011)Betancourt, B\"{u}rger, Karlsen and Tory}]{betancourt2011nonlocal}
\bibinfo{author}{Betancourt, F.}, \bibinfo{author}{B\"{u}rger, R.}, \bibinfo{author}{Karlsen, K.H.}, \bibinfo{author}{Tory, E.M.}, \bibinfo{year}{2011}.
\newblock \bibinfo{title}{On nonlocal conservation laws modelling sedimentation}.
\newblock \bibinfo{journal}{Nonlinearity} \bibinfo{volume}{24}, \bibinfo{pages}{855–885}.
\newblock \URLprefix \url{http://dx.doi.org/10.1088/0951-7715/24/3/008}, \DOIprefix\doi{10.1088/0951-7715/24/3/008}.
\bibitem[{Blandin and Goatin(2016)}]{blandin2016well}
\bibinfo{author}{Blandin, S.}, \bibinfo{author}{Goatin, P.}, \bibinfo{year}{2016}.
\newblock \bibinfo{title}{Well-posedness of a conservation law with non-local flux arising in traffic flow modeling}.
\newblock \bibinfo{journal}{Numerische Mathematik} \bibinfo{volume}{132}, \bibinfo{pages}{217–241}.
\newblock \URLprefix \url{http://dx.doi.org/10.1007/s00211-015-0717-6}, \DOIprefix\doi{10.1007/s00211-015-0717-6}.
\bibitem[{Borsche et~al.(2015)Borsche, Colombo, Garavello and Meurer}]{borsche2015differential}
\bibinfo{author}{Borsche, R.}, \bibinfo{author}{Colombo, R.M.}, \bibinfo{author}{Garavello, M.}, \bibinfo{author}{Meurer, A.}, \bibinfo{year}{2015}.
\newblock \bibinfo{title}{Differential equations modeling crowd interactions}.
\newblock \bibinfo{journal}{Journal of Nonlinear Science} \bibinfo{volume}{25}, \bibinfo{pages}{827–859}.
\newblock \URLprefix \url{http://dx.doi.org/10.1007/s00332-015-9242-0}, \DOIprefix\doi{10.1007/s00332-015-9242-0}.
\bibitem[{Bressan(2000)}]{Bressan2000}
\bibinfo{author}{Bressan, A.}, \bibinfo{year}{2000}.
\newblock \bibinfo{title}{Hyperbolic Systems of Conservation Laws: The One-Dimensional Cauchy Problem}.
\newblock \bibinfo{publisher}{Oxford University PressOxford}.
\newblock \URLprefix \url{http://dx.doi.org/10.1093/oso/9780198507000.001.0001}, \DOIprefix\doi{10.1093/oso/9780198507000.001.0001}.
\bibitem[{Bressan and Shen(2021)}]{bressan2021entropy}
\bibinfo{author}{Bressan, A.}, \bibinfo{author}{Shen, W.}, \bibinfo{year}{2021}.
\newblock \bibinfo{title}{Entropy admissibility of the limit solution for a nonlocal model of traffic flow}.
\newblock \bibinfo{journal}{Communications in Mathematical Sciences} \bibinfo{volume}{19}, \bibinfo{pages}{1447–1450}.
\newblock \URLprefix \url{http://dx.doi.org/10.4310/CMS.2021.v19.n5.a12}, \DOIprefix\doi{10.4310/cms.2021.v19.n5.a12}.
\bibitem[{Brezis(2011)}]{brezis2011}
\bibinfo{author}{Brezis, H.}, \bibinfo{year}{2011}.
\newblock \bibinfo{title}{Functional Analysis, Sobolev Spaces and Partial Differential Equations}.
\newblock \bibinfo{publisher}{Springer New York}.
\newblock \URLprefix \url{http://dx.doi.org/10.1007/978-0-387-70914-7}, \DOIprefix\doi{10.1007/978-0-387-70914-7}.
\bibitem[{Bruno et~al.(2011)Bruno, Tosin, Tricerri and Venuti}]{bruno2011non}
\bibinfo{author}{Bruno, L.}, \bibinfo{author}{Tosin, A.}, \bibinfo{author}{Tricerri, P.}, \bibinfo{author}{Venuti, F.}, \bibinfo{year}{2011}.
\newblock \bibinfo{title}{Non-local first-order modelling of crowd dynamics: {A} multidimensional framework with applications}.
\newblock \bibinfo{journal}{Applied Mathematical Modelling} \bibinfo{volume}{35}, \bibinfo{pages}{426–445}.
\newblock \URLprefix \url{http://dx.doi.org/10.1016/j.apm.2010.07.007}, \DOIprefix\doi{10.1016/j.apm.2010.07.007}.
\bibitem[{B\"{u}rger and Karlsen(2003)}]{burger2003diffusively}
\bibinfo{author}{B\"{u}rger, R.}, \bibinfo{author}{Karlsen, K.H.}, \bibinfo{year}{2003}.
\newblock \bibinfo{title}{On a diffusively corrected kinematic-wave traffic flow model with changing road surface conditions}.
\newblock \bibinfo{journal}{Mathematical Models and Methods in Applied Sciences} \bibinfo{volume}{13}, \bibinfo{pages}{1767–1799}.
\newblock \URLprefix \url{http://dx.doi.org/10.1142/S0218202503003112}, \DOIprefix\doi{10.1142/s0218202503003112}.
\bibitem[{Chalons et~al.(2018)Chalons, Goatin and Villada}]{chalons2018high}
\bibinfo{author}{Chalons, C.}, \bibinfo{author}{Goatin, P.}, \bibinfo{author}{Villada, L.M.}, \bibinfo{year}{2018}.
\newblock \bibinfo{title}{High-order numerical schemes for one-dimensional nonlocal conservation laws}.
\newblock \bibinfo{journal}{SIAM Journal on Scientific Computing} \bibinfo{volume}{40}, \bibinfo{pages}{A288–A305}.
\newblock \URLprefix \url{http://dx.doi.org/10.1137/16M110825X}, \DOIprefix\doi{10.1137/16m110825x}.
\bibitem[{Chiarello et~al.(2023)Chiarello, Contreras and Villada}]{chiarello2023existence}
\bibinfo{author}{Chiarello, F.A.}, \bibinfo{author}{Contreras, H.D.}, \bibinfo{author}{Villada, L.M.}, \bibinfo{year}{2023}.
\newblock \bibinfo{title}{Existence of entropy weak solutions for {1D} non-local traffic models with space-discontinuous flux}.
\newblock \bibinfo{journal}{Journal of Engineering Mathematics} \bibinfo{volume}{141}.
\newblock \URLprefix \url{http://dx.doi.org/10.1007/s10665-023-10284-5}, \DOIprefix\doi{10.1007/s10665-023-10284-5}.
\bibitem[{Chiarello and Goatin(2018)}]{chiarello2018global}
\bibinfo{author}{Chiarello, F.A.}, \bibinfo{author}{Goatin, P.}, \bibinfo{year}{2018}.
\newblock \bibinfo{title}{Global entropy weak solutions for general non-local traffic flow models with anisotropic kernel}.
\newblock \bibinfo{journal}{ESAIM: Mathematical Modelling and Numerical Analysis} \bibinfo{volume}{52}, \bibinfo{pages}{163–180}.
\newblock \URLprefix \url{http://dx.doi.org/10.1051/m2an/2017066}, \DOIprefix\doi{10.1051/m2an/2017066}.
\bibitem[{Chiarello et~al.(2019)Chiarello, Goatin and Rossi}]{chiarello2019stability}
\bibinfo{author}{Chiarello, F.A.}, \bibinfo{author}{Goatin, P.}, \bibinfo{author}{Rossi, E.}, \bibinfo{year}{2019}.
\newblock \bibinfo{title}{Stability estimates for non-local scalar conservation laws}.
\newblock \bibinfo{journal}{Nonlinear Analysis: Real World Applications} \bibinfo{volume}{45}, \bibinfo{pages}{668–687}.
\newblock \URLprefix \url{http://dx.doi.org/10.1016/j.nonrwa.2018.07.027}, \DOIprefix\doi{10.1016/j.nonrwa.2018.07.027}.
\bibitem[{Chiarello and Keimer(2024)}]{chiarello2024singular}
\bibinfo{author}{Chiarello, F.A.}, \bibinfo{author}{Keimer, A.}, \bibinfo{year}{2024}.
\newblock \bibinfo{title}{On the singular limit problem in nonlocal balance laws: Applications to nonlocal lane-changing traffic flow models}.
\newblock \bibinfo{journal}{Journal of Mathematical Analysis and Applications} \bibinfo{volume}{537}, \bibinfo{pages}{128358}.
\newblock \URLprefix \url{http://dx.doi.org/10.1016/j.jmaa.2024.128358}, \DOIprefix\doi{10.1016/j.jmaa.2024.128358}.
\bibitem[{Ciaramaglia et~al.(2025a)Ciaramaglia, Goatin and Puppo}]{ciaramaglia2025multi}
\bibinfo{author}{Ciaramaglia, I.}, \bibinfo{author}{Goatin, P.}, \bibinfo{author}{Puppo, G.}, \bibinfo{year}{2025}a.
\newblock \bibinfo{title}{A multi-class non-local macroscopic model with time delay for mixed autonomous / human-driven traffic}.
\newblock \URLprefix \url{https://arxiv.org/abs/2501.09440}, \DOIprefix\doi{10.48550/ARXIV.2501.09440}.
\bibitem[{Ciaramaglia et~al.(2025b)Ciaramaglia, Goatin and Puppo}]{ciaramaglia2024nonlocaltrafficflowmodels}
\bibinfo{author}{Ciaramaglia, I.}, \bibinfo{author}{Goatin, P.}, \bibinfo{author}{Puppo, G.}, \bibinfo{year}{2025}b.
\newblock \bibinfo{title}{Non-local traffic flow models with time delay: Well-posedness and numerical approximation}.
\newblock \bibinfo{journal}{Discrete and Continuous Dynamical Systems - B} \bibinfo{volume}{30}, \bibinfo{pages}{874–907}.
\newblock \URLprefix \url{http://dx.doi.org/10.3934/dcdsb.2024113}, \DOIprefix\doi{10.3934/dcdsb.2024113}.
\bibitem[{Cockburn et~al.(2012)Cockburn, Karniadakis and Shu}]{cockburn2012discontinuous}
\bibinfo{author}{Cockburn, B.}, \bibinfo{author}{Karniadakis, G.E.}, \bibinfo{author}{Shu, C.W.}, \bibinfo{year}{2012}.
\newblock \bibinfo{title}{Discontinuous Galerkin Methods: Theory, Computation and Applications}. volume~\bibinfo{volume}{11}.
\newblock \bibinfo{publisher}{Springer Berlin Heidelberg}.
\newblock \URLprefix \url{http://dx.doi.org/10.1007/978-3-642-59721-3}, \DOIprefix\doi{10.1007/978-3-642-59721-3}.
\bibitem[{Cockburn and Shu(2001)}]{cockburn2001runge}
\bibinfo{author}{Cockburn, B.}, \bibinfo{author}{Shu, C.W.}, \bibinfo{year}{2001}.
\newblock \bibinfo{title}{Runge–{K}utta discontinuous {G}alerkin methods for convection-dominated problems}.
\newblock \bibinfo{journal}{Journal of Scientific Computing} \bibinfo{volume}{16}, \bibinfo{pages}{173–261}.
\newblock \URLprefix \url{http://dx.doi.org/10.1023/A:1012873910884}, \DOIprefix\doi{10.1023/a:1012873910884}.
\bibitem[{Coclite et~al.(2022a)Coclite, Coron, De~Nitti, Keimer and Pflug}]{coclite2022general}
\bibinfo{author}{Coclite, G.M.}, \bibinfo{author}{Coron, J.M.}, \bibinfo{author}{De~Nitti, N.}, \bibinfo{author}{Keimer, A.}, \bibinfo{author}{Pflug, L.}, \bibinfo{year}{2022}a.
\newblock \bibinfo{title}{A general result on the approximation of local conservation laws by nonlocal conservation laws: The singular limit problem for exponential kernels}.
\newblock \bibinfo{journal}{Annales de l’Institut Henri Poincaré C, Analyse non linéaire} \bibinfo{volume}{40}, \bibinfo{pages}{1205–1223}.
\newblock \URLprefix \url{http://dx.doi.org/10.4171/aihpc/58}, \DOIprefix\doi{10.4171/aihpc/58}.
\bibitem[{Coclite et~al.(2025)Coclite, De~Nitti and Huang}]{nitti2025singular}
\bibinfo{author}{Coclite, G.M.}, \bibinfo{author}{De~Nitti, N.}, \bibinfo{author}{Huang, K.}, \bibinfo{year}{2025}.
\newblock \bibinfo{title}{Singular limit for a class of nonlocal conservation laws via compensated compactness}.
\newblock \URLprefix \url{https://arxiv.org/abs/2511.15631}, \DOIprefix\doi{10.48550/ARXIV.2511.15631}.
\bibitem[{Coclite et~al.(2022b)Coclite, De~Nitti, Keimer and Pflug}]{coclite2022existence}
\bibinfo{author}{Coclite, G.M.}, \bibinfo{author}{De~Nitti, N.}, \bibinfo{author}{Keimer, A.}, \bibinfo{author}{Pflug, L.}, \bibinfo{year}{2022}b.
\newblock \bibinfo{title}{On existence and uniqueness of weak solutions to nonlocal conservation laws with {BV} kernels}.
\newblock \bibinfo{journal}{Zeitschrift f\"{u}r angewandte Mathematik und Physik} \bibinfo{volume}{73}.
\newblock \URLprefix \url{http://dx.doi.org/10.1007/s00033-022-01766-0}, \DOIprefix\doi{10.1007/s00033-022-01766-0}.
\bibitem[{Colombo et~al.(2023)Colombo, Crippa, Marconi and Spinolo}]{colombo2023nonlocal}
\bibinfo{author}{Colombo, M.}, \bibinfo{author}{Crippa, G.}, \bibinfo{author}{Marconi, E.}, \bibinfo{author}{Spinolo, L.V.}, \bibinfo{year}{2023}.
\newblock \bibinfo{title}{Nonlocal traffic models with general kernels: Singular limit, entropy admissibility, and convergence rate}.
\newblock \bibinfo{journal}{Archive for Rational Mechanics and Analysis} \bibinfo{volume}{247}.
\newblock \URLprefix \url{http://dx.doi.org/10.1007/s00205-023-01845-0}, \DOIprefix\doi{10.1007/s00205-023-01845-0}.
\bibitem[{Colombo et~al.(2024)Colombo, Crippa and Spinolo}]{colombo2026multi}
\bibinfo{author}{Colombo, M.}, \bibinfo{author}{Crippa, G.}, \bibinfo{author}{Spinolo, L.V.}, \bibinfo{year}{2024}.
\newblock \bibinfo{title}{On multidimensional nonlocal conservation laws with {BV} kernels}.
\newblock \URLprefix \url{https://arxiv.org/abs/2408.02423}, \DOIprefix\doi{10.48550/ARXIV.2408.02423}.
\bibitem[{Colombo and Garavello(2025)}]{colombo2025non}
\bibinfo{author}{Colombo, R.M.}, \bibinfo{author}{Garavello, M.}, \bibinfo{year}{2025}.
\newblock \bibinfo{title}{Non-local hyperbolic dynamics of clusters}.
\newblock \bibinfo{journal}{Mathematical Modelling of Natural Phenomena} \bibinfo{volume}{20}, \bibinfo{pages}{10}.
\newblock \URLprefix \url{http://dx.doi.org/10.1051/mmnp/2025011}, \DOIprefix\doi{10.1051/mmnp/2025011}.
\bibitem[{Colombo et~al.(2012)Colombo, Garavello and L{\'e}cureux-Mercier}]{colombo2012class}
\bibinfo{author}{Colombo, R.M.}, \bibinfo{author}{Garavello, M.}, \bibinfo{author}{L{\'e}cureux-Mercier, M.}, \bibinfo{year}{2012}.
\newblock \bibinfo{title}{A class of nonlocal models for pedestrian traffic}.
\newblock \bibinfo{journal}{Mathematical Models and Methods in Applied Sciences} \bibinfo{volume}{22}.
\newblock \URLprefix \url{http://dx.doi.org/10.1142/S0218202511500230}, \DOIprefix\doi{10.1142/s0218202511500230}.
\bibitem[{Colombo et~al.(2010)Colombo, Herty and Mercier}]{colombo2011control}
\bibinfo{author}{Colombo, R.M.}, \bibinfo{author}{Herty, M.}, \bibinfo{author}{Mercier, M.}, \bibinfo{year}{2010}.
\newblock \bibinfo{title}{Control of the continuity equation with a non local flow}.
\newblock \bibinfo{journal}{ESAIM: Control, Optimisation and Calculus of Variations} \bibinfo{volume}{17}, \bibinfo{pages}{353–379}.
\newblock \URLprefix \url{http://dx.doi.org/10.1051/cocv/2010007}, \DOIprefix\doi{10.1051/cocv/2010007}.
\bibitem[{Contreras et~al.(2025)Contreras, Goatin and Villada}]{contreras2025wellposedness}
\bibinfo{author}{Contreras, H.D.}, \bibinfo{author}{Goatin, P.}, \bibinfo{author}{Villada, L.M.}, \bibinfo{year}{2025}.
\newblock \bibinfo{title}{{Well-posedness of a nonlocal upstream-downstream traffic model}}.
\newblock \URLprefix \url{https://hal.science/hal-05302417}. \bibinfo{note}{working paper or preprint}.
\bibitem[{Coron et~al.(2010)Coron, Kawski and Wang}]{Coron2010analysis}
\bibinfo{author}{Coron, J.M.}, \bibinfo{author}{Kawski, M.}, \bibinfo{author}{Wang, Z.}, \bibinfo{year}{2010}.
\newblock \bibinfo{title}{Analysis of a conservation law modeling a highly re-entrant manufacturing system}.
\newblock \bibinfo{journal}{Discrete and Continuous Dynamical Systems - B} \bibinfo{volume}{13}, \bibinfo{pages}{1337–1359}.
\newblock \URLprefix \url{http://dx.doi.org/10.3934/dcdsb.2010.14.1337}, \DOIprefix\doi{10.3934/dcdsb.2010.14.1337}.
\bibitem[{de~Courcel(2025)}]{courcel2025well-posedness}
\bibinfo{author}{de~Courcel, A.C.}, \bibinfo{year}{2025}.
\newblock \bibinfo{title}{Well-posedness of multidimensional nonlocal conservation laws with nonlinear mobility and bounded force}.
\newblock \URLprefix \url{https://arxiv.org/abs/2512.13535}, \DOIprefix\doi{10.48550/ARXIV.2512.13535}.
\bibitem[{Crippa and Lécureux-Mercier(2012)}]{crippa2012existence}
\bibinfo{author}{Crippa, G.}, \bibinfo{author}{Lécureux-Mercier, M.}, \bibinfo{year}{2012}.
\newblock \bibinfo{title}{Existence and uniqueness of measure solutions for a system of continuity equations with non-local flow}.
\newblock \bibinfo{journal}{Nonlinear Differential Equations and Applications NoDEA} \bibinfo{volume}{20}, \bibinfo{pages}{523–537}.
\newblock \URLprefix \url{http://dx.doi.org/10.1007/s00030-012-0164-3}, \DOIprefix\doi{10.1007/s00030-012-0164-3}.
\bibitem[{{Di Francesco} et~al.(2019){Di Francesco}, Fagioli and Radici}]{DiFrancesco2019deterministic}
\bibinfo{author}{{Di Francesco}, M.}, \bibinfo{author}{Fagioli, S.}, \bibinfo{author}{Radici, E.}, \bibinfo{year}{2019}.
\newblock \bibinfo{title}{Deterministic particle approximation for nonlocal transport equations with nonlinear mobility}.
\newblock \bibinfo{journal}{Journal of Differential Equations} \bibinfo{volume}{266}, \bibinfo{pages}{2830--2868}.
\newblock \URLprefix \url{https://www.sciencedirect.com/science/article/pii/S0022039618305102}, \DOIprefix\doi{https://doi.org/10.1016/j.jde.2018.08.047}.
\bibitem[{Friedrich et~al.(2018)Friedrich, Kolb and Göttlich}]{friedrich2018godunov}
\bibinfo{author}{Friedrich, J.}, \bibinfo{author}{Kolb, O.}, \bibinfo{author}{Göttlich, S.}, \bibinfo{year}{2018}.
\newblock \bibinfo{title}{A {G}odunov type scheme for a class of {LWR} traffic flow models with non-local flux}.
\newblock \bibinfo{journal}{Networks and Heterogeneous Media} \bibinfo{volume}{13}, \bibinfo{pages}{531--547}.
\newblock \URLprefix \url{https://www.aimsciences.org/article/id/29185119-abc7-4540-a74f-a5ed9e4c34ad}, \DOIprefix\doi{10.3934/nhm.2018024}.
\bibitem[{Goatin et~al.(2025)Goatin, Göttlich and Ziegler}]{goatin2025nonlocalmodelheterogeneousmaterial}
\bibinfo{author}{Goatin, P.}, \bibinfo{author}{Göttlich, S.}, \bibinfo{author}{Ziegler, F.}, \bibinfo{year}{2025}.
\newblock \bibinfo{title}{A non-local model for heterogeneous material flow on conveyor belts}.
\newblock \URLprefix \url{https://arxiv.org/abs/2510.17500}, \href{http://arxiv.org/abs/2510.17500}{{\tt arXiv:2510.17500}}.
\bibitem[{Goatin and Scialanga(2016)}]{goatin2016well}
\bibinfo{author}{Goatin, P.}, \bibinfo{author}{Scialanga, S.}, \bibinfo{year}{2016}.
\newblock \bibinfo{title}{Well-posedness and finite volume approximations of the {LWR} traffic flow model with non-local velocity}.
\newblock \bibinfo{journal}{Networks and Heterogeneous Media} \bibinfo{volume}{11}, \bibinfo{pages}{107–121}.
\newblock \URLprefix \url{http://dx.doi.org/10.3934/nhm.2016.11.107}, \DOIprefix\doi{10.3934/nhm.2016.11.107}.
\bibitem[{Hairer et~al.(1988)Hairer, Lubich and Schlichte}]{HAIRER198887}
\bibinfo{author}{Hairer, E.}, \bibinfo{author}{Lubich, C.}, \bibinfo{author}{Schlichte, M.}, \bibinfo{year}{1988}.
\newblock \bibinfo{title}{Fast numerical solution of weakly singular volterra integral equations}.
\newblock \bibinfo{journal}{Journal of Computational and Applied Mathematics} \bibinfo{volume}{23}, \bibinfo{pages}{87--98}.
\newblock \URLprefix \url{https://www.sciencedirect.com/science/article/pii/0377042788903329}, \DOIprefix\doi{https://doi.org/10.1016/0377-0427(88)90332-9}.
\bibitem[{Harten et~al.(1997)Harten, Engquist, Osher and Chakravarthy}]{harten1997uniformly}
\bibinfo{author}{Harten, A.}, \bibinfo{author}{Engquist, B.}, \bibinfo{author}{Osher, S.}, \bibinfo{author}{Chakravarthy, S.R.}, \bibinfo{year}{1997}.
\newblock \bibinfo{title}{Uniformly high order accurate essentially non-oscillatory schemes, {III}}.
\newblock \bibinfo{journal}{Journal of Computational Physics} \bibinfo{volume}{131}, \bibinfo{pages}{3–47}.
\newblock \URLprefix \url{http://dx.doi.org/10.1006/jcph.1996.5632}, \DOIprefix\doi{10.1006/jcph.1996.5632}.
\bibitem[{Holden and Risebro(2015)}]{holden2015front}
\bibinfo{author}{Holden, H.}, \bibinfo{author}{Risebro, N.H.}, \bibinfo{year}{2015}.
\newblock \bibinfo{title}{Front Tracking for Hyperbolic Conservation Laws}.
\newblock \bibinfo{publisher}{Springer Berlin Heidelberg}.
\newblock \URLprefix \url{http://dx.doi.org/10.1007/978-3-662-47507-2}, \DOIprefix\doi{10.1007/978-3-662-47507-2}.
\bibitem[{Holden and Risebro(2019)}]{holden2019models}
\bibinfo{author}{Holden, H.}, \bibinfo{author}{Risebro, N.H.}, \bibinfo{year}{2019}.
\newblock \bibinfo{title}{Models for dense multilane vehicular traffic}.
\newblock \bibinfo{journal}{SIAM Journal on Mathematical Analysis} \bibinfo{volume}{51}, \bibinfo{pages}{3694–3713}.
\newblock \URLprefix \url{http://dx.doi.org/10.1137/19M124318X}, \DOIprefix\doi{10.1137/19m124318x}.
\bibitem[{Hu et~al.(2024)Hu, Lee and Zheng}]{hu2021shock}
\bibinfo{author}{Hu, Y.}, \bibinfo{author}{Lee, Y.}, \bibinfo{author}{Zheng, S.}, \bibinfo{year}{2024}.
\newblock \bibinfo{title}{Shock formation in traffic flow models with nonlocal look ahead and behind flux}. \bibinfo{publisher}{Springer Nature Switzerland}.
\newblock p. \bibinfo{pages}{301–317}.
\newblock \URLprefix \url{http://dx.doi.org/10.1007/978-3-031-69710-4_13}, \DOIprefix\doi{10.1007/978-3-031-69710-4_13}.
\bibitem[{Keimer and Pflug(2017)}]{keimer2017existence}
\bibinfo{author}{Keimer, A.}, \bibinfo{author}{Pflug, L.}, \bibinfo{year}{2017}.
\newblock \bibinfo{title}{Existence, uniqueness and regularity results on nonlocal balance laws}.
\newblock \bibinfo{journal}{Journal of Differential Equations} \bibinfo{volume}{263}, \bibinfo{pages}{4023–4069}.
\newblock \URLprefix \url{http://dx.doi.org/10.1016/j.jde.2017.05.015}, \DOIprefix\doi{10.1016/j.jde.2017.05.015}.
\bibitem[{Keimer and Pflug(2019a)}]{Keimer2019}
\bibinfo{author}{Keimer, A.}, \bibinfo{author}{Pflug, L.}, \bibinfo{year}{2019}a.
\newblock \bibinfo{title}{Nonlocal conservation laws with time delay}.
\newblock \bibinfo{journal}{Nonlinear Differential Equations and Applications NoDEA} \bibinfo{volume}{26}.
\newblock \URLprefix \url{http://dx.doi.org/10.1007/s00030-019-0597-z}, \DOIprefix\doi{10.1007/s00030-019-0597-z}.
\bibitem[{Keimer and Pflug(2019b)}]{keimer2019approximation}
\bibinfo{author}{Keimer, A.}, \bibinfo{author}{Pflug, L.}, \bibinfo{year}{2019}b.
\newblock \bibinfo{title}{On approximation of local conservation laws by nonlocal conservation laws}.
\newblock \bibinfo{journal}{Journal of Mathematical Analysis and Applications} \bibinfo{volume}{475}, \bibinfo{pages}{1927–1955}.
\newblock \URLprefix \url{http://dx.doi.org/10.1016/j.jmaa.2019.03.063}, \DOIprefix\doi{10.1016/j.jmaa.2019.03.063}.
\bibitem[{Keimer et~al.(2018)Keimer, Pflug and Spinola}]{keimer2018existence}
\bibinfo{author}{Keimer, A.}, \bibinfo{author}{Pflug, L.}, \bibinfo{author}{Spinola, M.}, \bibinfo{year}{2018}.
\newblock \bibinfo{title}{Existence, uniqueness and regularity of multi-dimensional nonlocal balance laws with damping}.
\newblock \bibinfo{journal}{Journal of Mathematical Analysis and Applications} \bibinfo{volume}{466}, \bibinfo{pages}{18–55}.
\newblock \URLprefix \url{http://dx.doi.org/10.1016/j.jmaa.2018.05.013}, \DOIprefix\doi{10.1016/j.jmaa.2018.05.013}.
\bibitem[{Kloeden and Lorenz(2016)}]{kloeden2016nonlocal}
\bibinfo{author}{Kloeden, P.E.}, \bibinfo{author}{Lorenz, T.}, \bibinfo{year}{2016}.
\newblock \bibinfo{title}{Nonlocal multi-scale traffic flow models: analysis beyond vector spaces}.
\newblock \bibinfo{journal}{Bulletin of Mathematical Sciences} \bibinfo{volume}{6}, \bibinfo{pages}{453–514}.
\newblock \URLprefix \url{http://dx.doi.org/10.1007/s13373-016-0090-5}, \DOIprefix\doi{10.1007/s13373-016-0090-5}.
\bibitem[{Kružkov(1970)}]{kruvzkov1970first}
\bibinfo{author}{Kružkov, S.N.}, \bibinfo{year}{1970}.
\newblock \bibinfo{title}{First order quasilinear equations in several independent variables}.
\newblock \bibinfo{journal}{Mathematics of the USSR-Sbornik} \bibinfo{volume}{10}, \bibinfo{pages}{217–243}.
\newblock \URLprefix \url{http://dx.doi.org/10.1070/SM1970v010n02ABEH002156}, \DOIprefix\doi{10.1070/sm1970v010n02abeh002156}.
\bibitem[{L{\'e}cureux-Mercier(2011)}]{Mercier2011improved}
\bibinfo{author}{L{\'e}cureux-Mercier, M.}, \bibinfo{year}{2011}.
\newblock \bibinfo{title}{Improved stability estimates for general scalar conservation laws}.
\newblock \bibinfo{journal}{Journal of Hyperbolic Differential Equations} \bibinfo{volume}{08}, \bibinfo{pages}{727–757}.
\newblock \URLprefix \url{http://dx.doi.org/10.1142/S021989161100255X}, \DOIprefix\doi{10.1142/s021989161100255x}.
\bibitem[{Leoni(2024)}]{leoni2024first}
\bibinfo{author}{Leoni, G.}, \bibinfo{year}{2024}.
\newblock \bibinfo{title}{A first course in Sobolev spaces}. volume \bibinfo{volume}{181}.
\newblock \bibinfo{publisher}{American Mathematical Society}.
\bibitem[{LeVeque(2002)}]{leveque2002finite}
\bibinfo{author}{LeVeque, R.J.}, \bibinfo{year}{2002}.
\newblock \bibinfo{title}{Finite Volume Methods for Hyperbolic Problems}.
\newblock \bibinfo{publisher}{Cambridge University Press}.
\newblock \URLprefix \url{http://dx.doi.org/10.1017/CBO9780511791253}, \DOIprefix\doi{10.1017/cbo9780511791253}.
\bibitem[{Lighthill and Whitham(1955a)}]{lighthill1955kinematic}
\bibinfo{author}{Lighthill, M.J.}, \bibinfo{author}{Whitham, G.B.}, \bibinfo{year}{1955}a.
\newblock \bibinfo{title}{On kinematic waves {I}. {F}lood movement in long rivers}.
\newblock \bibinfo{journal}{Proceedings of the Royal Society of London. Series A. Mathematical and Physical Sciences} \bibinfo{volume}{229}, \bibinfo{pages}{281–316}.
\newblock \URLprefix \url{http://dx.doi.org/10.1098/rspa.1955.0088}, \DOIprefix\doi{10.1098/rspa.1955.0088}.
\bibitem[{Lighthill and Whitham(1955b)}]{lighthill1955kinematic2}
\bibinfo{author}{Lighthill, M.J.}, \bibinfo{author}{Whitham, G.B.}, \bibinfo{year}{1955}b.
\newblock \bibinfo{title}{On kinematic waves {II}. {A} theory of traffic flow on long crowded roads}.
\newblock \bibinfo{journal}{Proceedings of the Royal Society of London. Series A. Mathematical and Physical Sciences} \bibinfo{volume}{229}, \bibinfo{pages}{317–345}.
\newblock \URLprefix \url{http://dx.doi.org/10.1098/rspa.1955.0089}, \DOIprefix\doi{10.1098/rspa.1955.0089}.
\bibitem[{Nelson(2002)}]{nelson2002traveling}
\bibinfo{author}{Nelson, P.}, \bibinfo{year}{2002}.
\newblock \bibinfo{title}{Traveling-wave solutions of the diffusively corrected kinematic-wave model}.
\newblock \bibinfo{journal}{Mathematical and Computer Modelling} \bibinfo{volume}{35}, \bibinfo{pages}{561–579}.
\newblock \URLprefix \url{http://dx.doi.org/10.1016/S0895-7177(02)80021-8}, \DOIprefix\doi{10.1016/s0895-7177(02)80021-8}.
\bibitem[{Piccoli and Tosin(2009)}]{piccoli2009vehicular}
\bibinfo{author}{Piccoli, B.}, \bibinfo{author}{Tosin, A.}, \bibinfo{year}{2009}.
\newblock \bibinfo{title}{Vehicular traffic: {A} review of continuum mathematical models}. \bibinfo{publisher}{Springer New York}.
\newblock p. \bibinfo{pages}{9727–9749}.
\newblock \URLprefix \url{http://dx.doi.org/10.1007/978-0-387-30440-3_576}, \DOIprefix\doi{10.1007/978-0-387-30440-3_576}.
\bibitem[{Radici and Stra(2023)}]{radici2023entropy}
\bibinfo{author}{Radici, E.}, \bibinfo{author}{Stra, F.}, \bibinfo{year}{2023}.
\newblock \bibinfo{title}{Entropy solutions of mildly singular nonlocal scalar conservation laws with congestion via deterministic particle methods}.
\newblock \bibinfo{journal}{SIAM Journal on Mathematical Analysis} \bibinfo{volume}{55}, \bibinfo{pages}{2001--2041}.
\newblock \URLprefix \url{https://doi.org/10.1137/21M1462994}, \DOIprefix\doi{10.1137/21M1462994}, \href{http://arxiv.org/abs/https://doi.org/10.1137/21M1462994}{{\tt arXiv:https://doi.org/10.1137/21M1462994}}.
\bibitem[{Richards(1956)}]{richards1956shock}
\bibinfo{author}{Richards, P.I.}, \bibinfo{year}{1956}.
\newblock \bibinfo{title}{Shock waves on the highway}.
\newblock \bibinfo{journal}{Operations research} \bibinfo{volume}{4}, \bibinfo{pages}{42--51}.
\bibitem[{Rosini(2013)}]{rosini2013macroscopic}
\bibinfo{author}{Rosini, M.D.}, \bibinfo{year}{2013}.
\newblock \bibinfo{title}{Macroscopic models for vehicular flows and crowd dynamics: {T}heory and applications: {C}lassical and non–classical advanced mathematics for real life applications}.
\newblock \bibinfo{publisher}{Springer International Publishing}.
\newblock \URLprefix \url{http://dx.doi.org/10.1007/978-3-319-00155-5}, \DOIprefix\doi{10.1007/978-3-319-00155-5}.
\bibitem[{Rossi et~al.(2020)Rossi, Weißen, Goatin and G\"{o}ttlich}]{Rossi2020}
\bibinfo{author}{Rossi, E.}, \bibinfo{author}{Weißen, J.}, \bibinfo{author}{Goatin, P.}, \bibinfo{author}{G\"{o}ttlich, S.}, \bibinfo{year}{2020}.
\newblock \bibinfo{title}{Well-posedness of a non-local model for material flow on conveyor belts}.
\newblock \bibinfo{journal}{ESAIM: Mathematical Modelling and Numerical Analysis} \bibinfo{volume}{54}, \bibinfo{pages}{679–704}.
\newblock \URLprefix \url{http://dx.doi.org/10.1051/m2an/2019062}, \DOIprefix\doi{10.1051/m2an/2019062}.
\bibitem[{Shu and Osher(1988)}]{shu1988efficient}
\bibinfo{author}{Shu, C.W.}, \bibinfo{author}{Osher, S.}, \bibinfo{year}{1988}.
\newblock \bibinfo{title}{Efficient implementation of essentially non-oscillatory shock-capturing schemes}.
\newblock \bibinfo{journal}{Journal of Computational Physics} \bibinfo{volume}{77}, \bibinfo{pages}{439–471}.
\newblock \URLprefix \url{http://dx.doi.org/10.1016/0021-9991(88)90177-5}, \DOIprefix\doi{10.1016/0021-9991(88)90177-5}.
\bibitem[{Shu and Osher(1989)}]{shu1989efficient}
\bibinfo{author}{Shu, C.W.}, \bibinfo{author}{Osher, S.}, \bibinfo{year}{1989}.
\newblock \bibinfo{title}{Efficient implementation of essentially non-oscillatory shock-capturing schemes, {II}}.
\newblock \bibinfo{journal}{Journal of Computational Physics} \bibinfo{volume}{83}, \bibinfo{pages}{32–78}.
\newblock \URLprefix \url{http://dx.doi.org/10.1016/0021-9991(89)90222-2}, \DOIprefix\doi{10.1016/0021-9991(89)90222-2}.
\bibitem[{Simo and Hughes(1998)}]{SimoHughes1998}
\bibinfo{author}{Simo, J.C.}, \bibinfo{author}{Hughes, T.J.R.}, \bibinfo{year}{1998}.
\newblock \bibinfo{title}{Computational Inelasticity}.
\newblock Interdisciplinary Applied Mathematics, \bibinfo{publisher}{Springer}, \bibinfo{address}{New York}.
\newblock \DOIprefix\doi{10.1007/b98904}.
\bibitem[{Simon(1986)}]{simon1986compact}
\bibinfo{author}{Simon, J.}, \bibinfo{year}{1986}.
\newblock \bibinfo{title}{Compact sets in the space ${L}^{p}(0, {T}; {B})$}.
\newblock \bibinfo{journal}{Annali di Matematica Pura ed Applicata} \bibinfo{volume}{146}, \bibinfo{pages}{65–96}.
\newblock \URLprefix \url{http://dx.doi.org/10.1007/BF01762360}, \DOIprefix\doi{10.1007/bf01762360}.
\bibitem[{Zhang(2000)}]{zhang2000structural}
\bibinfo{author}{Zhang, H.}, \bibinfo{year}{2000}.
\newblock \bibinfo{title}{Structural properties of solutions arising from a nonequilibrium traffic flow theory}.
\newblock \bibinfo{journal}{Transportation Research Part B: Methodological} \bibinfo{volume}{34}, \bibinfo{pages}{583–603}.
\newblock \URLprefix \url{http://dx.doi.org/10.1016/S0191-2615(99)00041-7}, \DOIprefix\doi{10.1016/s0191-2615(99)00041-7}.
\bibitem[{Zhang(2002)}]{zhang2002non}
\bibinfo{author}{Zhang, H.}, \bibinfo{year}{2002}.
\newblock \bibinfo{title}{A non-equilibrium traffic model devoid of gas-like behavior}.
\newblock \bibinfo{journal}{Transportation Research Part B: Methodological} \bibinfo{volume}{36}, \bibinfo{pages}{275–290}.
\newblock \URLprefix \url{http://dx.doi.org/10.1016/S0191-2615(00)00050-3}, \DOIprefix\doi{10.1016/s0191-2615(00)00050-3}.

\end{thebibliography}

\end{document}